\documentclass[11pt,reqno]{amsart}
\usepackage[a4paper,margin=30mm,headheight=14pt]{geometry}
\usepackage{amsmath,amssymb,amsthm,mathtools,mathrsfs,booktabs,array}
\usepackage{xcolor,enumitem,etoolbox,microtype,needspace}
\usepackage[T1]{fontenc}
\usepackage{lmodern}
\usepackage{hyperref}
\usepackage{aliascnt}
\usepackage[nameinlink,capitalize,noabbrev]{cleveref}
\hypersetup{
  colorlinks=true,
  linkcolor=blue!45!black,
  citecolor=blue!45!black,
  urlcolor=blue!45!black,
  pdftitle={Endpoint gradient bounds for centered convex maximal operators and probability kernels},
  pdfauthor={Zhouyu Long},
  pdfsubject={Endpoint cancellation, probability kernels and BV singular structure},
  pdfkeywords={maximal averages, Sobolev endpoint, bounded variation, probability profiles, Cantor derivative}
}
\numberwithin{equation}{section}
\newtheorem{theorem}{Theorem}[section]
\theoremstyle{plain}
\newaliascnt{lemma}{theorem}
\newtheorem{lemma}[lemma]{Lemma}
\aliascntresetthe{lemma}
\theoremstyle{plain}
\newaliascnt{proposition}{theorem}
\newtheorem{proposition}[proposition]{Proposition}
\aliascntresetthe{proposition}
\theoremstyle{plain}
\newaliascnt{corollary}{theorem}
\newtheorem{corollary}[corollary]{Corollary}
\aliascntresetthe{corollary}
\theoremstyle{definition}
\newaliascnt{definition}{theorem}
\newtheorem{definition}[definition]{Definition}
\aliascntresetthe{definition}
\theoremstyle{remark}
\newaliascnt{remark}{theorem}
\newtheorem{remark}[remark]{Remark}
\aliascntresetthe{remark}
\theoremstyle{remark}
\newaliascnt{example}{theorem}
\newtheorem{example}[example]{Example}
\aliascntresetthe{example}
\theoremstyle{plain}
\crefname{theorem}{Theorem}{Theorems}
\Crefname{theorem}{Theorem}{Theorems}
\crefname{lemma}{Lemma}{Lemmas}
\Crefname{lemma}{Lemma}{Lemmas}
\crefname{proposition}{Proposition}{Propositions}
\Crefname{proposition}{Proposition}{Propositions}
\crefname{corollary}{Corollary}{Corollaries}
\Crefname{corollary}{Corollary}{Corollaries}
\crefname{definition}{Definition}{Definitions}
\Crefname{definition}{Definition}{Definitions}
\crefname{remark}{Remark}{Remarks}
\Crefname{remark}{Remark}{Remarks}
\crefname{example}{Example}{Examples}
\Crefname{example}{Example}{Examples}

\newcommand{\R}{\mathbb R}
\newcommand{\dd}{\,\mathrm d}
\newcommand{\Lip}{\operatorname{Lip}}
\newcommand{\supp}{\operatorname{supp}}
\newcommand{\dist}{\operatorname{dist}}
\newcommand{\1}{\mathbf 1}

\title[Endpoint gradient bounds for convex maxima]{Endpoint gradient bounds for centered convex maximal operators\\and probability kernels}
\author[Zhouyu Long]{\bfseries Zhouyu Long$^{*}$}
\thanks{$^{*}$Department of Mathematical Sciences, Tsinghua University,
Beijing 100084, China.\newline
E-mail: \href{mailto:longzy25@mails.tsinghua.edu.cn}{\nolinkurl{longzy25@mails.tsinghua.edu.cn}}}

\makeatletter
\patchcmd{\@setauthors}{\centering\footnotesize}{\centering\large}{}%
  {\ClassError{amsart}{Unable to set author font}{Check the amsart author layout.}}
\patchcmd{\@setauthors}{\MakeUppercase{\authors}}{\authors}{}%
  {\ClassError{amsart}{Unable to preserve author case}{Check the amsart author layout.}}
\def\author@andify{\nxandlist{\unskip,\quad}{\unskip,\quad}{\unskip,\quad}}
\let\addresses\@empty
\makeatother
\date{}

\begin{document}
\begin{abstract}
We establish the endpoint gradient estimate for centered homothetic
maximal averages in every dimension $d\ge2$ over each fixed convex
body with $C^{1,\theta}$ boundary, $0<\theta\le1$. In particular, the
centered ball maximal operator satisfies the $W^{1,1}$ gradient bound
in every dimension. Neither strict convexity nor a lower bound on the
support gap is required. The geometric argument combines concentration
in gradient values with nonhomogeneous calibrated rows and a summable
radial correction. It extends the scalar refinement, direction-history
construction and global assembly of the planar centered-disk proof
of OpenAI. A separate transfer theorem gives sharp singular-derivative
identities for bounded compact probability kernels once local BV output
is known, and gradient continuity under an additional spatial-tail
condition. Quantitative contact and translation estimates yield further
profile and kernel examples. The estimates retain the fixed body's
regularity data; the endpoint for arbitrary $C^1$ and nonsmooth convex
bodies remains open here.
\end{abstract}

\maketitle
\section{Introduction}\label{sec:intro}

We prove the endpoint gradient bound for centered homothetic averages
in every dimension $d\ge2$ over each fixed $C^{1,\theta}$ convex body,
$0<\theta\le1$, allowing asymmetry, flat pieces and infinite contact
order. Constants retain the fixed body's regularity.

Kinnunen~\cite{Kinnunen} established Sobolev regularity for $p>1$,
and Luiro~\cite{Luiro} proved continuity in that range.
Haj{\l}asz--Onninen~\cite{HO} formulated the endpoint problem;
Kurka~\cite{Kurka} proved the centered one-dimensional variation
bound, and Weigt~\cite{WeigtCubes} treats uncentered cubes.
Our scalar refinement, direction-history construction and signed
global assembly are adapted from OpenAI's planar centered-disk proof
\cite[Theorem~1.1, Corollary~6.3 and Sections~3--6]{OpenAI}.
Their dimension-dependent versions and the new geometric inputs are
proved below.

Fix $d\ge2$, $q_d=d/(d-1)$ and $\omega_d=|B(0,1)|$.
Write $D^s,D^j,D^c$ for singular, jump and Cantor derivatives,
$g^\wedge,g^\vee$ for approximate lower and upper limits, $g^*$
for the precise representative and $J_g$ for the jump set. Jump
polars point from the lower to the upper trace. Set
$\mathcal X_d=\{h\in L^{q_d}\cap BV_{\rm loc}:|Dh|(\R^d)<\infty\}$.
For a bounded Borel set $K\subset\R^d$ of positive volume, define
\begin{equation}\label{eq:averages}
 A_r^Kg(x)=\frac1{|K|r^d}\int_{x+rK}g(y)\dd y,\qquad
 S_{a,b}^Kg=\sup_{a\le r\le b}A_r^Kg,\qquad
 M_Kf=\sup_{r>0}A_r^K|f|.
\end{equation}
Uniform signed-band estimates permit constant subtraction in spatial
localization.

\begin{theorem}\label{r3:bodyendpoint}\label{thm:R5convexheadline}
Let $d\ge2$ and let $K$ be a bounded open convex body containing zero.
Assume either of the following:
\begin{enumerate}
\item $\partial K\in C^{1,\theta}$ for some $0<\theta\le1$;
\item $p=p_K\in C^1(\R^d\setminus\{0\})$ and $\Phi=p^2/2$
satisfy the normal and boundary-contact bounds
\eqref{eq:lognormalmodulus}--\eqref{eq:lognormalgap}, respectively
on nearby pairs in fixed positive annuli and nearby distinct boundary
points, with
$\beta>\kappa\ge1$.
\end{enumerate}
Then every real $g\in BV\cap L^1$ and $0<a\le b<\infty$ satisfy
\begin{equation}\label{r5:headlineband}
 S_{a,b}^Kg\in\operatorname{Lip}\cap W^{1,1},\qquad
 \|\nabla S_{a,b}^Kg\|_1\le C_K|Dg|(\mathbb R^d).
\end{equation}
For $f\in BV\cap L^1$,
\begin{gather}
 M_Kf\in L^{q_d}\cap BV_{\rm loc},\qquad
 |D(M_Kf)|(\mathbb R^d)\le C_K|D(|f|)|(\mathbb R^d),\label{r5:headlineBV}\\
 |D^s(M_Kf)|\le |D^s(|f|)|\le |D^sf|.\label{r5:headlinesingular}
\end{gather}
Constants retain the fixed boundary regularity and collar data in
(1), or the normal and contact data in (2). On $\mathcal X_d$, bands
belong to $L^{q_d}\cap\operatorname{Lip}$ with the same gradient
bound, and the full variation and sharp singular bounds hold.
Global $L^1$ membership of $M_Kf$ and uniform bounds over arbitrary
convex bodies are not asserted.
\end{theorem}
\begin{proof}
Corollary~\ref{thm:R4holderendpoint} and
Theorem~\ref{thm:lognormalendpoint} verify the local inputs in cases
(1) and (2). Theorem~\ref{r3:master} and
Corollary~\ref{r5:homogeneousBV} give the endpoint on both input
classes. Once local BV output is known,
Theorem~\ref{r10:conditionalkernelfine} and the homogeneous passage in
Theorem~\ref{r10:kerneltransfer} give sharp singular domination.
\end{proof}

\begin{corollary}\label{r5:W11}
\label{r6:fineheadline}\label{r6:continuityheadline}\label{r7:stabilityheadline}
For $K$ as in Theorem~\ref{r3:bodyendpoint} and $f\in BV\cap L^1$,
set $g=|f|$, $u=M_Kf$ and
$C_K(g)=\{u^*=g^*,\ |u^*|+|g^*|<\infty\}$. Then
\begin{equation}\label{r7:headlinesingularidentity}
 D^su=\nu_g(g^\vee-u)_+\mathcal H^{d-1}\llcorner J_g
                      +(D^cg)\llcorner C_K(g).
\end{equation}
Here $u$ in the jump density is the natural supremum. In particular,
\begin{equation}\label{r6:headlinefine}
 |D^ju|\le(1-\theta_K(\nu_g))|D^jg|,\qquad
 D^cu=(D^cg)\llcorner C_K(g),\quad
 \theta_K(\nu)=\frac{|K\cap\{z\cdot\nu>0\}|}{|K|}.
\end{equation}
The directional coefficient is sharp; its uniform value is $1/2$
for symmetric $K$. Sobolev inputs satisfy
\begin{equation}\label{r3:bodybound}
 M_Kf\in L^{q_d}\cap W^{1,1}_{\rm loc},\qquad
 \|\nabla M_Kf\|_1\le C_{d,K}\|\nabla f\|_1,
\end{equation}
and $SBV\cap L^1$ inputs give $SBV_{\rm loc}$ outputs. Uniformly for
common signed-band constants, outer radii and positive lower volume, for $f\in W^{1,1}$,
\begin{equation}\label{r7:headlinetail}
 \int_{\{|\nabla M_Kf|>\lambda\}}|\nabla M_Kf|
 \le C\int_{\{|\nabla f|>c\lambda\}}|\nabla f|.
\end{equation}
For convex $K_n,K$ with a common inner ball and these common bounds,
$\mathbf1_{K_n}\to\mathbf1_K$ in $L^1$ and $f_n\to f$ in $W^{1,1}$ imply
\begin{equation}\label{r6:headlinecontinuity}
 \|M_{K_n}f_n-M_Kf\|_{q_d}
 +\|\nabla M_{K_n}f_n-\nabla M_Kf\|_1\longrightarrow0.
\end{equation}
Apply Theorem~\ref{r10:kerneltransfer}; for arbitrary bounded compact
probability kernels, the fine identities require only local BV output
(Theorem~\ref{r10:conditionalkernelfine}).
\end{corollary}

\begin{corollary}\label{cor:all-dimensional-ball}
Let $d\ge2$ and $B=B(0,1)\subset\R^d$. For every $f\in W^{1,1}(\R^d)$,
\[
 \|\nabla M_Bf\|_1\le C_d\|\nabla f\|_1,\qquad
 M_Bf\in L^{q_d}\cap W^{1,1}_{\rm loc}.
\]
For $f\in BV\cap L^1$ and $f_n\to f$ in $W^{1,1}$, respectively,
\eqref{r5:headlineBV}--\eqref{r5:headlinesingular} and
\eqref{r6:headlinecontinuity} hold with $K_n=K=B$ and dimensional
constants.
\end{corollary}
\begin{proof}
Apply Theorem~\ref{r3:bodyendpoint} and Corollary~\ref{r5:W11} to $B$,
whose boundary is $C^{1,1}$.
\end{proof}

Flat pieces can contain distant points with identical
$q=\nabla(p_K^2/2)$. We combine gradient-value concentration with calibrated rows
constructed in $C^1$ boundary coordinates and satisfying
$\int_KX\cdot\nabla g=\lambda\int_Kw\cdot\nabla g$;
positive H\"older normals make their radial correction summable.
Theorems~\ref{thm:R5rowchain} and~\ref{r3:master} give the endpoint.
Appendices~\ref{sec:globalprices}--\ref{glob:regularity} extend
OpenAI's scalar refinement, direction history and global assembly
to all dimensions.

Conditional on local BV output, the fine identity uses the scalar
capacitary tools of~\cite{LW} and retains the Cantor derivative on
contact. For bounded compact probability kernels, common density, support
and signed-band bounds give local gradient continuity; spatial tails give global
continuity (cf.~\cite[Proposition~4.10]{BM}). This transfer allows
general Sobolev inputs and excludes heat and Poisson kernels.
For related uncentered endpoint continuity see~\cite{CMP,CGMSunrise};
for convolution and radial gradient bounds see~\cite{CS,CGR}.

The profile and translation criteria address specified classes in
the question of~\cite{CW}. Further kernel and contact criteria appear
in the appendices; the rough-body note is independent of the fixed
$C^{1,\theta}$ proof. Background is in~\cite{AFP,BNW}.

\section{A summable-price endpoint theorem}
\label{r3:mastersection}
Fix $d\ge2$. Let $\kappa\ge0$ be a bounded probability density
supported in $[-1,1]^d$, and put $B=\|\kappa\|_\infty$,
\[
 A_r^\kappa g(x)=\int\kappa(z)g(x+rz)\,dz,
 \qquad M_\kappa f=\sup_{r>0}A_r^\kappa|f|.
\]
For an axis-parallel unit cube $Q$, let
$\Omega_Q=\{y:\operatorname{dist}_\infty(y,Q)<10\}$.
Fix $0<\alpha_0<1$ and a nondecreasing finite price function
$w:(0,1]\to(0,\infty)$ satisfying
\begin{equation}\label{r12:pricehyp}
 w(s)\ge s^{\alpha_0},\qquad
 W:=\sum_{j=0}^\infty w(2^{-j})<\infty.
\end{equation}
An error decomposition is a finite sum of absolutely continuous
vector measures $\beta=\sum_l\beta_l$, with sizes $0<s_l\le1$
and price $A_w(\beta)=\sum_lw(s_l)|\beta_l|(\mathbb R^d)$.
For $s_l<1$ require zero total vector mass and support in
$\overline B(z_l,C_s s_l)$, with the fixed
$C_s=\max\{2,\sqrt d\}$. Size-one errors have no support or
mass-cancellation condition.
For smooth signed $g$, write $u_r=A_r^\kappa g$ and define
\[
\begin{split}
 I(x)&=\{r\in[1,2]:u_r(x)\ge u_t(x)
           \text{ whenever }t>0,\ |r-t|\le1/4\},\\
 J(x)&=\{r\in[1,2]:u_r(x)=\max_{1\le t\le2}u_t(x)\}\cup\{1,2\}.
\end{split}
\]
Empty nonnegative fiber suprema are zero.

\begin{theorem}\label{r3:master}
Suppose the following hold for every $Q$ and every real
$g\in C_c^\infty$, with $h=\nabla g\,dx$:
\begin{enumerate}
\item If $h=\sum_i v_i\sigma_i+\beta$ as measures on $\Omega_Q$,
where $\sigma_i\ge0$ are finite absolutely continuous measures and
$e,v_i$ are constant unit vectors with $e\cdot v_i\ge1/2$, then
\begin{equation}\label{r3:inputD}
 \int_Q\sup_{r\in I(x)}(-e\cdot\nabla u_r(x))_+\,dx
 \le C_D A_w(\beta).
\end{equation}
\item If $h=\beta$ on $\Omega_Q$, then
\begin{equation}\label{r3:inputB}
 \int_Q\sup_{r\in J(x)}|\nabla u_r(x)|\,dx
 \le C_B A_w(\beta).
\end{equation}
\end{enumerate}
Then for every signed $g\in BV\cap L^1$ and $0<a\le b<\infty$,
\begin{equation}\label{r3:signedband}
 S_{a,b}^\kappa g\in\operatorname{Lip}\cap W^{1,1},\qquad
 \|\nabla S_{a,b}^\kappa g\|_1
 \le C_{d,\alpha_0}(C_D+C_B+B)W|Dg|(\mathbb R^d).
\end{equation}
For $f\in BV\cap L^1$, set $g=|f|$, $V=|Dg|(\mathbb R^d)$,
$q_d=d/(d-1)$ and $H=D(M_\kappa f)$. One has
\begin{equation}\label{r5:mastervariation}
 M_\kappa f\in L^{q_d}\cap BV_{\rm loc},\qquad
 |H|(\mathbb R^d)\le CV,\qquad |H^s|\le C|D^sg|\le C|D^sf|,
\end{equation}
with $C=C_{d,\alpha_0}(C_D+C_B+B)W$. More precisely,
\begin{equation}\label{r5:masterlocal}
 |H|(E)\le C|Dg|(E_{c_dt})+Bt^{-d}|E|V,
 \qquad c_d=1+3\sqrt d,\quad t>0,
\end{equation}
for every Borel $E$ of finite volume. Hence
\begin{equation}\label{r3:globalendpoint}
 f\in W^{1,1}\ \Longrightarrow\
 M_\kappa f\in L^{q_d}\cap W^{1,1}_{\rm loc},\qquad
 \|\nabla M_\kappa f\|_1\le C\|\nabla f\|_1.
\end{equation}
The same homogeneous conclusions hold on $\mathcal X_d$.
\end{theorem}

For power prices write $A_\alpha(\beta)=\sum_l s_l^\alpha|\beta_l|(\R^d)$.
The indicator kernel $\kappa=|K|^{-1}\mathbf1_K$ and price
$w(s)=s^\alpha$, $\alpha_0=\alpha$, recover the original
two-input principle. Reclassifying a positive component as an
unrestricted error costs $w(1)$ times its mass, which is harmless
because $w(1)\le W$. The forced radii in $J$ remain indispensable.

\subsection{Value oscillation is automatic}
\begin{lemma}\label{r3:automaticoscillation}
If $\nabla g\,dx=\sum_i v_i\sigma_i+\beta$ on $\Omega_Q$,
with constant unit vectors $v_i$ and no cone restriction, then
\begin{equation}\label{r3:autovalue}
 \int_Q\sup_{0<t,t'\le4}
       |A_t^\kappa g(x)-A_{t'}^\kappa g(x)|\,dx
 \le C_{d,\alpha_0}B
       \left[A_w(\beta)+\sum_i\sigma_i(\mathbb R^d)\right].
\end{equation}
\end{lemma}
\begin{proof}
Choose a smooth probability density $\psi$ compactly supported in
the interior of $Q$, put $E_Q=Q+[-4,4]^d$, $c=\int\psi g$, and
$p_0=d/(d-1+\alpha_0)\in(1,d/(d-1))$. The set $E_Q$ is convex
and contained in $\Omega_Q$. Segment integration gives
\begin{equation}\label{r3:homotopy}
 g(x)-c=\int L(x,y)\cdot\nabla g(y)\,dy,\quad
 L(x,y)=-(y-x)\int_0^1
   \psi\bigl(x+(y-x)/t\bigr)t^{-d-1}\,dt.
\end{equation}
For $x\in E_Q$, the kernel is supported on the segments joining
$x$ to $\operatorname{supp}\psi$, and
\[
 |L(x,y)|\le C|x-y|^{1-d},\qquad
 |\nabla_yL(x,y)|\le C|x-y|^{-d},\qquad
 \sup_y\|L(\cdot,y)\|_{L^{p_0}(E_Q)}\le C_{\alpha_0}.
\]
Nonzero integrands require $|y-x|\le Ct$, giving the first two
bounds; the last uses $(d-1)p_0<d$.
For $|y-z|\le a$, these kernel bounds give
\[
 \|L(\cdot,y)-L(\cdot,z)\|_{L^{p_0}(E_Q)}
       \le Ca^{d/p_0-(d-1)}=Ca^{\alpha_0}.
\]
Indeed integrate the size bound on $B(z,2a)$ and
$Ca|x-z|^{-d}$ outside it; $1<p_0<d/(d-1)$ makes both integrals
finite. Zero mass and Minkowski then charge each small error by
$Cs_l^{\alpha_0}|\beta_l|\le Cw(s_l)|\beta_l|$.
The uniform kernel norm pays size-one errors and unit-vector sources,
so
\[
 \|g-c\|_{L^{p_0}(E_Q)}
 \le C_{d,\alpha_0}[A_w(\beta)+\sum_i\sigma_i(\R^d)].
\]
All averaging supports with $x\in Q$, $t\le4$ lie in $E_Q$.
The oscillation is at most $2M_\kappa((g-c)\mathbf1_{E_Q})$,
and $M_\kappa F\le C_dB M_{\rm ball}F$. The strong maximal
$L^{p_0}$ bound and H\"older on $Q$ conclude the proof.
\end{proof}

\begin{proof}[Proof of Theorem~\ref{r3:master}]
The terminal pairing in Appendix~\ref{sec:globalprices} gives
zero-mass pieces of relative size $2^{k-l}$ and unrestricted roots.
Applying the two inputs and Lemma~\ref{r3:automaticoscillation}
gives its common local prices. Column normalization and scalar
telescoping sum them to $C_dWV$; roots cost $w(1)a_q$.
The direction law of Appendix~\ref{sec:global-direction-law} is
independent of $w$. Appendix~\ref{sec:globalassembly} therefore pays
directional, endpoint and probability-row terms by
$C_{d,\alpha_0}(C_D+C_B+B)WV$: it uses only $w(1)\le W$ and the
row-derivative convolution $2^{k-l}$, with prefix cancellation
before absolute values. This proves the smooth dyadic-band estimate
without multiplicativity or homogeneity of $w$.
Proposition~\ref{r5:BVtransfer} supplies the last short band,
strict BV approximation and signed localization;
Corollary~\ref{r5:homogeneousBV} gives the homogeneous extension.
Since $W\ge1$, their constants are absorbed in the stated bound.
Neither fine structure nor stability is used.
\end{proof}

\section{Canonical gauge geometry in fixed dimension}\label{r3:canonicalsection}
\begin{definition}\label{def:geometry}
A body of power convexity type $(m,\vartheta)$ is a bounded open convex
set $K$ containing zero, with gauge
$p\in C^{1,\vartheta}_{\rm loc}(\R^d\setminus\{0\})$, $0<\vartheta\le1$,
such that
\begin{equation}\label{r3:supportgap}
 1-\nabla p(v)\cdot u\ge c|u-v|^m\qquad(u,v\in\partial K)
\end{equation}
for $m\ge2$, $c>0$, without symmetry. Geometric data are
$d,m,\vartheta,c$, gauge equivalence constants and a fixed-annulus
$C^{1,\vartheta}$ bound; uniform data means common bounds.
\end{definition}
In quantitative profile statements below, ``admissible'' means this
power-contact class. It is not a hypothesis of the raw-body
case (1) of Theorem~\ref{r3:bodyendpoint}.

For quantitative contact estimates assume Definition~\ref{def:geometry};
the moment identities below hold for every convex body. Put
\[
 \Phi=p^2/2,\qquad R=\max_{\overline K}|w|,\qquad
 D_\Phi(a,b)=\Phi(a)-\Phi(b)-\nabla\Phi(b)\cdot(a-b).
\]
Constants retain $d$ and the geometric data. The next lemma
transfers the support gap to bounded argument sets.

\begin{lemma}\label{r3:canonical}
The function $\Phi=p^2/2$ is convex, positively homogeneous of degree
two, and belongs to $C^{1,\vartheta}_{\mathrm{loc}}(\R^d)$.
When $\vartheta=1$, its gradient is globally Lipschitz.
For every $B>0$ there is $c_B>0$ such that
\begin{equation}\label{r3:canonicalBregman}
 D_\Phi(a,b)\ge c_B|a-b|^m
 \qquad(a,b\in\overline B(0,B)).
\end{equation}
Conversely, this bound on a ball containing $\overline K$ implies
\eqref{r3:supportgap}, with the same constant.
\end{lemma}
\begin{proof}
Convexity follows from that of the nonnegative gauge. Since
$q=p\nabla p=O(|w|)$, it extends by $q(0)=0$. Degree-one
homogeneity and the annular H\"older bound give
$|q(a)-q(b)|\le CB^{1-\vartheta}|a-b|^\vartheta$ on $B_B$:
use annular scaling for comparable radii, and
$|a-b|\ge\max(|a|,|b|)/2$ otherwise. For $\vartheta=1$ this
constant is independent of $B$.

Write $a=su,b=tv$, $u,v\in\partial K$, with
$s,t\le T=B\Lip(p)$. Euler's identity gives
\begin{equation}\label{r3:canonicalidentity}
 D_\Phi(su,tv)=\tfrac12(s-t)^2+
                 st[1-\nabla p(v)\cdot u].
\end{equation}
If $t/2\le s\le2t$, combine
$|a-b|\le R|s-t|+t|u-v|$ with the support gap;
$m\ge2$ and $s,t\le T$ yield $|a-b|^m\le C_TD_\Phi(a,b)$.
Otherwise $s+t\le3|s-t|$ and the radial term suffices,
including zero arguments. Restriction to $\partial K$ gives the
converse.
\end{proof}

\subsection{Euler contraction and cap cancellation}
Write $q(w)=\nabla\Phi(w)$. For a pivot $a\ne0$ put
\[
 h_a=a\cdot q(a)=p(a)^2,\qquad
 C_a(w)=(a\cdot q(w))I-q(w)\otimes a,
\]
and define
\[
 B_a(w)=q(a)\otimes w+C_a(w),\qquad
 W_a(w)=I-h_a^{-1}B_a(w).
\]
Here $(u\otimes v)\zeta=u(v\cdot\zeta)$. The pivot is kept fixed
throughout each comparison of kernels. The following identity replaces
an auxiliary moment frame and its inverse.

\begin{lemma}\label{r5:euler}\label{r3:canonicalframe}
For every $a\ne0$,
\begin{equation}\label{r3:canonicaldet}
 B_a(a)=h_aI,\qquad W_a(a)=0.
\end{equation}
Writing $\Delta q=q(w)-q(a)$, one has the exact formula
\begin{equation}\label{r5:eulerdifference}
 W_a(w)=\frac{\Delta q\otimes a-q(a)\otimes(w-a)
                         -(a\cdot\Delta q)I}{h_a}.
\end{equation}
Fix $0<a_0<a_1<\infty$ and put
$\mathcal A=\{a_0\le p\le a_1\}$. For $a,w\in\mathcal A$,
\[
 \|W_a(w)\|\le C,\qquad
 \|W_a(w)\|\le C\rho^\vartheta
                       \quad (|w-a|<\rho\le1).
\]
There are bounded global $\vartheta$-H\"older extensions
$\widetilde B_a,\widetilde W_a$, agreeing with $B_a,W_a$ on
$\mathcal A$, such that
\[
 \widetilde W_a(a)=0,\qquad
 \|\widetilde W_a(w)\|
       \le C\min\{|w-a|^\vartheta,1\}.
\]
If $\Delta$ is globally Lipschitz, $|\Delta|\le1$,
$\Lip\Delta\le C/\delta$, $0<\delta\le1$, and supported in
$\mathcal A$, then for $E=\widetilde B_a$ or
$E=\widetilde W_a$,
\begin{equation}\label{r3:globalproducts}
 \|E(w)\Delta(w)-E(w')\Delta(w')\|
 \le C\left[\min\{|w-w'|/\delta,1\}
                    +\min\{|w-w'|^\vartheta,1\}\right].
\end{equation}
All constants may depend on $a_0,a_1$ and are uniform in
the pivot $a\in\mathcal A$.
\end{lemma}
\begin{proof}
Substitution and subtraction give~\eqref{r3:canonicaldet}
and~\eqref{r5:eulerdifference}. Since $h_a\ge a_0^2$ and
$a,q(a)$ are bounded, the latter formula gives the shell bounds.
For a smooth compact cutoff $\chi=1$ on $\mathcal A$, supported
away from zero, set
$\widetilde B_a=\chi B_a$ and
$\widetilde W_a=I-h_a^{-1}\widetilde B_a$.
Their uniform H\"older bounds and $\widetilde W_a(a)=0$ give
the global cap bound; splitting a product increment gives
\eqref{r3:globalproducts}. Extensions are used only on kernel
increments; whole-kernel moments retain the original fields.
\end{proof}

\begin{proposition}\label{r5:tangent}\label{prop:higher-distribution-moments}
\label{prop:roughconvexmoments}
The following tangential and radial identities hold for every bounded
open convex body containing zero, with $q=\nabla(p^2/2)$ almost everywhere.
At a nonsmooth pivot, the coefficient $q(a)$ in $B_a$ is any fixed
subgradient of $p^2/2$ at $a$.
For $g\in C_c^\infty(\R^d;\R)$, every center $x$, radius $t>0$,
and vector $a\in\R^d$,
\begin{equation}\label{eq:higher-tangential-moment}
 \int_{x+tK}C_a(y-x)\nabla g(y)\dd y=0.
\end{equation}
If $u_t=A_t^Kg$, then
\begin{equation}\label{eq:higher-radial-moment}
 \partial_tu_t(x)=\frac1{|K|t^{d+1}}
                \int_{x+tK}(y-x)\cdot\nabla g(y)\dd y.
\end{equation}
Consequently, every probability mixture
\[
 \kappa(y)=\int\1_{x+tK}(y)\dd\mu(t),
\]
with $\mu$ supported in a compact subinterval of $(0,\infty)$,
satisfies $\int\kappa C_a(y-x)\nabla g(y)\dd y=0$ for every
$a$. If its scalar radial moment is also zero, then
\begin{equation}\label{r5:balancedcontraction}
 \int\kappa(y)B_a(y-x)\nabla g(y)\dd y=0
                        \qquad\text{for every }a\ne0.
\end{equation}
\end{proposition}
\begin{proof}
The Lipschitz gauge gives $\Phi\in W^{1,\infty}_{\rm loc}$.
Each row of $C_a$ is distributionally divergence free by commutation
of weak mixed derivatives. Choose smooth $0\le\chi_\varepsilon\le1$
converging off $t^2/2$ to $\1_{(-\infty,t^2/2)}$ and let
$F_\varepsilon'=\chi_\varepsilon$. The chain rule gives
$\nabla(F_\varepsilon\circ\Phi)=\chi_\varepsilon(\Phi)q$,
so $\chi_\varepsilon(\Phi)C_a$ is the same divergence-free
construction. Test against $g(x+\cdot)$ and pass to the limit;
convex boundaries are Lebesgue null. This proves the tangential
identity without boundary smoothness. Differentiating
$|K|^{-1}\int_Kg(x+tw)\,dw$ gives the radial identity.
Fubini preserves both under compact scale mixtures, and
$B_a=q(a)\otimes w+C_a$ gives the balanced contraction.
\end{proof}

\begin{remark}\label{r3:flux}
Tangential moments hold for every pivot; only the scalar radial
moment needs balancing. Support and Lipschitz bounds alone do not
replace the radial probability-mixture hypothesis.
\end{remark}

\section{Concentration for compact probability profiles}
\label{r12:concentrationsection}
Fix an admissible body $K$ and a probability $\pi$ on a compact
positive interval $[\ell,L]$. Put $d\sigma(a)=a^{-d}d\pi(a)$ and
$B=\ell^{-d}$, so $\sigma([\ell,L])\le B$.
The single-body case $\pi=\delta_1$ has no window dependence.
For $r\in[1,2]$ write
\[
 D_a(x,r,\delta)=\{y:|p(y-x)-ar|\le a\delta\}.
\]
Measures may have atoms, and weighted masses count overlapping
layers with multiplicity. Only the far-source truncation needs the
stronger profile width.

\begin{lemma}\label{lem:R7sweep}\label{lem:R7anchor}
Let $S$ be a finite positive measure and $Q_t$ a jointly Borel
increasing family. If $a,b\ge0$ and
$R_t\subset Q_{t+a}\setminus Q_{t-b}$ is jointly Borel, then for
every interval $I$, $U>0$ and $T\ge0$,
\[
 |\{t\in I:S(Q_{t-b})\le T,\ S(R_t)>U\}|
                 \le(a+b)(1+T/U).
\]
The families need only be defined on the enlarged interval, with
constant extension outside it; product measures are allowed.
Also, for finite Borel $\mu$ and $H(x,\zeta)\ge0$ Borel, put
\[
 E=\{x:\int H(x,\zeta)d\mu(\zeta)\ge U\}.
\]
Then $|E|\le U^{-1}\iint H\,d\mu\,dx$. In particular, if failure
at any radius forces $H\ge1$ on actual anchors of $\mu$-mass at
least $U$, this $E$ is a radius-independent Borel exceptional set.
\end{lemma}
\begin{proof}
For $F(t)=S(Q_t)$ and $G=\min(F,T+U)$, the first bad event gives
$G(t+a)-G(t-b)\ge U$. A bounded monotone function satisfies
$\int_{\R}[G(t+a)-G(t-b)]dt\le(a+b)(T+U)$; use Tonelli on its
Stieltjes measure. Markov proves the sweep. The witness assertion
is Markov--Tonelli applied to the Borel integral of $H$.
No continuity of $S$ is used.
\end{proof}

\begin{lemma}\label{r11:profiletailenergy}
For a finite positive measure $\sigma$ on $[\ell,L]$ and
$0<h\le1/2$, set $W_h(a)=\sigma([a(1-h),L])$. Then
\[
 \int\frac{d\sigma(a)}{W_h(a)}\le C_{\ell,L}h^{-1}.
\]
The denominator is positive for $\sigma$-almost every $a$.
\end{lemma}
\begin{proof}
Partition $[\ell,L]$ into half-open intervals of length $h\ell/2$,
with a possibly shorter last interval. For $a\in J$, $a(1-h)$ lies below the lower endpoint,
so $W_h(a)\ge\sigma(J)$. Each positive-mass interval contributes
at most one, and there are at most $C_{\ell,L}/h$ such intervals.
\end{proof}

\begin{lemma}\label{a:pairstrip-expanded}
Fix $0<t_0<t_1<\infty$ and $0<q_0\le q_1<\infty$.
Suppose that the canonical potential $\Phi=p^2/2$ has the bounded-region
Bregman estimate of Lemma~\ref{r3:canonical} with exponent $m$.
For distinct $y,z\in\R^d$, $q\in[q_0,q_1]$ and
$0<\epsilon\le\epsilon_0$, put
\[
 E_{y,z,q,\epsilon}=
 \{x:t_0\le p(z-x)\le t_1,
       \ |p(y-x)-q p(z-x)|\le\epsilon\}.
\]
Then
\[
 |E_{y,z,q,\epsilon}|
       \le C\epsilon |y-z|^{-(m-1)}.
\]
The constants depend only on the indicated compact intervals and the
geometric data, and the assertion is valid without assumptions on
the measures which will later be integrated in $y,z$.
\end{lemma}
\begin{proof}
For $q<1$, interchange $y,z$ and replace $(q,\epsilon)$ by
$(q^{-1},\epsilon/q)$; both radii remain in fixed positive
intervals. Assume $q\ge1$, put $r=|y-z|$, $e=(y-z)/r$, and
choose $B$ containing both strip arguments. On
$D=B(y,B)\cap B(z,B)$ the Bregman bound gives
\[
 [\nabla\Phi(y-x)-\nabla\Phi(z-x)]\cdot e
 \ge2c_Br^{m-1}.
\]
For $F=\Phi(y-x)-q^2\Phi(z-x)$, therefore,
$\partial_eF\le-2c_Br^{m-1}+(q^2-1)M_B$,
where $M_B=\sup_{B_B}|\nabla\Phi|$.
If $q-1\le c_*r^{m-1}$, this is at most $-c_Br^{m-1}$;
the strip has $|F|\le C\epsilon$, so linear Fubini gives the
claim. Otherwise write $x=z-t\theta$, $\theta\in\partial K$.
Subadditivity makes $p(y-z+t\theta)-qt$ decrease at rate at
least $q-1$. Its strip has length at most
$2\epsilon/(q-1)\le C\epsilon r^{-(m-1)}$;
gauge-polar integration finishes the proof.
\end{proof}

\begin{lemma}\label{audit:farbandlength}
Fix $0<t_-<t_+<\infty$, $0<q_-\le q_+<\infty$ and $C_0<\infty$.
Let $\theta\in\partial K$, $w\ne0$, and $q\in[q_-,q_+]$.
For sufficiently small $\delta>0$, depending only on these data and $K$, put
\[
 \mathcal B=\{t\in[t_-,t_+]:
          |p(w+t\theta)-qt|\le C_0\delta\}.
\]
Then $\mathcal B$ is the union of at most two intervals and
\begin{equation}\label{audit:farbandstrong}
 |\mathcal B|\le \frac{C\delta^{1/m}}{|w|}.
\end{equation}
In particular, if $q<1$ and $\delta\le c(1-q)$, then
\begin{equation}\label{audit:farbandoriginal}
 |\mathcal B|\le \frac{C\delta^{1/m}}{1-q}.
\end{equation}
If $|w|\ge\rho$, the stronger bound
$|\mathcal B|\le C\delta^{1/m}/\rho$ holds without a comparison
between $\delta$ and $1-q$.
\end{lemma}
\begin{proof}
Convexity of $F(t)=p(w+t\theta)-qt$ gives at most two band
components. For $t_1<t_2$ in one, put $t_0=(t_1+t_2)/2$,
$a_i=w+t_i\theta$ and $r_i=p(a_i)=qt_i+O(\delta)$, $i=0,1,2$.
These radii stay in a fixed positive interval. With
$u_i=a_i/r_i$, $\lambda=r_1/(r_1+r_2)$ and
$z=\lambda u_1+(1-\lambda)u_2$, both weights are bounded below and
\[
 p(z)=\frac{2r_0}{r_1+r_2},\qquad
 C\delta\ge\tfrac12-\Phi(z)
 =\lambda D_\Phi(u_1,z)+(1-\lambda)D_\Phi(u_2,z)
 \ge c|u_1-u_2|^m.
\]
Since $v_i=a_i/(qt_i)$ satisfy $|u_i-v_i|\le C\delta$ and
$|v_1-v_2|=|w||t_2-t_1|/(qt_1t_2)$, each component has
diameter at most $C\delta^{1/m}/|w|$.
If $q<1$ and the band is nonempty, then
$\Lip(p)|w|\ge(1-q)t_--C_0\delta$.
For $\delta\le t_-(1-q)/(2C_0)$ this gives $|w|\ge c(1-q)$.
The two remaining bounds follow.
\end{proof}

\begin{theorem}\label{r12:concentrationmaster}\label{r11:probabilitySC}\label{r9:weightedSC}\label{r10:finiteweightedSC}\label{thm:SC}
For finite positive Borel $S,\nu$, $U>0$, $T\ge0$,
$0<\rho\le1$ and $0<\delta\le c\rho^{2m}$, there is a
radius-independent Borel set $E$ such that
\begin{align*}
 |E|\le C\bigg[&\delta\rho^{-(m-1)}(\nu(\mathbb R^d)/U)^2
 +\delta\rho^{-m}(\nu(\mathbb R^d)/U)(1+T/U)\\
 &+\delta^{1/m}\rho^{-(m+1)}(\nu(\mathbb R^d)/U)(T/U)\bigg].
\end{align*}
Write $D_a=\{y:|p(y-x)-ar|\le a\delta\}$. For $x\notin E$,
$r\in[1,2]$, the premises
\[
 \int\nu(D_a)\,d\sigma(a)\ge U,\qquad
 \int S(x+a(r-\delta)K)\,d\sigma(a)\le T
\]
imply that an anchor $z$ in one of the shells $D_a$ satisfies
\[
 \int(S+\nu)(D_a\setminus B(z,\rho))\,d\sigma(a)\le2U.
\]
All constants depend only on $K,\ell,L$.
Two refinements preserve the stronger special cases:
\begin{enumerate}
\item If $S=0$, no source premise and no relation between $\delta$
and $\rho$ are required (apart from $0<\delta\le c$); the bound is
\begin{equation}\label{r12:sourcefreeSC}
 |E|\le C\delta\rho^{-(m-1)}(\nu(\mathbb R^d)/U)^2,
\end{equation}
and the remote weighted $\nu$ mass is at most $U$.
\item If $\pi=\delta_1$, the width hypothesis improves to
$\delta\le c\rho^m$, with constants depending only on $K$, and
\begin{equation}\label{eq:SCsize}
 |E|\le C\left[\frac\delta{\rho^{m-1}}
       (\nu(\mathbb R^d)/U)^2+
       \frac\delta{\rho^m}\frac{\nu(\mathbb R^d)}U(1+T/U)\right].
\end{equation}
In this case the premises and conclusion are exactly
\begin{equation}\label{eq:SCpremise}
 \nu(D_1)\ge U,\qquad S(x+(r-\delta)K)\le T,
\end{equation}
\begin{equation}\label{eq:SCcap}
 (S+\nu)(D_1\setminus B(z,\rho))\le2U.
\end{equation}
\end{enumerate}
\end{theorem}
\begin{proof}
Fix $I=[\ell/2,3L]\Subset I'=(\ell/4,4L)$. For an anchor
$(a,z)$ write $t=p(z-x)\in I$, $\theta=(z-x)/t$, $q=b/a$, and set
\[
 R_{ab}=\{y:|p(y-x)-qt|\le C_1\delta,\ |y-z|\ge\rho\},
 \qquad
 P_a=\int S(x+[(b/a)t-A\delta]K)\,d\sigma(b),
\]
where $C_1\ge4L$ and $A\ge2C_1$. At every actual anchor
$z\in D_a$, $P_a\le T$; all inner radii are positive for small
fixed widths. Put $h=c_0\rho^m$, $q_*=1-h$, and split target
scales into $N_a=\{b:b/a\ge1-2h\}$ and $F_a=N_a^c$.

\emph{A common truncated sweep.}
Let $q_{\max}=L/\ell$. Gauge equivalence gives a fixed $M$ such
that both $p(w+v\theta)<qv$ ($v\in I'$, $0<q\le q_{\max}$)
and the bands defining $R_{ab}$ imply $|w|<M$.
For $\rho\le|w|\le M$ put
\[
 F_{q,w}(v)=v[\Phi(\theta+w/v)-q^2/2],\qquad v\in I'.
\]
Its arguments lie in a fixed ball, so Bregman coercivity gives
$D_\Phi(\theta,\theta+w/v)\ge c_*\rho^m$. Choose $c_0\le1/8$ small enough.
Degree-two homogeneity then yields, uniformly for $1-2h\le q\le q_{\max}$,
\begin{equation}\label{eq:R9truncatedmonotonicity}
 F_{q,w}'(v)=-D_\Phi(\theta,\theta+w/v)+(1-q^2)/2
                  \le-\gamma\rho^m,\qquad \gamma=c_*/2.
\end{equation}
Consequently
$\widetilde E_v^q=\{y:|y-z|\ge\rho,\ p(y-z+v\theta)<qv\}$
increases on $I'$: on $|y-z|\le M$ membership is $F_{q,y-z}(v)<0$;
outside that ball it is always false. This is the required truncated
family; the untruncated family need not increase when $q<1$.

On a band, $|F_{q,w}(t)|\le C_b\delta$. Choose $A_0$ with
$\gamma A_0>C_b+q_{\max}A+1$ and put
$v_0=A_0\delta\rho^{-m}$. Integration of
\eqref{eq:R9truncatedmonotonicity} gives the strict crossing
\begin{equation}\label{eq:R9nearbandcrossing}
 y\in R_{ab},\ b\in N_a
 \quad\Longrightarrow\quad
 y\in\widetilde E_{t+v_0}^{q}\setminus\widetilde E_{t-v_0}^{q}.
\end{equation}
The same derivative estimate proves
\begin{equation}\label{eq:R9earliersource}
 \widetilde E_{t-v_0}^{q}
       \subset z-t\theta+(qt-A\delta)K.
\end{equation}
Indeed membership implies $F_{q,w}(t)<-\gamma A_0\delta$, hence
\[
 p(w+t\theta)^2<q^2t^2-2tq_{\max}A\delta<(qt-A\delta)^2.
\]
The width $\delta\le c\rho^m$ keeps all shifted parameters in $I'$.

Apply Lemma~\ref{lem:R7sweep} to the product measure
$\sigma|_{N_a}\otimes S$, the increasing family
$\mathcal Q_v=\{(b,y):b\in N_a,\ y\in\widetilde E_v^{b/a}\}$,
and its remote band. The earlier mass is at most $P_a$ by
\eqref{eq:R9earliersource}; crossing uses shifts $v_0,v_0$.
With threshold $U/2$ and gauge-polar integration this gives
\begin{equation}\label{r11:nearprofilebudget}
 |\{x:t\in I,\ P_a\le T,\
       \int_{N_a}S(R_{ab})d\sigma(b)>U/2\}|
       \le C\delta\rho^{-m}(1+T/U).
\end{equation}
For the sweep, extend the family constantly from a closed enlarged
interval containing all shifts inside $I'$. Its strict crossings
include closed-band boundary atoms; no continuity of masses is used.

\emph{Far scales.}
For $b\in F_a$, the ratio gap $q_*-q>h$ gives on $R_{ab}$
\[
 F_{q_*,w}(t)=F_{q,w}(t)+(q^2-q_*^2)t/2\le-c_f\rho^m.
\]
The absolute derivative is uniformly bounded on the fixed ball.
Thus $\epsilon=c_2\rho^m$, with $c_2$ small, satisfies
\begin{equation}\label{eq:R9farentry}
 y\in R_{ab},\ b\in F_a
      \quad\Longrightarrow\quad y\in\widetilde E_{t-\epsilon}^{q_*}.
\end{equation}
The stronger width $\delta\le c\rho^{2m}$ ensures
$\epsilon>2v_0$ and all parameters remain in $I'$.
Put $W(a)=\sigma([aq_*,L])$, positive for almost every anchor
scale by Lemma~\ref{r11:profiletailenergy}. Nestedness gives
\[
 P_a\ge W(a)S(z-t\theta+(q_*t-A\delta)K).
\]
For fixed $a,z,\theta$, call $t\in I$ good when the right side
is at most $T$. If there are good parameters, let $\tau$ be their
supremum and $G_\theta=\widetilde E_{\tau-\epsilon}^{q_*}$.
By monotonicity and~\eqref{eq:R9farentry}, every far-band point
at a good parameter lies in $G_\theta$. Also
$S(G_\theta)\le T/W(a)$: choose a good $t$ so close to $\tau$
that $\tau-\epsilon<t-v_0$, then use
\eqref{eq:R9earliersource}. This avoids any limiting boundary mass.
If there is no good parameter, the far event is empty.

Lemma~\ref{audit:farbandlength} bounds each far-band parameter
length by $C\delta^{1/m}/\rho$. Tonelli on each ray therefore gives
\[
 \int_{t\ {\rm good}}\int_{F_a}S(R_{ab})d\sigma(b)\,dt
 \le C\delta^{1/m}\rho^{-1}T/W(a).
\]
Markov at $U/2$ and gauge-polar integration yield
\begin{equation}\label{r11:farprofilebudget}
 \begin{split}
 |\{x:t\in I,&\ W(a)S(x+(q_*t-A\delta)K)\le T,\\
                &\int_{F_a}S(R_{ab})d\sigma(b)>U/2\}|
       \le C\delta^{1/m}\rho^{-1}\frac{T}{UW(a)}.
 \end{split}
\end{equation}
Only this original Borel event is integrated in $\theta$;
$\tau,G_\theta$ are raywise auxiliaries and need no measurable choice.

\emph{One anchor witness.}
Let $C_a^N,C_a^F$ be the indicators of the near and far events
above and put $V_a=U^{-1}\int\nu(R_{ab})d\sigma(b)$.
Apply Lemma~\ref{lem:R7anchor} with
\[
 d\mu=d\sigma(a)d\nu(z),\qquad
 H=\1_{\{p(z-x)\in I\}}(V_a+C_a^N+C_a^F).
\]
Failure of every cap makes $H\ge1$ on all actual anchors, of
joint mass at least $U$. Lemma~\ref{a:pairstrip-expanded} bounds
the $V_a$ contribution by
$C\delta\rho^{-(m-1)}(\nu(\R^d)/U)^2$.
Equation~\eqref{r11:nearprofilebudget} gives the second term, while
\eqref{r11:farprofilebudget} and
$\int W(a)^{-1}d\sigma(a)\le C\rho^{-m}$ give the third.
Null anchor scales with $W(a)=0$ are discarded.
For $S=0$ only $V_a$ is needed, so neither source width restriction
is used and the remote mass is at most $U$.
For $\pi=\delta_1$, $q=1$ and $F_a$ is empty; the pair and near
bounds give~\eqref{eq:SCsize} with only $\delta\le c\rho^m$.
\end{proof}

\subsection{Concentration in gradient values}
\label{subsec:gradientcaps}
This alternative concentration target requires only a $C^1$ convex
gauge; it may extend along an entire flat boundary piece.

\begin{lemma}\label{lem:gradientcoercivity}
Let $\Phi$ be convex and $C^1$ on a neighborhood of a fixed
closed ball, and let $q=\nabla\Phi$ have continuity modulus
$\Omega$ there.  Suppose the points and the steps below remain
in that ball.  If $|q(a)-q(b)|\ge\varepsilon$ and
$\Omega(r)\le\varepsilon/2$, then
\[
 D_\Phi(a,b)\ge\varepsilon r/2.
\]
The same assertion holds with $a,b$ interchanged.
\end{lemma}
\begin{proof}
For $e=(q(a)-q(b))/|q(a)-q(b)|$, convexity at $b$ and the
fundamental theorem of calculus on $[a-re,a]$ give
\[
 D_\Phi(a,b)\ge r|q(a)-q(b)|-\int_0^r\Omega(t)\,dt
                  \ge\varepsilon r/2.
\]
Interchange $a,b$ for the other inequality.
\end{proof}

\begin{proposition}\label{prop:C1gradientcaps}
Let $K$ be a bounded open convex body containing zero, with
$p\in C^1(\R^d\setminus\{0\})$.  Put $\Phi=p^2/2$ and
$q=\nabla\Phi$, using its continuous extension at zero.
Fix a nondecreasing continuous modulus $\Omega$, with
$\Omega(0)=0$, for $q$ on a sufficiently large bounded argument
ball.  There are constants $r_0,c_0,c_1,c_2>0$, depending only
on the fixed body data, with the following property.
For sufficiently small $\varepsilon>0$ put
\[
 r_\varepsilon=\sup\{0\le r\le r_0:
                              \Omega(r)\le c_0\varepsilon\},
 \qquad m_\varepsilon=c_1\varepsilon r_\varepsilon.
\]
If $0<\delta\le\min\{\delta_0,c_2m_\varepsilon r_\varepsilon\}$,
with a fixed sufficiently small $\delta_0>0$, finite
positive Borel measures $S,\nu$ and numbers $U>0,T\ge0$ admit
a radius-independent Borel set $E$ satisfying
\[
 |E|\le\frac{C\delta}{m_\varepsilon}
 \left[\left(\frac{\nu(\R^d)}U\right)^2
       +\frac{\nu(\R^d)}U\left(1+\frac TU\right)\right].
\]
For $x\notin E$, $1\le r\le2$, write
$D=D(x,r,\delta)=\{y:|p(y-x)-r|\le\delta\}$.
If
\[
 \nu(D)\ge U,\qquad S(x+(r-\delta)K)\le T,
\]
there is an actual anchor $z\in D$ for which
\[
 (S+\nu)\bigl(D\cap
   \{|q(y-x)-q(z-x)|>\varepsilon\}\bigr)\le2U.
\]
No strict convexity or positive Euclidean support-gap modulus is
assumed.  For $S=0$ the quadratic exceptional-set term suffices,
and the extra width restriction involving $r_\varepsilon$ is
unnecessary; $0<\delta\le\delta_0$ is still assumed.
\end{proposition}
\begin{proof}
Use $I=[3/4,9/4]\Subset[1/2,5/2]$ and a fixed argument ball
containing all relevant arguments and a unit descent step.
The preceding lemma gives $D_\Phi(a,b)\ge cm_\varepsilon$
whenever $|q(a)-q(b)|\ge c\varepsilon$, after fixing $c_0,c_1$.
For $v=y-z\ne0$, the function
$F(x)=\Phi(y-x)-\Phi(z-x)$ is monotone along lines parallel to
$v$, because
\[
 [q(y-x)-q(z-x)]\cdot v
 =D_\Phi(y-x,z-x)+D_\Phi(z-x,y-x)\ge0.
\]
On the separated subset,
$-\partial_eF\ge cm_\varepsilon/|v|\ge c'm_\varepsilon$,
$e=v/|v|$, and the shell conditions give $|F|\le C\delta$.
Monotonicity bounds the integral of $|\partial_eF|$ through this
value interval by $2C\delta$, hence the restricted line length by
$C\delta/m_\varepsilon$. Contributing centers lie in a fixed
ball about $z$, with bounded projection onto $e^\perp$.
Fubini therefore gives
\begin{equation}\label{eq:gradientcap-pair}
 \begin{split}
 \bigl|\{x:\ &p(z-x)\in I,\ |p(y-x)-p(z-x)|\le2\delta,\\
 &|q(y-x)-q(z-x)|>\varepsilon\}\bigr|
 \le C\delta/m_\varepsilon.
 \end{split}
\end{equation}
For $v=0$ the set is empty.

For the source sweep fix $z\in\R^d$ and $\theta\in\partial K$ and write
$x=z-t\theta$, $w=y-z$. The defect
\[
 h_w(t)=t[\Phi(\theta+w/t)-1/2]
       =\frac{p(w+t\theta)^2-t^2}{2t},\qquad
 h_w'(t)=-D_\Phi(\theta,\theta+w/t)\le0
\]
is globally monotone. Let $R_t$ denote the two restricted shell
conditions in~\eqref{eq:gradientcap-pair}. On $R_t$,
$|h_w(t)|\le C\delta$ and
$|q(\theta+w/t)-q(\theta)|\ge c\varepsilon$.
Set $v_0=C_3\delta/m_\varepsilon$. Our width restriction makes
the arguments move by at most $r_\varepsilon$ on $[t,t+v_0]$,
so their gradient separation persists and $-h_w'\ge m_\varepsilon$.
Choosing $C_3$ large gives, for fixed $B>4$,
\begin{equation}\label{eq:gradientcap-strictentry}
 y\in R_t\quad\Longrightarrow\quad h_w(t+v_0)\le-B\delta.
\end{equation}
The enlarged parameter interval contains these steps.

Convexity and $\theta\in\partial K$ make
$Q_t=z-t\theta+(t-2\delta)K$ increasing. The closed shell has
$R_t\cap Q_t=\varnothing$, while~\eqref{eq:gradientcap-strictentry},
with $s=t+v_0$, gives
$p(y-z+s\theta)^2\le s^2-2Bs\delta<(s-2\delta)^2$.
Thus $R_t\subset Q_{t+v_0}\setminus Q_t$, including boundary
atoms. Lemma~\ref{lem:R7sweep} with shifts $(v_0,0)$ and
gauge-polar integration yield

\begin{equation}\label{eq:gradientcap-source}
 |\{x:p(z-x)\in I,\ P(x,z)\le T,
                  S(R_\varepsilon(x,z))>U\}|
 \le\frac{C\delta}{m_\varepsilon}(1+T/U),
\end{equation}
where
\[
 \begin{gathered}
 P(x,z)=S(x+(p(z-x)-2\delta)K),\\
 R_\varepsilon(x,z)=\{y:|p(y-x)-p(z-x)|\le2\delta,
                   |q(y-x)-q(z-x)|>\varepsilon\}.
 \end{gathered}
\]
Apply Lemma~\ref{lem:R7anchor} with $d\mu=d\nu(z)$ and
\[
 H(x,z)=\1_{\{p(z-x)\in I\}}
 \left[\frac{\nu(R_\varepsilon(x,z))}{U}
       +\1_{\{P(x,z)\le T,\ S(R_\varepsilon(x,z))>U\}}\right].
\]
At an actual anchor, $t=p(z-x)$ has $t-2\delta\le r-\delta$,
so $P(x,z)\le T$, and the failed cap's remote mass lies in
$R_\varepsilon(x,z)$. If it exceeds $2U$, then $H(x,z)\ge1$.
Failure at one radius therefore forces $\int H\,d\nu\ge U$.
The pair and source bounds control this single Borel exceptional
set. For $S=0$, only the pair term is needed.

\end{proof}

\section{One comparison theorem for local cancellation}
\label{r3:localestimates}
Let $K$ be bounded, open and convex with $0\in K$, and let $\pi$
be a probability on $[\ell,L]\Subset(0,\infty)$, with
$d\sigma(a)=a^{-d}d\pi(a)$. Normalize $LK\subset[-1,1]^d$ and
use $Q,\Omega_Q,g,h$ from Section~\ref{r3:mastersection}. Put
\[
 H_r(y)=\int\1_{x+arK}(y)d\sigma(a),\quad
 u_r=(|K|r^d)^{-1}\int H_rg,\quad B_H=\sigma([\ell,L]).
\]
The raw and probability bounds are respectively $H_r\le B_H$ and
$\|\kappa_\pi\|_\infty=B_H/|K|$, where
$\kappa_\pi=|K|^{-1}\int\1_{aK}d\sigma(a)$; equality holds on
$\ell K$. Thus $\nabla u_r=(|K|r^d)^{-1}\int H_r\,dh$.
Constants retain $K,\ell,L$ and this fixed normalization.
Write $\mathcal J$ for the master's fiber $J$, including its endpoints.
Assume on $\Omega_Q$ that
\begin{equation}\label{eq:decomposition}
 h=\sum_i v_i\sigma_i+\beta,\quad S=\sum_i\sigma_i,\quad
 \beta=\sum_\ell\beta_\ell,\qquad |e|=|v_i|=1,\ e\cdot v_i\ge\kappa>0.
\end{equation}
Components are finite and absolutely continuous; put $0<s_\ell\le1$,
$m_\ell=|\beta_\ell|(\R^d)$ and $A=\sum_\ell w(s_\ell)m_\ell$.
As in the master theorem, $s_\ell<1$ requires
\begin{equation}\label{eq:errors}
 \beta_\ell(\R^d)=0,\qquad
 \supp\beta_\ell\subset\overline B(z_\ell,C_ss_\ell),
\end{equation}
while size-one pieces are unrestricted. The single body is $\pi=\delta_1$.

\subsection{Common balance, one pivot and the fine prices}
Use the matrices $C_a,B_a,W_a$ of Section~\ref{r3:canonicalsection},
with a fixed subgradient of $p^2/2$ at a nonsmooth pivot.
The distributional identities of Proposition~\ref{r5:tangent},
followed by Fubini, give
\begin{equation}\label{eq:moments}
 \int H_r C_a\,dh=0\quad(a\ne0),\qquad
 \partial_ru_r(x)=(|K|r^{d+1})^{-1}
                  \int H_r(y)(y-x)\cdot dh(y).
\end{equation}
\begin{lemma}\label{lem:balanced}\label{lem:W}
If $u_r\ge u_{r-\delta},u_{r+\delta}$, $r\in[1,2]$ and
$0<\delta\le1/4$, there is a probability density $k$ on
$[r-\delta,r+\delta]$, bounded by $C/\delta$, for which
$K_\delta=\int k(t)H_tdt$ satisfies
\begin{equation}\label{eq:kernelproperties}
 \begin{gathered}
 0\le K_\delta\le B_H,\quad \Lip K_\delta\le CB_H/\delta,\quad
 K_\delta\ge\int\1_{x+a(r-\delta)K}d\sigma(a),\\
 \int K_\delta(y)(y-x)\cdot dh(y)=0.
 \end{gathered}
\end{equation}
Tangential moments vanish, so every fixed pivot has
$\int K_\delta B_a\,dh=0$. For $\pi=\delta_1$ the plateau is one
and the support is $x+(r+\delta)\overline K$.
For a fixed shell anchor $z$, put $a=z-x$ and
$W_z(y)=I-h_a^{-1}B_a(y-x)$. Two balanced kernels, or a sharp
balanced limit, satisfy
\begin{equation}\label{eq:Widentity}
 \int(K_1-K_0)dh=\int W_z(K_1-K_0)dh.
\end{equation}
On fixed positive argument annuli $W_z$ is uniformly bounded
almost everywhere. If $q$ is $\vartheta$-H\"older, $0<\vartheta\le1$,
then
\begin{equation}\label{r3:holdercap}
 \|W_z(y)\|\le C\rho^\vartheta\quad(|y-z|<\rho).
\end{equation}
The shell-supported products then have bounded global H\"older
extensions. For a general modulus $\Omega$ of $q$, the cap bound
is $C[\Omega(\rho)+\rho]$.
\end{lemma}
\begin{proof}
On the two halves of $[r-\delta,r+\delta]$, normalize
$t^{-d-1}$;~\eqref{eq:moments} gives opposite weak radial signs,
so a convex combination balances them. Tail integration of $k$
at $p(y-x)/a$ gives all scalar bounds. Tangential moments pass by
Fubini, and subtraction of the zero $B_a$ moments gives
\eqref{eq:Widentity}. The exact Euler difference formula gives
the stated cap moduli; under H\"older regularity its cutoff
construction gives the extensions. Whole-kernel moments always
retain the original fields.
\end{proof}

\subsection{Concentration and cap rows}
\label{subsec:meshprice}\label{subsec:R5rows}
Write $\omega_j$ for the inverse-volume weighted shell indicator
at width $\delta_j$ and $\iota_j$ for its inner plateau; for a single
body they are $\1_{\{|p(y-x)-r|\le\delta_j\}}$ and
$\1_{x+(r-\delta_j)K}$. A concentration budget $G_j$ means that
finite positive $S,\nu$ and $U>0,T\ge0$ admit a radius-independent
Borel exceptional set with
\begin{equation}\label{eq:meshSCbudget}
 |E|\le CG_j[(\nu(\R^d)/U)^2+(\nu(\R^d)/U)(1+T/U)],
\end{equation}
outside which the shell/interior premises select an actual shell
anchor $z$ and a Borel cap $\mathcal C_j(x,r,z)$ containing it, with
\[
 \int_{\mathcal C_j(x,r,z)^c}\omega_j\,d(S+\nu)\le CU.
\]
Both Euclidean and gradient-value caps are allowed. Balanced kernels,
their balanced convex combinations and sharp limits satisfy
$|K_i-K_j|\le C\omega_j$ and $K_i\ge\iota_j$ for $i\ge j$.
All row conditions below are uniform in the center, radius, direction
and anchor, without measurable selections. Fine products may use
bounded global extensions agreeing on increment supports; whole-kernel
moments use the original rows on $\Omega_Q$.

\begin{theorem}\label{thm:R5rowchain}
Let $w$ satisfy \eqref{r12:pricehyp}. For positive thresholds $u_j$ let
$0<h_{j+1}\le h_j\le\delta_{j+1}\le\delta_j\downarrow0$, and put
\[
 b_j=\sup_{0<s\le1}\frac{\min(1,s/\delta_j)}{w(s)},\qquad
 H_j=w(h_j)^{-1},\qquad C_j=1+H_j+b_j.
\]
Let $a_j,r_j,t_j,f_j$ be nonnegative and let $D_j\ge1$ be
nondecreasing. Suppose that after a cap is selected
at level $j$, for every unit direction $e$ there is one row $B_j$
with $W_j=e-B_j$ satisfying
\[
 \|W_j\|_{L^\infty(\Omega_Q)}\le D_j,\qquad
 |W_j|\le a_j\quad\hbox{on the cap inside the level-$j$ shell}.
\]
For every balanced kernel $K$ of width at most $\delta_j$, including
its relevant convex combinations, suppose a correction row $R_K$
satisfies
\begin{equation}\label{eq:R5rowmoment}
 \int B_jK\cdot dh=\int R_K\cdot dh,\qquad
 |R_K|\le r_jK,\qquad
 \left|\int R_K\cdot d\beta\right|\le A t_j,
 \quad A=A_w(\beta).
\end{equation}

The scalar kernel and fine-increment bounds also cover the initial
radial smoothings used before balance is available at a forced endpoint.
These are nonnegative probability smoothings of the stated widths and
have the same density and shell bounds; the row moment identity is
used only after a balanced comparison is available.
Assume, for all such scalar increments, that
\[
 \left|\sum_{s_\ell<h_i}\int(K_{i+1}-K_i)\,d\beta_\ell\right|
 \le CAf_i.
\]
For each fixed row selected at $j$ and every $i\ge j$, assume
\begin{equation}\label{eq:R5fine}
 \left|\sum_{s_\ell<h_i}\int W_j(K_{i+1}-K_i)\cdot d\beta_\ell\right|
 \le CD_j[Af_i+F_i(x)],\qquad
 \|F_i\|_{L^1(Q)}\le CAf_i.
\end{equation}
Here the nonnegative measurable $F_i$ are common majorants for all
radii, rows, anchors and tail levels; $F_i=0$ is permitted.
Suppose
\begin{equation}\label{eq:R5rowsums}
 \sum_j\bigl[(a_j+r_j)C_j+D_j(u_j+f_j)+t_j\bigr]<\infty,
 \qquad
 \sum_je_j\log(e+b_{j+1})<\infty,
\end{equation}
where
\[
 e_j=G_j(1+C_j/u_j)^2+\delta_jC_j/u_j.
\]
Then the two local inputs of Theorem~\ref{r3:master} hold. If
concentration is verified only for $S=0$, the pure-error input holds.
All conclusions are uniform over families with common data in these
hypotheses.
\end{theorem}
\begin{proof}
For each finite kernel family chosen below, put
$p_j=\int K_j\,d\beta$, $P_j=\int K_j\,dh$,
$N_j=(-e\cdot P_j)_+$,
$\nu_j=\sum_{s_\ell\ge h_j}|\beta_\ell|$ and
$V_j=\int\omega_j\,d\nu_j$. Constant subtraction, the kernel
Lipschitz bound and the scalar fine hypothesis give
\begin{equation}\label{eq:meshprices}
 |p_j|\le CAb_j,\qquad \nu_j(\R^d)\le AH_j,
\end{equation}
\begin{equation}\label{eq:meshincrement}
 \left|\int(K_{i+1}-K_i)\,d\beta\right|\le C(V_i+Af_i).
\end{equation}
For $A=0$, cone positivity (or $h=0$ in the pure-error case)
settles the claim. Otherwise fix $\Lambda\ge L_0A$ and set
$U_j=\Lambda u_j$, $T_j=C_0\Lambda C_j$.
Call $j$ active if $V_j\ge U_j$ and apply concentration.
Also exclude the endpoint tests
\[
 E_j^\partial=\left\{x:\nu_j\{y:
 \min_{t\in\{1,2\}}|p(y-x)-t|\le2\delta_j\}\ge U_j\right\}.
\]
Annular volume and Markov give
$|E_j^\partial|\le C\delta_j\nu_j(\R^d)/U_j$;
compact profiles use the same annuli integrated against their
inverse-volume weighted scale law. Outside these tests, activation
keeps $r$ more than $\delta_j$ from the forced endpoints, so all
later kernels can be balanced. The combined exceptional sets satisfy
$|E_j|\le C(A/\Lambda)e_j$.
Discard a finite prefix so that $\delta_0\le1/4$. Fix $J\le L$,
$x\notin\bigcup_{i=J}^{L-1}E_i$ and an admissible radius
($r\in I(x)$, or $r\in\mathcal J(x)$ when $S=0$). Set
$k=\min(\{j\in[J,L):j\text{ is active}\}\cup\{L\})$.
Choose radial smoothings $K_i$ of width $\delta_i$, arbitrary if
$k=L$; otherwise choose them arbitrarily before $k$ and balanced
from $k$ through $L$.
This is possible because activation outside $E_k^\partial$ places
$r$ more than $\delta_k$ from either forced endpoint. At a forced
endpoint no index can be active, so $k=L$ and no balance is needed.
All these kernels obey the scalar estimates
\eqref{eq:meshprices}--\eqref{eq:meshincrement}, independently of balance.

Let $j<l\le L$, where $j$ is active and the indices $j<i<l$ are
inactive. Take $l$ to be the next active index, or the terminal layer
$L$ if there is no earlier active index. The right endpoint need
not be active. Both kernels are balanced by the preceding
construction. Scalar increments on the intervening inactive gap give

\[
 |p_j|+|p_l|\le CA(b_j+H_j+f_j)
       +C\sum_{j<i<l}(U_i+Af_i)\le C\Lambda C_j.
\]
When $N_l>N_j$, take $K_*=K_j$ if $e\cdot P_j\le0$;
otherwise take the convex combination of $K_j,K_l$ with zero
$e$-moment. Then
$(-e\cdot\int K_*dh)_+=N_j$ and
$K_l-K_*=\xi(K_l-K_j)$, $0\le\xi\le1$.
Write $h_S=\sum_i v_i\sigma_i$. Since $N_l>N_j$, both $e\cdot P_l$ and $e\cdot P_*$ are
nonpositive. For either of these nonnegative kernels, cone
positivity gives
\[
 \kappa\int K\,dS
 \le e\cdot\int K\,dh_S
 =e\cdot\int K\,dh-e\cdot\int K\,d\beta
 \le\left|\int K\,d\beta\right|.
\]
The previously obtained bounds for $p_j,p_l$, also valid for their
convex combination, imply
$\int(K_l+K_*)dS\le C\Lambda C_j$. This verifies the source
premise required to choose the cap at level $j$.

The plateau and activation therefore select a cap; retain its one
row throughout the gap. The exact comparison is
\[
 e\cdot(P_l-P_*)=\int W_j(K_l-K_*)\cdot dh
                         +\int(R_l-R_*)\cdot dh.
\]
Its source terms cost $C\Lambda a_jC_j+CD_jU_j$ on and off the
cap. The correction costs
$C\Lambda r_jC_j+CAt_j$. For clarity, let
$\beta_j^{\rm c}=\sum_{s_\ell\ge h_j}\beta_\ell$ and
$\beta_j^{\rm f}=\sum_{s_\ell<h_j}\beta_\ell$.
The monotonicity of $h_j$ gives, for $i\ge j$,
\[
 \beta_j^{\rm f}
 =\beta_i^{\rm f}+\sum_{h_i\le s_\ell<h_j}\beta_\ell.
\]
Pair the single increment $K_l-K_j$ on $\beta_j^{\rm c}$ with
the fixed row $W_j$. Its cap part costs $CAa_jH_j$ and its
off-cap part costs $CD_jU_j$. Telescope its action on
$\beta_j^{\rm f}$ over $j\le i<l$. At each step the displayed
decomposition leaves the fine sum covered by~\eqref{eq:R5fine};
the remaining subcollection is dominated by $\nu_i$ and costs
$CD_jV_i$. For $j<i<l$ these indices are inactive, so
$V_i<U_i$. The same bounds apply to $K_l-K_*$, since this
increment is a scalar multiple of $K_l-K_j$ with multiplier in
$[0,1]$. No additional fine-subcollection estimate is assumed.
Together these bounds give
\begin{equation}\label{eq:meshactivegap}
 \begin{split}
 N_l\le N_j
 &+C\Lambda[(a_j+r_j)C_j+D_ju_j]
   +CA(D_jf_j+t_j)+CD_jF_j(x)\\
 &+CD_j\sum_{j<i<l}[\Lambda u_i+Af_i+F_i(x)].
 \end{split}
\end{equation}
\emph{Finite layers and the sharp limit.}
Put
\[
 \Gamma_{J,L}=\sum_{i=J}^{L-1}
 [(a_i+r_i)C_i+D_i(u_i+f_i)+t_i],\qquad
 \mathcal F_{J,L}=\sum_{i=J}^{L-1}D_iF_i,
\]
with empty sums zero. Cone positivity and the inactive scalar
increments before $k$ give
\[
 N_k\le|p_k|\le CAb_J+C\sum_{i=J}^{k-1}(U_i+Af_i).
\]
For $k=L$ this already proves the following bound, including $J=L$.
Otherwise list the active indices in $[k,L)$ and append $L$.
For each successive pair use~\eqref{eq:meshactivegap}; when
$N_l\le N_j$ the same upper bound is automatic. The intervals
$[j,l)$ are disjoint, $D_j\le D_i$ for $i\ge j$, and $A\le C\Lambda$.
Each index is charged once, giving
\begin{equation}\label{eq:R9finitecomparison}
 N_L\le CAb_J+C\Lambda\Gamma_{J,L}+C\mathcal F_{J,L}(x).
\end{equation}
Its constant is independent of $J,L$, all selections and the number
of error pieces.

Now fix $x\notin\bigcup_{i\ge J}E_i$ and let $L\to\infty$.
The terminal choices may depend on $L$, but each is a probability
average over $|t-r|\le\delta_L$. Continuity of
$t\mapsto\int H_t\,dh$ for smooth input implies
$P_L\to\int H_r\,dh$, regardless of those choices.
With $N_\infty=(-e\cdot\int H_r\,dh)_+$ and
$\mathcal F=\sum_iD_iF_i$, this gives
\begin{equation}\label{eq:meshpointwise}
 N_\infty\le CAb_J+C\Lambda\Gamma_J+C\mathcal F(x),\qquad
 \Gamma_J=\sum_{i\ge J}
 \bigl[(a_i+r_i)C_i+D_i(u_i+f_i)+t_i\bigr].
\end{equation}
Here
$\|\mathcal F\|_{L^1(Q)}\le CA\sum_iD_if_i<\infty$.
The exceptional sets and $\mathcal F$ are independent of $r$;
all preceding estimates are uniform in the admissible radius.
Thus~\eqref{eq:meshpointwise} may be supremized over that fiber.
Since
$\nabla u_r=(|K|r^d)^{-1}\int H_r\,dh$ and $1\le r\le2$,
the fixed normalization factors are absorbed in $C$.
The resulting output supremum $Y$ is measurable: the averages and
their gradients are continuous, the radius fibers have closed graphs
in the compact interval $[1,2]$, and their nonnegative fiber maxima
are upper semicontinuous, with value zero on an empty fiber. No
measurable selection of kernels, rows or anchors is used.

Since $b_j$ increases and $b_j\ge1/w(\delta_j)\to\infty$,
choose $J_*$ with $C\Gamma_{J_*}<1/4$ and $b_{J_*}\ge1$.
Fix $K_0\ge\max\{4C,L_0\}$ and partition the upper tail into
$[K_0Ab_j,K_0Ab_{j+1})$, $j\ge J_*$; repeated endpoints give empty
intervals. On the $j$th interval, $CAb_j\le\Lambda/4$, so
\eqref{eq:meshpointwise} and the exceptional-set bound imply
\[
 |\{Y>\Lambda\}|\le|\{\mathcal F>c\Lambda\}|
             +C\frac A\Lambda\sum_{i\ge j}e_i.
\]
Because $b_j\uparrow\infty$, these intervals cover the tail.
Its exceptional contribution is bounded by $C$ times the following
quantity, evaluated by Tonelli and logarithmic telescoping:
\[
 A\sum_{j\ge J_*}\log\frac{b_{j+1}}{b_j}\sum_{i\ge j}e_i
 =A\sum_{i\ge J_*}e_i\log\frac{b_{i+1}}{b_{J_*}}
 \le A\sum_i e_i\log(e+b_{i+1}).
\]
The other contributions are at most $c^{-1}\|\mathcal F\|_1$ and
$K_0Ab_{J_*}|Q|$. This proves the directional $L^1$ bound.
For $S=0$, use both coordinate signs on $\mathcal J(x)$ to obtain
the pure-error input, then apply Theorem~\ref{r3:master}.

\end{proof}

\subsection{Exact Euler cancellation as a corollary}
For Euclidean caps $B(z,\rho_j)$, suppose the canonical remainder and
its bounded extension obey
\[
 |W_z|\le C\Omega(\rho_j)\text{ on the cap},\qquad
 |\widetilde W_z(y)-\widetilde W_z(y')|\le C\Omega(|y-y'|).
\]
The modulus includes the fixed support-radius constants of error pieces.
For a $C^{1,\vartheta}$ gauge take $\Omega(s)=Cs^\vartheta$.

\begin{corollary}\label{thm:meshpricechain}
Assume the exact balanced Euler identities, cap and increment
modulus $\Omega$, and concentration budget $G_j$. Keep
$w,\delta_j,h_j,u_j,b_j,H_j,C_j,e_j$ from
Theorem~\ref{thm:R5rowchain}, with $\delta_j\le1/4$, choose
$0<\rho_j\le1$, and set
\[
 f_j=\sup_{0<s<h_j}
 \frac{\min(1,s/\delta_{j+1})+\Omega(s)}{w(s)}.
\]
If
\begin{equation}\label{eq:meshsummability}
 \sum_j\{\Omega(\rho_j)C_j+u_j+f_j\}<\infty,\qquad
 \sum_je_j\log(e+b_{j+1})<\infty,
\end{equation}
then both local inputs of Theorem~\ref{r3:master} hold with
$A=\sum_\ell w(s_\ell)|\beta_\ell|(\R^d)$, and hence its signed-band
and variation conclusions. If concentration is known only for
$S=0$, this gives the pure-error input alone. Constants retain the
common geometric data, meshes and displayed series, not the number
of error pieces or layers.
\end{corollary}
\begin{proof}
Apply Theorem~\ref{thm:R5rowchain} with
\[
 B_j=e\cdot h_a^{-1}B_a,\quad D_j=C,\quad
 a_j=C\Omega(\rho_j),\quad r_j=t_j=0,\quad F_j=0.
\]
The balanced moment identity gives zero correction. Constant
subtraction with the displayed product modulus verifies the fine
hypotheses. The scalar increment estimates use only kernel Lipschitz bounds and
cancellation, so they also hold for the initial unbalanced smoothings.
The two series are exactly \eqref{eq:meshsummability}.
\end{proof}

\subsection{Power contact and compact profiles as verifications}
One parameter choice covers the three power-contact regimes.
\begin{lemma}\label{lem:parameters}
For $0<s\le1$, $t\ge m-1$ and
$0<\alpha<s\vartheta/[t+(2+s)\vartheta]$, choose
\[
 \frac1{1-\alpha}<\tau<
 \frac{s\vartheta}{\alpha(t+2\vartheta)},\quad b=\tau\alpha,
 \quad f=\min\{\tau(1-\alpha)-1,\tau(\vartheta-\alpha)\}.
\]
There are $\varepsilon,\gamma>0$ with
\begin{equation}\label{r3:holderparameters}
 f>\varepsilon,\qquad\vartheta\gamma-b>\varepsilon,\qquad
 a=s-t\gamma-2b-2\varepsilon>0,\qquad t\gamma<1.
\end{equation}
\end{lemma}
\begin{proof}
The stated $\alpha$ range is precisely nonemptiness of the $\tau$
interval and implies $\alpha<\vartheta$. Hence $f>0$ and
$d_0=s-(t/\vartheta+2)b>0$. Set
$\varepsilon=\min\{f/2,d_0/[4(t/\vartheta+1)]\}$ and
$\gamma=(b+2\varepsilon)/\vartheta$. Then
$a\ge d_0/2$, $\vartheta\gamma-b=2\varepsilon$, and
$t\gamma<s\le1$.
\end{proof}

\begin{theorem}\label{r3:localinputs}\label{thm:main}
\label{glob:localtriple}\label{r10:profileinputs}
For a compact probability profile, the directional local input holds
with $w(s)=s^\alpha$ for
$0<\alpha<\vartheta/[m(m+1)+(2m+1)\vartheta]$.
For a single body its range is
$0<\alpha<\vartheta/(m+3\vartheta)$.
For every profile, the pure-error band-size input holds for
$0<\alpha<\vartheta/(m-1+3\vartheta)$.
The constants retain the fixed body and compact-window data.
\end{theorem}
\begin{proof}
Take $\Omega(s)=Cs^\vartheta$, $w(s)=s^\alpha$ and
\[
 \delta_j=2^{-j},\quad h_j=\delta_j^\tau,\quad
 \rho_j=\delta_j^\gamma,\quad u_j=\delta_j^\varepsilon,
\]
with parameters from Lemma~\ref{lem:parameters}. Discard a finite
prefix for $h_j\le\delta_{j+1}$ and fixed smallness conditions.
Theorem~\ref{r12:concentrationmaster}
gives the following data, using its source-free case for $S=0$:
\[
\begin{array}{c|c|c|c}
\text{regime}&G_j&(s,t)&\text{width}\\ \hline
\pi=\delta_1&C\delta_j\rho_j^{-m}&(1,m)&
                    \delta_j\le c\rho_j^m\\
\text{compact }\pi&C(\delta_j\rho_j^{-m}+
             \delta_j^{1/m}\rho_j^{-(m+1)})&(1/m,m+1)&
                    \delta_j\le c\rho_j^{2m}\\
S=0&C\delta_j\rho_j^{-(m-1)}&(1,m-1)&\text{none}.
\end{array}
\]
Thus $G_j\le C\delta_j^{s-t\gamma}$,
$C_j\le C\delta_j^{-b}$ and $f_j\le C\delta_j^f$,
where $b=\tau\alpha$ and $f$ is the lemma's fine exponent.
Its inequalities give both mesh series, since
\[
 e_j\le C(\delta_j^a+\delta_j^{1-b-\varepsilon}),\qquad
 \log(e+b_{j+1})\le C(j+1),\qquad
 a=s-t\gamma-2b-2\varepsilon>0.
\]
The other anomaly power is positive because $s\le1$ and $a>0$.
For the widths, $m\gamma<1$ in the single-body case; in the
profile case $a>0$ gives $\gamma<1/[m(m+1)]$, hence
$2m\gamma<1$. Apply Corollary~\ref{thm:meshpricechain}.
\end{proof}

\subsection{Fine errors integrated in the center}
\label{sec:R4integrated}

An integrated fine estimate removes the separate Dini requirement
in Corollary~\ref{thm:meshpricechain}; its coarse cap and concentration
assumptions remain.

\begin{lemma}\label{lem:R4BVtranslation}\label{lem:R4fineenvelope}
Suppose $B_\rho\subset K\subset B_R$ and $q=\nabla(p_K^2/2)$
almost everywhere. Then $|Dq|(B_T)\le C_dT^d\rho^{-2}$.
Choose a fixed smooth compact cutoff $\chi=1$ on the argument
annulus, put $Q_0=\chi q$, $Q_1=\chi w$, and for pivots in a
fixed positive gauge annulus set
\[
 \widetilde W_a(w)=I-p(a)^{-2}\{q(a)\otimes Q_1(w)
       +(a\cdot Q_0(w))I-Q_0(w)\otimes a\}.
\]
At nonsmooth pivots choose any subgradient for $q(a)$.
Uniformly in $a$, its increments are bounded by
$C\sum_{k=0}^1|Q_k(w+v)-Q_k(w)|$, whose integral is at most
$C|v|$. If a zero-mass vector measure $\beta_\ell$ is supported
in $B(z_\ell,C_ss_\ell)$, with mass $m_\ell$, and
$|\Delta(x,r,y)|\le C$, $\Lip_y\Delta\le C/\delta$, then a
measurable $F_\ell$ independent of $a,r$ satisfies
\[
 \left|\int\widetilde W_a(y-x)\Delta(x,r,y)d\beta_\ell(y)\right|
 \le F_\ell(x),\qquad
 \int_QF_\ell\le Cs_\ell(1+\delta^{-1})m_\ell.
\]
Constants retain only $d,\rho,R$, the annuli and cutoff.
\end{lemma}
\begin{proof}
The Hessian $Dq$ is a positive matrix measure; its variation is
bounded by its trace. Test that trace with a cutoff of $B_T$ in
$B_{2T}$ and use $|q(w)|\le|w|/\rho^2$. Thus $Q_0,Q_1$ are
BV, and bounded pivot coefficients and their BV translation
inequality give the common increment bound. Subtract the product
at $z_\ell$ using zero mass. A common majorant is
\[
 C m_\ell\min(1,C_ss_\ell/\delta)
 +C\int\sum_{k=0}^1|Q_k(y-x)-Q_k(z_\ell-x)|d|\beta_\ell|(y).
\]
Integrate the first term on the unit cube $Q$, and the others on
all space by Tonelli and the increment bound. No restriction on
the center $z_\ell$ or on the $x,r$ dependence of $\Delta$ is used.
\end{proof}

\begin{corollary}\label{thm:R4integratedmesh}
Assume the balanced identities, coarse cap estimates and admissible
concentration budgets of Corollary~\ref{thm:meshpricechain} for a convex
gauge. In that theorem replace $f_j$ by
\[
 g_j=\sup_{0<s<h_j}\frac{\min\{1,s/\delta_{j+1}\}}{w(s)}.
\]
If the two summability conditions hold with this replacement, both
local inputs and their endpoint consequences remain valid.
No continuity modulus for fine Euler increments is required.
For a nonsmooth gauge the assumed coarse cap estimates must separately
be verified for the chosen almost-everywhere Euler fields.
\end{corollary}
\begin{proof}
Use Theorem~\ref{thm:R5rowchain} with the exact Euler rows,
$D_j=C$, $a_j=C\Omega(\rho_j)$, $r_j=t_j=0$ and $f_j=g_j$.
Unweighted fine increments have the pointwise bound $CAg_j$.
Lemma~\ref{lem:R4fineenvelope}, summed over $s_\ell<h_j\le\delta_{j+1}$,
gives the common $F_j$ with $\|F_j\|_1\le CAg_j$.
This verifies both fine hypotheses of the unified theorem and proves
the claim without another active-interval or tail argument.
\end{proof}

\section{Raw convex bodies with H\"older continuous normals}
\label{sec:R4holder}
Fix a bounded open $C^{1,\theta}$ convex body $K\ni0$,
$0<\theta\le1$, with gauge $p$ and $q=\nabla(p^2/2)$.
Radial transversality gives $p\in C^{1,\theta}$ on positive annuli;
homogeneity gives $q\in C^{0,\theta}$ on bounded sets.
Constants retain $K,\theta$, without a curvature lower bound.
A \emph{calibrated row} is a continuous vector field $X$ on
$\overline K$ for which a scalar $\lambda$ satisfies
\begin{equation}\label{R4H:eq:calibration}
 \int_K X\cdot\nabla g=\lambda\int_Kw\cdot\nabla g
 \quad(g\in C^\infty(\overline K)).
\end{equation}
The condition is exact. Our rows have the form $X=\lambda w+Y$,
with $Y$ divergence free and tangent to the boundary, constructed
by a $C^1$ change of variables.

\begin{lemma}\label{R4H:lem:collar}
There are fixed positive constants $a_*,b_*$ such that the following
holds uniformly in $u\in\partial K$. Rotate so $n=\nu(u)=e_d$,
and use the global upper graph $w=(s,f(s))$, $s\in D_n$.
The boundary portion $\nu(w)\cdot n\ge3/4$ projects at distance
at least $b_*$ from $\partial D_n$, and a vertical collar of depth
$a_*$ below that portion lies in $K$. On a fixed neighborhood
of this portion the graph has uniform $C^{1,\theta}$ norm, and
$Q(s)=q(s,f(s))$ has a uniform $\theta$-H\"older constant.
\end{lemma}
\begin{proof}
A projection-rim point has a horizontal supporting normal, which
normal uniqueness excludes on $\nu(w)\cdot n\ge3/4$.
Compactness of these pairs $(n,w)$ gives a uniform projection
margin and inward depth. Transverse graph charts on the larger
portion $\nu\cdot n\ge1/2$ give uniform $C^{1,\theta}$ bounds
on a neighborhood. Bounded slopes and composition with $q$ give
the asserted modulus of $Q$. No curvature lower bound is used.
\end{proof}

\begin{lemma}\label{R4H:lem:rows}
For all sufficiently small $\varepsilon>0$, every
$u\in\partial K$ and every unit test vector $e$ admit a calibrated
row $X=X_{u,e,\varepsilon}$ with a bounded global
$C^{0,\theta}$ extension satisfying
\[
 \|X\|_{C^{0,\theta}}\le D_\varepsilon,
 \qquad D_\varepsilon\le C\varepsilon^{-N},\quad N=2/\theta,
\]
and
\[
 |X(v)-e|\le C\varepsilon
 \quad(v\in\partial K,\ |q(v)-q(u)|\le\varepsilon).
\]
\end{lemma}
\begin{proof}
Work in the preceding global graph with $u_d=c>0$.
Write the test vector $e=(e',e_d^{\rm test})$ and set
$\lambda=e\cdot q(u)=e_d^{\rm test}/c$.
Let $L\ge1$ bound the H\"older modulus of $Q$, set
$\tau=(\varepsilon/(4L))^{1/\theta}$, and define
\[
 E_\varepsilon=\{s\in D_n:|Q(s)-q(u)|\le2\varepsilon\},
 \qquad \eta=\1_{E_\varepsilon}*\zeta_\tau,
\]
where $\zeta_\tau$ is a nonnegative unit mollifier supported in
$B_\tau$, and the indicator is extended by zero. The collar lemma
gives a fixed interior margin. H\"older continuity makes $\eta=1$
on $|Q-q(u)|\le\varepsilon$ and confines its support to
$|Q-q(u)|<3\varepsilon$ for small $\varepsilon$; moreover
$\|\eta\|_{C^2}\le C\tau^{-2}=C\varepsilon^{-2/\theta}$.

Choose a fixed smooth $J$ on $[-a_*,0]$, equal to $z$ near zero and zero near
$-a_*$. Put
\[
 b'(s)=\eta(s)(e'-\lambda s),\qquad
 V'(s,z)=J'(z)b'(s),\qquad
 V_d(s,z)=-J(z)\operatorname{div}_s b'(s).
\]
The field $V$ is divergence free, tangent on $z=0$, and zero
near the sides and bottom. Push it through the determinant-one map
$\Psi(s,z)=(s,f(s)+z)$:
\[
 Y(\Psi(s,z))=D\Psi(s,z)V(s,z)
             =(V',\nabla f(s)\cdot V'+V_d).
\]
Extend $Y$ by zero outside the collar in $K$. For every smooth $g$,
a $C^1$ change of variables gives the exact identity
\[
 \int_KY\cdot\nabla g
 =\int V\cdot\nabla(g\circ\Psi)=0.
\]
The boundary term vanishes, proving calibration of $X=\lambda w+Y$
without differentiating $\nabla f$. The bounds for $D\Psi$ and
$b'$ give the claimed H\"older norm. Componentwise McShane extension
in the metric $|x-y|^\theta$, followed by a fixed cutoff, gives
the bounded global extension with the same bound.
On the smaller cap $\eta=1$. Put $H=f-s\cdot\nabla f$; then
$Q=(-\nabla f,1)/H$ and direct substitution gives the exact identity
\[
 X(s,f(s))-e=H(s)\,[\lambda-e\cdot Q(s)]e_d.
\]
Since $\lambda=e\cdot q(u)$ and $H$ is uniformly bounded,
$|X-e|\le C|Q-q(u)|\le C\varepsilon$. This proves the full cap bound.
\end{proof}

\begin{lemma}\label{R4H:lem:comm}\label{R4H:prop:price}
Let $X$ be calibrated with global $C^{0,\theta}$ norm $D\ge1$.
For $r\in[1,2]$ and a probability density $k$ supported on
$[r-\delta,r+\delta]$, $\|k\|_\infty\le C_k/\delta$, set
\[
 \begin{aligned}
 K_k&=\int k(t)\1_{tK}\dd t,\quad
 Z_k(w)=\int k(t)\1_{tK}(w)tX(w/t)\dd t,\quad
 \\
 B_r(w)&=X(w/r),\quad R_k=K_kB_r-r^{-1}Z_k.
 \end{aligned}
\]
For $h=\nabla g\,dw$ and $M_t(h)=\int_{tK}w\cdot\dd h$,
\[
 \int Z_k\cdot\dd h=\lambda\int k(t)M_t(h)\dd t,
\]
\begin{equation}\label{R4H:eq:Rbounds}
 |R_k|\le CD\delta^\theta K_k,\qquad
 \|R_k\|_{C^{0,\theta}}\le CD.
\end{equation}
For balanced $k_0,k_1$ this gives
\begin{equation}\label{R4H:eq:comparison}
 \int B_r(K_{k_1}-K_{k_0})\cdot\dd h
       =\int(R_{k_1}-R_{k_0})\cdot\dd h.
\end{equation}
If $w(s)\ge s^\alpha$, $0<\alpha<\theta$, and
$A=\sum_\ell w(s_\ell)|\beta_\ell|$ is the error price of Section~\ref{r3:mastersection},
then
\[
 \left|\int R_k\cdot\dd\beta\right|
 \le CDA\delta^{\theta-\alpha},\qquad
 \left|\int R_k\cdot\dd h_S\right|
 \le CD\delta^\theta\int K_k\dd S.
\]
Here $h_S=\sum_iv_i\sigma_i$, $|v_i|=1$, $S=\sum_i\sigma_i$.
These bounds persist under convex combinations of arbitrary widths,
using each component's own width; the sharp balanced kernel has $R=0$.
\end{lemma}
\begin{proof}
Scaling calibration gives
$\int_{tK}tX(w/t)\cdot dh=\lambda M_t(h)$.
For $U_t(w)=X(w/r)-(t/r)X(w/t)$, one has
$\|U_t\|_\infty\le CD\delta^\theta$ and
$[U_t]_{C^{0,\theta}}\le CD$ on the fixed support. In
\[
 R_k(w)=\int_{[r-\delta,r+\delta]\cap(p(w),\infty)}k(t)U_t(w)\dd t,
\]
this gives the weighted amplitude bound. Common and differing
integration intervals give, for $s=|w-w'|$,
\[
 |R_k(w)-R_k(w')|
 \le CD\{s^\theta+\min(\delta^\theta,\delta^{\theta-1}s)\}
 \le CDs^\theta.
\]
Thus~\eqref{R4H:eq:Rbounds} holds across the support boundary;
balance gives~\eqref{R4H:eq:comparison}. A zero-mass size-$s$
piece costs $CD\min(\delta^\theta,s^\theta)|\beta_\ell|$.
Use $\min(\delta^\theta,s^\theta)\le\delta^{\theta-\alpha}w(s)$;
size-one pieces use amplitude, as does the source estimate.
Linearity preserves each width's bound in convex combinations.
\end{proof}

\paragraph{Shell rows.}
If $|p(a)-r|,|p(w)-r|\le\delta$ and
$|q(w)-q(a)|\le\varepsilon$, normalize $u=a/p(a)$ and
$v=w/p(w)$. Homogeneity gives
$|q(v)-q(u)|\le C(\varepsilon+\delta)$.
Choose the cap row at radius $C(\varepsilon+\delta)$.
For $\delta\le c\varepsilon$ the central row satisfies
\begin{equation}\label{R4H:eq:shellrow}
 |B_r(w)-e|\le C\varepsilon+CD\delta^\theta,
 \qquad \|B_r\|_{C^{0,\theta}}\le CD,
 \qquad D\le C\varepsilon^{-2/\theta}.
\end{equation}
The $D\delta^\theta$ shell-normalization cost is retained.

\begin{corollary}\label{thm:R4holderendpoint}\label{cor:R5calibrated}\label{lem:R5powers}
For every fixed $C^{1,\theta}$ convex body containing zero,
$0<\theta\le1$, the raw kernel satisfies the two local inputs of
Theorem~\ref{r3:master} with $w(s)=s^\alpha$,
$\alpha=\theta^3/1600$. No support-gap lower bound is required.
Its raw signed-band and regularity conclusions follow from
Theorem~\ref{r3:master} and the kernel transfer, with constants
retaining the fixed normal and collar data.
\end{corollary}
\begin{proof}
Set
\[
 \alpha=\theta^3/1600,\quad \delta_j=2^{-j},\quad
 h_j=\delta_j^{2/\theta},\quad
 \varepsilon_j=\delta_j^{\theta^2/200},\quad u_j=\delta_j^{\theta/50}.
\]
Proposition~\ref{prop:C1gradientcaps} gives
$G_j\le C\delta_j\varepsilon_j^{-(1+1/\theta)}$ for
$\delta_j\le c\varepsilon_j^{1+2/\theta}$.
Lemmas~\ref{R4H:lem:rows}--\ref{R4H:lem:comm} and
\eqref{R4H:eq:shellrow} give row data for
Theorem~\ref{thm:R5rowchain}, with monotone $D_j$:
\[
 D_j=C\varepsilon_j^{-2/\theta},\quad
 a_j=C(\varepsilon_j+D_j\delta_j^\theta),\quad
 r_j=CD_j\delta_j^\theta,\quad
 t_j=CD_j\delta_j^{\theta-\alpha},\quad F_j=0.
\]
The product estimate gives
\[
 f_j=\sup_{0<s<h_j}\frac{\min(1,s/\delta_{j+1})+s^\theta}{s^\alpha}
 \le C[\delta_j^{2/\theta-1-\theta^2/800}
              +\delta_j^{2-\theta^2/800}],
\]
and scalar increments use only the kernel Lipschitz bound, including
the initial unbalanced smoothings. Since
$C_j\le C\delta_j^{-\theta^2/800}$,
\[
 (a_j+r_j)C_j+D_j(u_j+f_j)+t_j
 \le C[\delta_j^{3\theta^2/800}+\delta_j^{\theta/100}].
\]
The cap exponent is $3\theta^2/800$; all others are at least
$\theta/100$. The two anomaly exponents are
$1-(3\theta^2+18\theta)/400\ge379/400$ and
$1-\theta^2/800-\theta/50>0$.
Together with $\log(e+b_{j+1})\le C(1+j)$ these prove summability.
Finally
\[
 \delta_j/\varepsilon_j^{1+2/\theta}
 =\delta_j^{1-(\theta^2+2\theta)/200}\longrightarrow0,
\]
so all widths hold after a finite prefix. Apply the comparison,
master and transfer theorems.
\end{proof}

\begin{remark}\label{rem:R4scope}
Positive H\"older normals suffice for raw bodies even at infinite
contact order or along flat pieces. Additional contact data still
serve quantitative estimates and compact profiles. Arbitrary $C^1$
bodies and polytopes remain outside the result: continuous-normal
Piola rows have no summable modulus supplied here, and arbitrary
convex limits require a new uniform budget.
\end{remark}

\section{Contact moduli and logarithmic normals}
\label{sec:R8lognormal}\label{subsec:lognormal}
Quantitative contact permits normal moduli outside every
$C^{1,\theta}$ class. The compact-profile hypothesis remains
stronger than the single-body hypothesis.

\begin{lemma}\label{lem:R6gaugedictionaries}
\begin{enumerate}[label=\textup{(\roman*)}]
\item Let $p$ be a $C^1$ convex gauge, $\Phi=p^2/2$, and suppose
$1-\nabla p(v)\cdot u\ge\eta(|u-v|)$ for close boundary pairs,
where $\eta>0$ is nondecreasing. On each bounded argument ball,
\[
 D_\Phi(x,y)\ge\eta_B(|x-y|),\qquad
 \eta_B(r)=cr^2\eta(c'r),
\]
after extending, decreasing and capping $\eta$ by fixed constants.
\item Let $F\in C^1$ be convex, $F(0)=0$, and let $K=\{F<1\}$
be bounded. If $\nabla F\ne0$ on $\partial K$, its gauge is $C^1$
off zero and
\[
 1-\nabla p(v)\cdot u=\frac{D_F(u,v)}{\nabla F(v)\cdot v},
 \qquad
 \nabla p(x)=\frac{\nabla F(x/p(x))}
 {\nabla F(x/p(x))\cdot(x/p(x))}.
\]
A modulus for $\nabla F$ therefore transfers, up to a linear term,
to $\nabla(p^2/2)$ on fixed positive annuli.
\end{enumerate}
\end{lemma}
\begin{proof}
A positive local support gap excludes boundary segments; compactness
extends it to distant pairs. Use~\eqref{r3:canonicalidentity} with
$x=su,y=tv$ and $h=|x-y|$. Either the radial term controls
$h^2$, or $s,t\ge ch$ and $|u-v|\ge c'h$ on the bounded
argument ball. After capping $\eta$, this gives (i).
For (ii), convexity yields $\nabla F(v)\cdot v\ge1$ on
$\partial K$. Differentiate $F(x/p(x))=1$ by the radial implicit
function theorem. On positive annuli, $x\mapsto x/p(x)$ is
Lipschitz and the denominators stay positive, giving the modulus
claim.
\end{proof}

\begin{theorem}\label{lem:modulusconcentration}\label{lem:profilemodulusconcentration}
Let $p$ be a $C^1$ convex gauge with
$D_{p^2/2}(a,b)\ge\eta(|a-b|)$ on the fixed bounded argument ball,
where $\eta>0$ is nondecreasing. Put $h=c\eta(c'\rho)$ and
$N_\nu=\nu(\R^d)$. The concentration conclusion of
Theorem~\ref{r12:concentrationmaster}, with its radius-independent
Borel set and actual shell anchor, has the following verifiers:
\begin{enumerate}[label=\textup{(\roman*)}]
\item For a single body and $\delta\le c''h$,
\[
 |E|\le C\frac\delta h
 \left[(N_\nu/U)^2+(N_\nu/U)(1+T/U)\right].
\]
For $S=0$ only the quadratic term remains, without a width restriction.
\item For a compact positive profile window, assume additionally
that $\eta$ is continuous, $\eta(0)=0$, and $\eta(r)\le r^2$ near zero.
With $R_\eta(s)=\sup\{r\in[0,r_*]:\eta(r)\le s\}$ and
$\delta\le c''h^2$,
\[
 |E|\le C\left[\frac\delta h
 \{(N_\nu/U)^2+(N_\nu/U)(1+T/U)\}
 +\frac{R_\eta(C\delta)}{\rho h}\frac{N_\nu}U\frac TU\right].
\]
\end{enumerate}
\end{theorem}
\begin{proof}
Repeat the proof of Theorem~\ref{r12:concentrationmaster} with
$h=c\eta(c'\rho)$. On its fixed argument ball, for $|y-z|\ge\rho$,
\[
 \frac{D_\Phi(y-x,z-x)+D_\Phi(z-x,y-x)}{|y-z|}
 \ge c\eta(c'\rho).
\]
Ratio perturbations of size $O(h)$ are absorbed by this derivative
bound; larger gaps use polar monotonicity as in
Lemma~\ref{a:pairstrip-expanded}. Thus the pair strip costs
$C\delta/h$.

For the source sweep, writing $x=z-t\theta$, $p(\theta)=1$, gives
\[
 \frac d{dt}\{t[\Phi(\theta+(y-z)/t)-q^2/2]\}
 =-D_\Phi(\theta,\theta+(y-z)/t)+(1-q^2)/2\le-ch
\]
for $q\ge1-2h$. The fixed-$M$ truncated families in the cited
proof are therefore increasing, with crossing width $C\delta/h$.
Lemma~\ref{lem:R7sweep} on target-scale times source space gives
the near-source term, including atoms. For a single body $q=1$
and this proves (i).

For profiles use reference ratio $1-h$. The width $\delta\le c''h^2$
separates the near sweep from the far entry margin $ch$.
The common far-source truncation has mass $T/W(a)$, and
Lemma~\ref{r11:profiletailenergy} gives
$\int W(a)^{-1}d\sigma(a)\le C/h$.
In a component of $|p(w+t\theta)-qt|\le C\delta$, the normalized
endpoints have Jensen deficit $O(\delta)$, so
$\eta(c|u_1-u_2|)\le C\delta$. The argument of
Lemma~\ref{audit:farbandlength} bounds its length by
$C[R_\eta(C\delta)+\delta]/|w|
 \le CR_\eta(C\delta)/\rho$;
here $\eta(r)\le r^2$ absorbs $\delta$. There are at most two
components. Combining these bounds and the Borel anchor test gives
(ii).
\end{proof}

\begin{theorem}\label{thm:lognormalendpoint}
Let $K$ be bounded, open and convex, contain zero, and have a $C^1$
gauge. On the fixed positive annuli suppose
\begin{equation}\label{eq:lognormalmodulus}
 |\nabla\Phi(x)-\nabla\Phi(y)|
       \le C[\log(e/|x-y|)]^{-\beta},\qquad\beta>1,
\end{equation}
and for close boundary pairs suppose
\begin{equation}\label{eq:lognormalgap}
 1-\nabla p(v)\cdot u\ge
 c\exp\{-C[\log(e/|u-v|)]^\kappa\},\qquad\kappa\ge1.
\end{equation}
If $\beta>\kappa$, the single-body signed-band and full variation
endpoints hold. If $\beta>\kappa^2$, they hold uniformly over all
probability profiles on a fixed compact positive window. The
bounded-kernel regularity and stability conclusions follow with their
separate uniformity and spatial-tail hypotheses. Constants retain the
normal, contact and, for profiles, window data.
\end{theorem}
\begin{proof}
The distributional moments require only $C^1$. Euler's formula and
an annular cutoff give cap and increment modulus
$\Omega(s)=C\log(e/s)^{-\beta}$, absorbing the linear term.
Lemma~\ref{lem:R6gaugedictionaries}(i) yields
\[
 \eta_B(r)=cr^2e^{-C\log(e/r)^\kappa},\qquad
 \log(C/\eta_B(r))\le C'\log(e/r)^\kappa,\qquad
 R_{\eta_B}(s)\le Ce^{-c[\log(1/s)]^{1/\kappa}}.
\]
Choose $P>1$, $\sigma>0$ with $P<\sigma\beta$ and
$\sigma\kappa<1$ for a single body, or
$\sigma\kappa<1/\kappa$ for profiles. These choices are possible
precisely in the stated ranges; they also give $P<\beta$.
For fixed $K_0>2$ set
\[
 w(s)=s^{1/2}+\log(e/s)^{-P},\quad
 \delta_j=e^{-2^j},\quad h_j=e^{-K_02^{j+1}},\quad
 \rho_j=e^{-2^{\sigma j}},\quad u_j=2^{-j}.
\]
Splitting at $s=\delta_j$ gives $b_j,H_j,C_j\le C2^{Pj}$;
below that threshold use
$e^{2^j-t}(1+t)^P\le C(1+2^j)^P$. Hence
\[
 \Omega(\rho_j)C_j\le C2^{-(\sigma\beta-P)j},\qquad
 f_j\le Ce^{-(K_0/2-1)2^{j+1}}+C2^{-(\beta-P)j}.
\]
The two fine terms use respectively the power floor and
$\Omega(s)/w(s)\le C\log(e/s)^{-(\beta-P)}$.
Concentration gives, for a single body,
$G_j\le C e^{-2^j+C2^{\sigma\kappa j}}$, and for profiles,
\[
 G_j\le C e^{-2^j+C2^{\sigma\kappa j}}
       +C e^{-c2^{j/\kappa}+C2^{\sigma\kappa j}+C2^{\sigma j}}.
\]
The parameter inequalities make these faster than geometric and
give all width restrictions after finitely many initial scales,
including $\delta_j\le c\eta_B(c'\rho_j)^2$ for profiles.
Both mesh series converge, since
$\log(e+b_{j+1})=O(j+1)$. Apply
Corollary~\ref{thm:meshpricechain} and the master and transfer
theorems. The hypotheses are sufficient and do not cover arbitrary
$C^1$ bodies or flat faces.
\end{proof}

\begin{example}\label{ex:nonHoldernormal}
For every $d\ge2$ and $\beta>1$, a bounded strictly convex $C^1$
body with quadratic support gap satisfies
\eqref{eq:lognormalmodulus} but has no $C^{1,\vartheta}$ boundary,
$\vartheta>0$. It has all the single-body and compact-profile
endpoints above with $\kappa=1$.
\end{example}
\begin{proof}
Write $L(t)=\log(e/t)$ and choose $\epsilon$ with
$L(2\epsilon)>2(\beta+1)$. Let $\chi$ be even, smooth and
nonincreasing on $[0,\infty)$, equal to one on $[-\epsilon,\epsilon]$
and supported in $(-2\epsilon,2\epsilon)$. Set
\[
 b(t)=2+\chi(t)\frac{\beta}{|t|L(|t|)^{\beta+1}},\quad
 \psi'(t)=\int_0^t b(s)\,ds,\quad \psi(t)=\int_0^t\psi'(s)\,ds,
\]
with the singular term zero off the cutoff support.
Then $\psi$ is even, $C^1$, quadratic at infinity, and $\psi''\ge2$.
For $0<t\le\epsilon$,
\[
 \psi'(t)=2t+L(t)^{-\beta},\qquad
 t^2+ctL(t)^{-\beta}\le\psi(t)\le t^2+tL(t)^{-\beta}.
\]
The decreasing singular curvature has mass at most
$CL(r)^{-\beta}$ on every length-$r$ interval, giving this modulus
for $\psi'$. For $F(x',t)=|x'|^2+\psi(t)$ and $K=\{F<1\}$,
$D_F(u,v)\ge|u-v|^2$ and $\nabla F\ne0$ on $\partial K$;
Lemma~\ref{lem:R6gaugedictionaries}(ii) gives the gap and modulus.
For $v=(e_1,0)$ and $u_t=(\sqrt{1-\psi(t)}e_1,t)$,
$|u_t-v|\asymp t$, but the last components of the gauge gradient
and unit normal differ by at least $cL(t)^{-\beta}$.
This exceeds every $Ct^\vartheta$ near zero.
\end{proof}

\begin{corollary}\label{cor:logcontactprofile}
If the gauge is $C^{1,\vartheta}$ and
$1-\nabla p(v)\cdot u\ge ce^{-C\log(e/|u-v|)^\kappa}$,
$1\le\kappa<\infty$, then every compact positive probability scale
profile has the signed-band and variation endpoints, uniformly over
the profile law, with the applicable transfer conclusions.
\end{corollary}
\begin{proof}
Use $s^\vartheta\le C_\beta\log(e/s)^{-\beta}$ with any fixed
$\beta>\kappa^2$ in Theorem~\ref{thm:lognormalendpoint}.
\end{proof}

\section{Models and geometric corollaries}\label{r3:examples}
\begin{corollary}\label{r5:finiteheadline}\label{r5:analytic-body}
\label{r3:smoothheadline}\label{r4:ball}\label{cor:power}\label{r3:affineheadline}
Theorem~\ref{r3:bodyendpoint} applies to
$K_s=\{\sum|z_i|^s<1\}$, $1<s<\infty$, and every smooth bounded
convex body containing zero, including analytic bodies and bodies
with flat pieces. Its class is preserved by invertible linear maps
and translations retaining zero inside. For invertible $B$,
\[
 \int|B^T\nabla M_{BK}f|\le C_K\int|B^T\nabla f|,\qquad
 C_{BK}\le\|B\|\|B^{-1}\|C_K.
\]
For finite-volume, finite-perimeter $E$,
$\int_0^1P(\{M_K\1_E>t\})\dd t\le C_KP(E)$.
\end{corollary}
\begin{proof}
Power balls have boundary exponent $\min(1,s-1)$; smooth compact
boundaries have exponent one. Both regularity classes are affine
invariant. If $g(x)=f(Bx)$, then
$M_{BK}f(Bx)=M_Kg(x)$; the weak chain rule and cancellation of the
Jacobian give the displayed inequalities. Coarea gives the last claim.
\end{proof}

\subsection{Additional contact data for probability profiles}
These are additional profile data; the raw theorem needs no gap.
For the Euclidean ball, $q(w)=w$ and the gap is $|u-v|^2/2$.
For $K_s$ put $G_s=s^{-1}\sum|w_i|^s$; on its boundary
$\nabla G_s=\nabla p_s$. If $s\ge2$, Taylor's integral formula gives
\[
 D_{|\cdot|^s/s}(b+t,b)
 =(s-1)|t|^s\int_0^1(1-v)|b/t+v|^{s-2}\dd v
 \ge(s-1)8^{1-s}|t|^s.
\]
Indeed on a subset of $[0,1/2]$ of length at least $1/4$ the
absolute value is at least $1/8$. If $1<s<2$, on $[-L,L]$ the
locally absolutely continuous first derivative instead gives
$D_{|\cdot|^s/s}(a,b)\ge(s-1)L^{s-2}|a-b|^2/2$.
Summing coordinates yields the profile data
$m=\max(2,s)$ and $\vartheta=\min(1,s-1)$.
The identity $p_{BK}=p_K\circ B^{-1}$ transfers these data to
linear images.

\begin{lemma}\label{r3:translations}
If $L$ has a power support gap of order $m$ and a
$C^{1,\vartheta}$ gauge, then $K=L+a$ has the same properties
whenever $0\in K$.
\end{lemma}
\begin{proof}
Write $u=a+u_0,v=a+v_0$ and $n_0=\nabla p_L(v_0)$.
Compactness and $0\in K$ give $1+n_0\cdot a\ge c>0$, and
\[
 \nabla p_K(v)=\frac{n_0}{1+n_0\cdot a},\qquad
 1-\nabla p_K(v)\cdot u
 =\frac{1-n_0\cdot u_0}{1+n_0\cdot a}.
\]
The gap follows. Translation preserves the $C^{1,\vartheta}$
boundary, and the positive denominator gives radial transversality,
hence the gauge regularity on positive annuli.
\end{proof}

\subsection{Coercive mixtures of potentials}
\begin{corollary}\label{r5:potentialmixtures}
Let $\Phi_j$, $1\le j\le N$, be nonnegative convex degree-two
homogeneous $C^1$ functions whose gradients, for a fixed $0<\vartheta\le1$, are locally
$\vartheta$-H\"older off zero. If $a_j>0$ and
$\Phi=\sum a_j\Phi_j$ satisfies $\min_{|w|=1}\Phi(w)>0$, then
$K=\{2\Phi<1\}$ has all raw-body conclusions of
Theorem~\ref{r3:bodyendpoint}. A Bregman lower bound of order $m$
for one component gives the same quantitative contact order for $K$.
In particular $\Phi=\varepsilon|w|^2/2+\Psi$, $\varepsilon>0$,
has contact order two.
\end{corollary}
\begin{proof}
Coercivity bounds $K$, and $p=\sqrt{2\Phi}\in C^{1,\vartheta}$
off zero. Euler's identity gives $\nabla\Phi(w)\cdot w=1$ on
$\partial K$, so the raw theorem applies. Since Bregman divergences
add and are nonnegative, $D_\Phi\ge a_jD_{\Phi_j}$;
restriction to the boundary gives the contact assertion.
\end{proof}
For example
$\Phi_\varepsilon=|w|^2/2+\varepsilon(w_1)_+^2/2$ is globally
$C^{1,1}$ and has $D_{\Phi_\varepsilon}(a,b)\ge|a-b|^2/2$.
Its nonsymmetric body joins a ball to an ellipsoid across $w_1=0$;
the Hessian jumps, but both raw and profile conclusions apply.

\subsection{Finite tangent contact as a quantitative corollary}
\begin{lemma}\label{r5:convex-taylor}
Let $A$ be compact and let convex $h_a:[-r,r]\to\R$, $a\in A$,
satisfy $h_a(0)=h_a'(0)=0$. Suppose their derivatives through order
$m\ge2$ depend continuously on $(a,t)$ and
$c_0=\inf_a\sum_{j=2}^m|h_a^{(j)}(0)|/j!>0$.
Then $h_a(t)\ge c|t|^m$ uniformly near zero.
\end{lemma}
\begin{proof}
Write $a_j=h_a^{(j)}(0)/j!$. Uniform Taylor remainder gives
$h_a(t)=\sum_{j=2}^ma_jt^j+O(\omega(|t|)|t|^m)$,
$\omega(s)\to0$. Finite-dimensional polynomial norm equivalence
and $0\le h_a(tu)\le h_a(t)$, $0\le u\le1$, imply for $|t|\le1$
\[
 c_0|t|^m\le\sum_{j=2}^m|a_jt^j|
 \le C_m\{h_a(t)+\omega(|t|)|t|^m\}.
\]
Absorb the remainder. Polynomial norm equivalence follows, for example,
by Lagrange interpolation at $0,1/m,\ldots,1$.
\end{proof}
For the inward normal graph $v+w-h_v(w)n(v)$,
finite tangent-line contact means that each $h_v(te)$, $|e|=1$,
has a nonzero finite-order jet. Equivalently a smooth defining
function has finite vanishing order on $v+te$: its restriction is
$h_v(te)$ times a nonvanishing smooth factor.

\begin{proposition}\label{prop:smooth-finite-type}
A bounded smooth convex body containing zero with finite contact at
every tangent line has a finite even maximum contact order $m$ and
\begin{equation}\label{eq:smooth-support-gap}
 n(v)\cdot(v-u)\ge c_K|u-v|^m\qquad(u,v\in\partial K).
\end{equation}
Its gauge is smooth off zero and verifies Definition~\ref{def:geometry}
with $\vartheta=1$.
\end{proposition}
\begin{proof}
The compact tangent-unit bundle admits a common normal-graph radius
and continuous jets, by a finite implicit-function atlas. A finite
subcover of the nonzero-jet sets gives a maximum first nonzero order
$m$, even by nonnegativity, and a positive uniform jet sum.
Lemma~\ref{r5:convex-taylor} gives $h_v(w)\ge c|w|^m$.
Since $|u-v|\le C|w|$ in these graphs, this proves the near-diagonal
gap. A boundary segment would have infinite tangent contact;
strict convexity and compactness therefore extend the gap to all
pairs. Finally $\nabla p(v)=n(v)/(n(v)\cdot v)$, with positive
denominator, transfers the gap, and radial implicit differentiation
gives smoothness. The argument needs only $C^m$ regularity when
jets through one fixed order $m$ witness every contact.
\end{proof}

\begin{example}\label{ex:R4stadium}
The $C^{1,1}$ stadium $B_1+[-1,1]e_1$ has the raw endpoint but
zero support gap along its upper segment. There $a=e_2$ and
$w=se_1+e_2$ have equal canonical gradients, yet
$W_a(w)=-s e_2\otimes e_1$. Thus gradient-value concentration
requires the additional calibrated-row cancellation.
\end{example}

\section{Singular derivatives of probability-kernel maxima}
\label{r6:jumpsection}\label{r6:cantorsection}
Natural jump traces and the Cantor contact identity require only
local BV output; neither an endpoint estimate nor open detachment is used.

Fix $\kappa\ge0$ with
\begin{equation}\label{r10:conditionalkernelbounds}
 \int\kappa=1,\qquad \|\kappa\|_\infty=B<\infty,
 \qquad\supp\kappa\subset B_R,\qquad R\ge1,
\end{equation}
and write $A_r^\kappa g(x)=\int\kappa(z)g(x+rz)\,dz$.
For a nonnegative input set
\[
 u(x)=M_\kappa g(x)=\sup_{r>0}A_r^\kappa g(x),\qquad
 u_T(x)=\sup_{r\ge T}A_r^\kappa g(x).
\]
These are the natural pointwise suprema. At a scalar BV jump orient
$\nu_g$ from the lower trace $g^\wedge$ to the upper trace $g^\vee$,
using a Borel choice on $J_g$. Put
\begin{equation}\label{r6:halffraction}
 \theta_\kappa(\nu)=\int_{z\cdot\nu>0}\kappa(z)\,dz,
 \qquad \nu\in S^{d-1}.
\end{equation}
For a BV function $v$, write $v^*$ for its approximate value at
approximate continuity points and its mean trace on the jump set.
Assign Borel values on the remaining $\mathcal H^{d-1}$-null set,
which Cantor measures ignore.

\begin{theorem}
\label{r10:conditionalkernelfine}\label{r7:exactjump}
\label{r6:cantoridentity}\label{r6:jumpattenuation}
Let $d\ge2$, let $\kappa$ satisfy~\eqref{r10:conditionalkernelbounds},
and let $g\ge0$ belong to $BV(\mathbb R^d)\cap L^1(\mathbb R^d)$.
Assume, independently, that $u=M_\kappa g\in BV_{\rm loc}$.
For the finite contact set
\[
 C_\kappa(g)=\{u^*=g^*,\quad |u^*|+|g^*|<\infty\},
\]
one has the vector-measure identity
\begin{equation}\label{r10:conditionalfineformula}
 D^su=\nu_g(g^\vee-u)_+\mathcal H^{d-1}\llcorner J_g
                         +(D^cg)\llcorner C_\kappa(g).
\end{equation}
The jump density uses the natural value $u(x)$. More precisely,
\begin{equation}\label{r7:exactjumpmeasure}
 D^ju=\nu_g(g^\vee-u)_+\mathcal H^{d-1}\llcorner J_g,
 \qquad J_u=\{x\in J_g:u(x)<g^\vee(x)\}
                       \quad\bmod\mathcal H^{d-1},
\end{equation}
and
\begin{equation}\label{r6:exactCantor}
 D^cu=(D^cg)\llcorner C_\kappa(g),\qquad
 |D^cu|=|D^cg|\llcorner C_\kappa(g).
\end{equation}
In particular,
\begin{equation}\label{r6:orientedjump}
 |D^ju|\le(1-\theta_\kappa(\nu_g))|D^jg|,
 \qquad |D^su|\le|D^sg|.
\end{equation}
Natural and precise maximal values agree at the relevant singular points:
\begin{equation}\label{r6:representativeagreement}
 u=u^*\ge g^*,\qquad |u^*|+|g^*|<\infty
 \quad\text{at }(|D^cg|+|D^cu|)\text{-almost every point}.
\end{equation}
No signed-band estimate, lower bound near zero, or constant-one singular
estimate is assumed in this theorem.
\end{theorem}

\subsection{Natural traces from three radius ranges}
The density and support bounds imply
\[
 A_r^\kappa F(x)\le Br^{-d}\int_{B(x,Rr)}F,\qquad
 M_\kappa F\le B\omega_dR^dM_{B(0,1)}F\quad(F\ge0),
\]
and hence the ball weak estimate gives
\begin{equation}\label{r6:weakcover}
 |\{M_\kappa F>\lambda\}|
 \le C_{d,B,R}\lambda^{-1}\|F\|_1.
\end{equation}
Translation continuity makes the positive-radius averages continuous
and $u$ lower semicontinuous. Moreover
\begin{equation}\label{r7:largeuniformcontinuity}
 |u_T(y+z)-u_T(y)|\le BT^{-d}\|g(\cdot+z)-g\|_1,
 \qquad u_T\le BT^{-d}\|g\|_1,
\end{equation}
and, for BV inputs,
\begin{equation}\label{r6:largeradiusLip}
 \operatorname{Lip}u_T\le BT^{-d}|Dg|(\R^d).
\end{equation}

\begin{lemma}\label{r7:exacttraces}
Let $g\ge0$ be in $L^1$, and suppose that at $x$ it has the finite
$L^1$ half-space blow-up
\[
 h_x(y)=\begin{cases}
 \beta,&(y-x)\cdot\nu>0,\\
 \alpha,&(y-x)\cdot\nu<0,
 \end{cases}\qquad
 0\le\alpha\le\beta<\infty,\qquad
 \int_{B(x,s)}|g-h_x|=o(s^d).
\]
Then $u(x)<\infty$, and its one-sided approximate limits are
\begin{equation}\label{r7:traceformula}
 \operatorname*{ap\,lim}_{\substack{y\to x\\(y-x)\cdot\nu<0}}u(y)=u(x),
 \qquad
 \operatorname*{ap\,lim}_{\substack{y\to x\\(y-x)\cdot\nu>0}}u(y)
                           =\max\{\beta,u(x)\}.
\end{equation}
In the constant-model case $\alpha=\beta$, $u$ is approximately
continuous at $x$, with approximate value equal to its natural value.
\end{lemma}
\begin{proof}
Put $e=|g-h_x|$. The blow-up and the kernel bound give
\begin{equation}\label{r6:jumpaverage}
 A_r^\kappa g(x)\longrightarrow
 a_0=\alpha+(\beta-\alpha)\theta_\kappa(\nu),\qquad r\downarrow0.
\end{equation}
Together with $A_r^\kappa g(x)\le Br^{-d}\|g\|_1$ away from zero,
this proves $a_0\le u(x)<\infty$.
Fix $\delta>0$ and choose $T$ with
$B\int_{B(x,(R+1)r)}e\le\delta r^d$ for $0<r<T$.
For $y\in B(x,s)$ split the radii as follows:
\begin{equation}\label{r7:halfspaceerror}
 \left|B(x,s)\cap
 \left\{\sup_{0<r<s}A_r^\kappa e>\delta\right\}\right|
 \le C\delta^{-1}\int_{B(x,(R+1)s)}e=o(s^d).
\end{equation}
Here~\eqref{r6:weakcover} is applied to the integrable local
truncation of $e$. For $s\le r<T$, the error is at most $\delta$;
for $r\ge T$,~\eqref{r7:largeuniformcontinuity} bounds the envelope
near $x$ by $u(x)+\delta$.
On the negative side the model averages satisfy
\[
 A_r^\kappa h_x(y)=\alpha+(\beta-\alpha)
 \int_{z\cdot\nu>-(y-x)\cdot\nu/r}\kappa(z)\,dz\le a_0,
\]
and on the positive side they are at most $\beta$.
Thus the one-sided approximate upper limits are at most $u(x)$
and $\max\{\beta,u(x)\}$. Lower semicontinuity gives the lower
bound $u(x)$ on both sides; $u\ge g$ almost everywhere gives the
additional bound $\beta$ on the positive side. This proves the
formula, including its constant-model case.
\end{proof}

\subsection{Scalar excision and relative Lipschitz neighborhoods}
We use \cite[Section~4]{LW}: $\operatorname{Cap}_1$-null sets are
exactly the $\mathcal H^{d-1}$-null sets and are ignored by BV
derivatives. Cantor derivatives also ignore $\sigma$-finite
$\mathcal H^{d-1}$ sets, including jump sets \cite{AFP,LW}.

\begin{lemma}\label{r6:diffuselocality}
If $v,w\in BV_{\rm loc}$ and their finite precise representatives
agree on a Borel set $A$, then
$(D^cv)\llcorner A=(D^cw)\llcorner A$.
\end{lemma}
\begin{proof}
Discard jump sets and representative exceptions, which both Cantor
measures ignore. On the rest of $A$, $h=v-w$ has approximate value
zero, hence lies outside the measure-theoretic boundary of $\{h>t\}$
for every $t\ne0$. Localized coarea gives $|Dh|(A)=0$ there;
linearity of the Cantor part proves the claim.
\end{proof}

\begin{lemma}\label{r6:capacitaryexcision}\label{r6:relativelipschitz}
For every $\varepsilon>0$ there are open $U\subset V$ with
$\operatorname{Cap}_1(V)<\varepsilon$ such that, for
$F=\R^d\setminus U$, the restriction $u|_F$ is continuous at every
$x\in F\setminus J_g$. If $u(x)>g^*(x)$ there, then on a relative
neighborhood
\begin{equation}\label{r6:truncatedagreement}
 u=u_{\delta_x}\quad\text{on }B_x\cap F
 \qquad\text{for some }\delta_x>0.
\end{equation}
At every $x\notin V\cup J_g$, $u$ has approximate value equal to
its finite natural value.
\end{lemma}
\begin{proof}
Lemmas~4.5--4.6 of \cite{LW}, the former applied to the locally BV
function $1$, give a Borel $G\subset U\subset V$ with the prescribed
capacity, such that $G$ contains the nonjump Lebesgue-point exceptions
and $\{M_{B(0,1)}g=\infty\}$, $g^\vee$ is finite and relatively
upper semicontinuous off $G$, and locally uniformly
\begin{equation}\label{r6:badaverage}
 |B(x,r)|^{-1}\int_{B(x,r)}\mathbf1_Gg\longrightarrow0
 \quad(x\notin U),
\end{equation}
\begin{equation}\label{r6:densityexcision}
 |U\cap B(x,r)|/|B(x,r)|\longrightarrow0
 \quad(x\notin V),\qquad r\downarrow0.
\end{equation}
Choose the initial tolerances so $\operatorname{Cap}_1(V)<\varepsilon$.
Ball comparison makes $u$ finite on $F$. Splitting averages over
$G$ and its complement, the kernel bound,
\eqref{r6:badaverage} and relative upper semicontinuity give
\begin{equation}\label{r6:smallradiusupper}
 \lim_{\rho\downarrow0}\sup_{\substack{z\in F\cap B(x,\rho)\\0<r<\rho}}
 A_r^\kappa g(z)\le g^\vee(x)=g^*(x)
 \quad(x\in F\setminus J_g).
\end{equation}
The approximate identity gives $u(x)\ge g^*(x)$.
For $z_j\to x$ in $F$, pass almost maximizing radii to a limit
in $[0,\infty]$. The three upper bounds are $g^*(x),u(x),0$,
by~\eqref{r6:smallradiusupper}, positive-band continuity and $L^1$
decay. Positive-band continuity follows from smooth $L^1$
approximation with error $Ba^{-d}\|g-\widetilde g\|_1$ on $[a,b]$.
Lower semicontinuity proves relative continuity.
If $u(x)>g^*(x)$, a positive-radius average and
\eqref{r6:smallradiusupper} exclude small maximizing radii locally,
giving~\eqref{r6:truncatedagreement}. Density one of $F$ off $V$
then gives the precise representative assertion.
\end{proof}

\subsection{Proof of the conditional theorem}
\begin{proof}[Proof of Theorem~\ref{r10:conditionalkernelfine}]
Lemma~\ref{r7:exacttraces}, applied to the BV half-space and constant
blow-ups, gives traces $u(x),\max\{g^\vee(x),u(x)\}$ on $J_g$ and
excludes new jumps outside $J_g$ up to an $\mathcal H^{d-1}$-null set.
The jump representation proves~\eqref{r7:exactjumpmeasure}, while
\begin{equation}\label{r6:lowerjumptrace}
 u(x)\ge\theta_\kappa(\nu_g)g^\vee+
 (1-\theta_\kappa(\nu_g))g^\wedge
\end{equation}
gives directional domination.

Set $\mu=|D^cg|+|D^cu|$ and choose excision sets with
$\operatorname{Cap}_1(V_j)\to0$. Off $V_j\cup J_g$, the relative
lemma gives finite $u=u^*\ge g^*$ up to $\mu$-null exceptions.
Discard $J_g$ and the countable union of these exceptions below.
On $D\setminus V_j$, where $D=\{u^*>g^*\}$, take a countable cover
by neighborhoods from~\eqref{r6:truncatedagreement}. Density one of
$F=\mathbb R^d\setminus U_j$ off $V_j$ identifies the precise values
of $u$ and $u_{\delta_x}$ on $B_x\setminus V_j$.
Lemma~\ref{r6:diffuselocality} therefore gives
$|D^cu|(B_x\setminus V_j)=0$.
Both representative failures and remaining detached Cantor mass
lie in the capacity-zero set $\bigcap_jV_j$. Thus
\eqref{r6:representativeagreement} holds and
\begin{equation}\label{r6:Cantordetachmentzero}
 |D^cu|(\{u^*>g^*\})=0.
\end{equation}
On the finite contact set, diffuse locality identifies $D^cu$ with
$D^cg$; representative agreement shows these sets exhaust the
relevant Cantor mass, proving~\eqref{r6:exactCantor}.
The mutually singular jump and Cantor parts now give both
\eqref{r10:conditionalfineformula} and the constant-one bound.
\end{proof}

\begin{corollary}\label{r9:localfine}
Let $\kappa$ and $g$ satisfy the kernel and input hypotheses of
Theorem~\ref{r10:conditionalkernelfine}, and let $\Omega\subset\R^d$
be open. If $u=M_\kappa g\in BV_{\rm loc}(\Omega)$, then its
singular-derivative identities and jump bounds hold as measures on
$\Omega$, with the jump and contact sets restricted to $\Omega$.
The representative assertion holds at
$(|D^cg|+|D^cu|)$-almost every point in $\Omega$.
\end{corollary}
\begin{proof}
The trace and excision lemmas are independent of output regularity.
On relatively compact subsets of $\Omega$, choose all relative
neighborhoods inside $\Omega$; local BV suffices for diffuse
locality and the countable-cover proof of
Theorem~\ref{r10:conditionalkernelfine}. Exhaust $\Omega$.
\end{proof}

\subsection{Consequences and sharp coefficients}
\begin{corollary}\label{r6:Cantorcriterion}
Under the conditional theorem's hypotheses,
$D^cM_\kappa g=0$ if and only if $|D^cg|(C_\kappa(g))=0$.
Thus input-Cantor-almost-everywhere detachment gives an
$SBV_{\rm loc}$ output; if also $D^jg=0$, the output is locally
$W^{1,1}$. If $g=|f|$ with $f\in BV\cap L^1$, then
\begin{equation}\label{r6:Cantorone}
 |D^cM_\kappa f|\le|D^c|f||\le|D^cf|,
 \qquad |D^sM_\kappa f|\le|D^s|f||\le|D^sf|,
\end{equation}
provided $M_\kappa f\in BV_{\rm loc}$. A global gradient bound,
when available, is supplied separately by the endpoint theorem.
\end{corollary}
\begin{proof}
Apply the contact identity, jump domination and the BV chain rule.
\end{proof}

By dominated convergence, $\theta_\kappa$ is continuous and
$\theta_\kappa(-\nu)=1-\theta_\kappa(\nu)$, so
\begin{equation}\label{r6:uniformjump}
 \eta_\kappa:=1-\min_\nu\theta_\kappa(\nu)\in[1/2,1],\qquad
 |D^ju|\le\eta_\kappa|D^jg|.
\end{equation}
Even kernels have $\eta_\kappa=1/2$. An inner positive ball or zero
centroid gives $\eta_\kappa<1$: zero half-space mass and zero first
moment would force the density onto a hyperplane. General densities
can have $\eta_\kappa=1$.
\begin{lemma}\label{r9:flatpatchtest}
Let $\kappa$ satisfy~\eqref{r10:conditionalkernelbounds}, fix
$\nu\in S^{d-1}$, and put
\[
 E=\{y\cdot\nu>0\}\cap B_2,\qquad
 P=\nu^\perp\cap B_1,\qquad u=M_\kappa\mathbf1_E.
\]
Then $u\in SBV_{\rm loc}(B_1)$, its natural value on $P$ is
$\theta_\kappa(\nu)$, and
\begin{equation}\label{r6:flatsharpness}
 D^su=(1-\theta_\kappa(\nu))\nu\,
       \mathcal H^{d-1}\llcorner P
       \quad\text{as measures on }B_1.
\end{equation}
In particular, for every Borel $A\subset P$,
$|D^ju|(A)=(1-\theta_\kappa(\nu))|D\mathbf1_E|(A)$.
\end{lemma}
\begin{proof}
Put $a=(2R)^{-1}$ and
$v=\sup_{r\ge a}A_r^\kappa\mathbf1_E$. By the translation estimate,
\[
 \operatorname{Lip}v\le Ba^{-d}|D\mathbf1_E|(\mathbb R^d),
\]
without an endpoint assumption. For $x\in B_1$ and $r\le a$,
the support lies in $B_{3/2}$ and
\[
 A_r^\kappa\mathbf1_E(x)
   =\int_{z\cdot\nu>-x\cdot\nu/r}\kappa(z)\,dz.
\]
On the negative half-ball this increases with $r$, so $u=v$;
on the positive half-ball sufficiently small radii give $u=1$.
For $x\in P$, half-space containment bounds every average by
$\theta_\kappa(\nu)$, attained by $r\le a$. Thus
$u=v=\theta_\kappa(\nu)$ on $P$. Gluing the Lipschitz negative
half-ball and constant positive half-ball gives
$u\in SBV_{\rm loc}(B_1)$ with jump
$(1-\theta_\kappa(\nu))\nu$, proving~\eqref{r6:flatsharpness}.
\end{proof}

Thus the directional coefficient, and $\eta_\kappa$ by a minimizing
direction, are optimal in Corollary~\ref{r9:localfine} for every
bounded compact probability kernel. Under the signed-band hypothesis,
Lemma~\ref{lem:R7bandextension} gives $u\in BV_{\rm loc}(\mathbb R^d)$
for this test, proving optimality also under the global output hypothesis;
this includes Theorem~\ref{r3:bodyendpoint}. No global BV mapping
property is asserted for arbitrary bounded compact kernels.

\begin{corollary}\label{r7:perimeterjump}
Let $E$ have finite volume and perimeter, and assume
$u=M_\kappa\mathbf1_E\in BV_{\rm loc}$. Then
\begin{equation}\label{r7:perimeterjumpformula}
 D^ju=(1-u)D\mathbf1_E,\qquad
 |D^ju|=(1-u)|D\mathbf1_E|.
\end{equation}
If $\kappa\ge c\mathbf1_{B_\rho}$ for some $c,\rho>0$, then
\begin{equation}\label{r7:perimeterjumpset}
 J_u=\partial^*E\quad\bmod\mathcal H^{d-1}.
\end{equation}
The jump density is then strictly positive almost everywhere on the
reduced boundary; no uniform positive lower bound is asserted.
\end{corollary}
\begin{proof}
Traces $0,1$ and $0\le u\le1$ give the measure formula. At a reduced
boundary point, small-radius averages tend to
$\theta_\kappa(\nu_E^{\rm in})<1$, large-radius averages tend to zero,
and every positive-radius average is below one because the positive
inner kernel sees $E^c$. Positive-radius continuity, from $L^1$
dilation continuity of $\kappa$, therefore gives $u(x)<1$.
Apply~\eqref{r7:exactjumpmeasure}.
\end{proof}

\begin{remark}\label{r6:jumpaffine}
For $\kappa=|K|^{-1}\mathbf1_K$, local BV output is the only
additional hypothesis, even for bounded Borel $K$. The endpoint
supplies it for the bodies of Theorem~\ref{r3:bodyendpoint}, so the
sharp singular conclusion is noncircular. An inner ball gives
$\kappa_K=1-\min_\nu\theta_K(\nu)<1$. For invertible $B_0$,
\[
 \theta_{(B_0)_\#\kappa}(\nu)
 =\theta_\kappa(B_0^T\nu/|B_0^T\nu|),
\]
so the uniform coefficient is linearly invariant; translations may
change it.
\end{remark}

\begin{remark}\label{r6:one-sidedellipsoid}
For $\epsilon>0$ and $K_\epsilon=\{|w|^2+\epsilon(w_1)_+^2<1\}$, put
$q=\sqrt{1+\epsilon}$ and $J_d=\omega_{d-1}/(d+1)$.
Its positive and negative half-volumes have ratio $q^{-1}:1$,
and first moments $q^{-2}J_de_1,-J_de_1$. Thus its sharp
$e_1$ coefficient is $q/(1+q)$ and its centroid points along $-e_1$;
gradient continuity needs no centroid restriction.
\end{remark}

\begin{remark}
The contact identity does not eliminate every Cantor part, even at
zero centroid; \cite[Theorem~1.1]{LW} gives a sufficient ball criterion.
Later closures invoke this conditional kernel theorem once.
\end{remark}

\section{Endpoint regularity and stability transfer}
\label{r6:continuitysection}\label{r10:kerneltransfersection}
The signed-band estimate gives the endpoint and local stability;
global gradient continuity additionally requires spatial tails.

Fix $d\ge2$, $R\ge1$, $B,C_0>0$, and a family $\mathcal P$ of kernels
such that
\begin{equation}\label{r7:uniformbodydata}
 \kappa\ge0,\qquad\int\kappa=1,\qquad
 \supp\kappa\subset B_R,\qquad\|\kappa\|_\infty\le B.
\end{equation}
Write $A_r^\kappa h(x)=\int\kappa(z)h(x+rz)\,dz$,
$S_{a,b}^\kappa h=\sup_{a\le r\le b}A_r^\kappa h$, and
$M_\kappa f=\sup_{r>0}A_r^\kappa|f|$. Assume
\begin{equation}\label{r7:uniformsignedband}
 \|\nabla S_{a,b}^\kappa h\|_1\le C_0\|\nabla h\|_1
 \quad(h\in C_c^\infty(\R^d;\R),\ 0<a\le b<\infty).
\end{equation}
All bands act on real signed inputs. Put $q=q_d=d/(d-1)$ and recall the homogeneous BV space $\mathcal X_d$ defined in the introduction.

\begin{theorem}\label{r10:kerneltransfer}
Under~\eqref{r7:uniformbodydata}--\eqref{r7:uniformsignedband}:
\begin{enumerate}
\item The finite-band and full-variation conclusions of
Lemma~\ref{lem:R7bandextension} hold uniformly:
signed bands satisfy~\eqref{r5:allbands}, belong to $L^q\cap\operatorname{Lip}$
for $h\in\mathcal X_d$, and are globally $W^{1,1}$ for $h\in BV\cap L^1$.
The full operator maps $\mathcal X_d$ into itself with
\eqref{r5:BVbound}, retaining exactly $C_0$.
\item For $g=|f|$, $u=M_\kappa f$ and $f\in\mathcal X_d$,
\begin{equation}\label{r10:kernelfine}
 D^su=\nu_g(g^\vee-u)_+\mathcal H^{d-1}\llcorner J_g
 +(D^cg)\llcorner C_\kappa(g),\quad
 C_\kappa(g)=\{u^*=g^*,\ |u^*|+|g^*|<\infty\}.
\end{equation}
The jump density uses the natural maximal value. In particular
$|D^su|\le|D^sg|$ and
$|D^ju|\le(1-\theta_\kappa(\nu_g))|D^jg|$, where
$\theta_\kappa(\nu)=\int_{z\cdot\nu>0}\kappa(z)\,dz$.
The directional coefficient is sharp. Natural and precise maximal
values agree at $|D^cg|+|D^cu|$-almost every point.
\item There are $c,C>0$, depending only on $d,R,B,C_0$, such that
\begin{equation}\label{r10:kerneltruncation}
 \int_{\{|\nabla M_\kappa f|>\lambda\}}|\nabla M_\kappa f|
 \le C\int_{\{|\nabla f|>c\lambda\}}|\nabla f|
 \quad(f\in W^{1,1},\ \lambda>0).
\end{equation}
The Orlicz conclusion in Corollary~\ref{r7:orlicz} follows as well.
\item If $\kappa_n,\kappa\in\mathcal P$, $\kappa_n\to\kappa$ in
$L^1$, and $f_n\to f$ in $W^{1,1}$, then
\begin{equation}\label{r10:kernellocalcontinuity}
 \|M_{\kappa_n}f_n-M_\kappa f\|_q\to0,\qquad
 \nabla M_{\kappa_n}f_n\to\nabla M_\kappa f
                                    \quad\hbox{in }L^1_{\rm loc}.
\end{equation}
For a fixed positive finite band the signed envelopes converge
globally in $W^{1,1}$ without using~\eqref{r7:uniformsignedband}.
\item If the family admits the spatial-tightness certificate
of Definition~\ref{r12:tailcertificate}, the gradient convergence in
\eqref{r10:kernellocalcontinuity} is global in $L^1$.
\end{enumerate}
The certificate is additional data; a general Dini or bounded probability
kernel has only local gradient continuity until it is verified.
\end{theorem}

\subsection{Endpoint transfer and kernel convergence}
The common kernel bounds give
\begin{equation}\label{r10:basicbounds}
 A_r^\kappa F(x)\le Br^{-d}\int_{B(x,Rr)}F,\qquad
 M_\kappa F\le B\omega_dR^dM_BF\quad(F\ge0).
\end{equation}

\begin{lemma}\label{lem:R7kernelvalues}
Under~\eqref{r7:uniformbodydata}, if $\kappa_n\to\kappa$ in $L^1$
and $f_n\to f$ in $L^p$, $1<p<\infty$, then
$M_{\kappa_n}f_n\to M_\kappa f$ in $L^p$.
No signed-band estimate is required.
\end{lemma}
\begin{proof}
Put $g=|f|$. At a Lebesgue point, the common ball bound gives
$A_r^{\kappa_n}g(x)\to g(x)$ uniformly in $n$ as $r\downarrow0$;
H\"older gives uniform large-radius decay $Cr^{-d/p}\|g\|_p$.
On $[a,b]$, smooth $L^p$ approximation leaves kernel error
$\|\widetilde g\|_\infty\|\kappa_n-\kappa\|_1$ and approximation
error $Ca^{-d/p}\|g-\widetilde g\|_p$. The three ranges give
\begin{equation}\label{r7:pointwisevalueconvergence}
 M_{\kappa_n}g(x)\longrightarrow M_\kappa g(x).
\end{equation}
Domination by $CM_Bg\in L^p$ gives strong convergence, and
$CM_B|f_n-f|$ handles variable inputs.
\end{proof}

For the homogeneous fine-structure passage use
\begin{equation}\label{r10:homogeneousbound}
 |A_r^\kappa h(x)|\le\|\kappa\|_{q'}r^{-d/q}\|h\|_q.
\end{equation}
For $g=|f|\in BV\cap L^1$, Theorem~\ref{r10:conditionalkernelfine}
gives part (2). For $g=|f|\in\mathcal X_d$, apply scalar excision to
cutoffs agreeing with $g$ on the small averaging supports of each
compact observation set. Retain the original envelope $u_T$:
Lemma~\ref{lem:R7bandextension}'s translation bound gives
$\operatorname{Lip}u_T\le BT^{-d}|Dg|(\mathbb R^d)$, while
\eqref{r10:homogeneousbound} ensures finiteness and infinite-radius
decay. The trace, relative-continuity and diffuse-locality proof
therefore applies by exhaustion, without a prior singular estimate.
Sharpness is~\eqref{r6:flatsharpness}. In particular
\begin{equation}\label{r7:transferredendpoint}
 M_\kappa f\in L^q\cap W^{1,1}_{\rm loc},\qquad
 \|\nabla M_\kappa f\|_1\le C_0\|\nabla f\|_1
 \quad(f\in W^{1,1}).
\end{equation}

\subsection{Differentiable averages and truncation}
\begin{lemma}\label{r6:averageC1}
For $h\in W^{1,1}$, $(x,r)\mapsto A_r^\kappa h(x)$ is $C^1$ on
$\R^d\times(0,\infty)$ and
\begin{equation}\label{r7:averageidentities}
 \nabla_xA_r^\kappa h=A_r^\kappa\nabla h,\qquad
 \partial_rA_r^\kappa h(x)
       =\int\kappa(z)z\cdot\nabla h(x+rz)\,dz.
\end{equation}
If $\kappa_n\to\kappa$ in $L^1$ under the common data and
$h_n\to h$ in $W^{1,1}$, values and first derivatives converge
uniformly on $\R^d\times[a,b]$, $0<a\le b<\infty$.
\end{lemma}
\begin{proof}
Smooth approximation errors on $[a,b]$ in value, spatial derivative
and radial derivative are bounded respectively by
\[
 Ba^{-d}\|h-\widetilde h\|_1,\quad
 Ba^{-d}\|\nabla h-\nabla\widetilde h\|_1,\quad
 BRa^{-d}\|\nabla h-\nabla\widetilde h\|_1.
\]
Differentiate smooth approximants under the integral and use these
bounds for continuity and the identities. For a fixed approximant,
kernel errors cost its corresponding supremum norm times
$\|\kappa_n-\kappa\|_1$, with factor $R$ for the radial derivative.
Approximation gives the asserted uniform convergence.
\end{proof}

At an approximately differentiable point, a Sobolev function and
a touching $C^1$ minorant have the same gradient: otherwise their
nonnegative difference is negative on a positive-density cone.
Thus every positive maximizing radius gives
\begin{equation}\label{r7:positiveradiustouching}
 \nabla M_\kappa g(x)=A_r^\kappa\nabla g(x)
 \quad\text{at almost every relevant point}.
\end{equation}
Radius zero on $\{M_\kappa g=g\}$ is handled by Sobolev locality.

\begin{lemma}\label{r7:gradienttails}
There are $c,C>0$ depending only on $d,R,B,C_0$ such that every
nonnegative $g\in W^{1,1}$ satisfies
\begin{equation}\label{r7:amplitudetail}
 \int_{\{|\nabla M_\kappa g|>\lambda\}}|\nabla M_\kappa g|
 \le C\int_{\{|\nabla g|>c\lambda\}}|\nabla g|.
\end{equation}
Uniformly $L^1$-bounded and uniformly integrable input gradients
therefore give uniformly integrable output gradients.
\end{lemma}
\begin{proof}
Put $u=M_\kappa g$, $H=|\nabla g|$ and
$E_\lambda=\{|\nabla u|>\lambda\}$. A detached Lebesgue point has
a positive maximizing radius, because the averages tend to $g(x)$
at zero and to zero at infinity. Touching implies
\begin{equation}\label{r7:radiuscaptureinput}
 \int_{B(x,Rr)}H\ge B^{-1}\lambda r^d
 \quad\text{at almost every detached point of }E_\lambda.
\end{equation}
Choose $L_0$ large and
$0<\varepsilon<\min\{1/2,[2B\omega_d(L_0+R)^d]^{-1}\}$.
For $O_\lambda=\{M_BH>\varepsilon\lambda\}$ the preceding bound gives
\begin{equation}\label{r7:radiuscapture}
 B(x,L_0r)\subset O_\lambda,\qquad
 r\le L_0^{-1}\dist(x,O_\lambda^c).
\end{equation}
Indeed a ball centered at $z\in B(x,L_0r)$ of radius $(L_0+R)r$
contains $B(x,Rr)$. Contact points of $E_\lambda$ also lie in
$O_\lambda$ almost everywhere.

Take Whitney cubes with $20Q\subset O_\lambda$,
$100Q\cap O_\lambda^c\ne\varnothing$ and bounded overlap of $5Q$.
Choose $L_0$ so large that~\eqref{r7:radiuscapture} puts every
relevant active radius on $Q$ below $b_Q/4$, where
$b_Q=\ell(Q)/R$. The increasing sets
\[
 E_\lambda\cap Q\cap\{u>g\}\cap
 \{u=S_{1/j,b_Q}^\kappa g\},\qquad 1/j\le b_Q,
\]
cover the detached part almost everywhere. Locality and the common
cube estimate~\eqref{r7:whitneyconstantcost} in
Lemma~\ref{lem:R7bandextension} bound its gradient integral by
$C_dC_0\int_{5Q}H$, uniformly in $j$. Monotone convergence, bounded
overlap and the contact cost $\int_{O_\lambda}H$ give
\begin{equation}\label{r7:whitneylevelbound}
 \int_{E_\lambda}|\nabla u|\le C\int_{O_\lambda}H.
\end{equation}
For $c=\varepsilon/2$ and $H_\lambda=H\mathbf1_{\{H>c\lambda\}}$,
$O_\lambda\subset\{M_BH_\lambda>c\lambda\}$. Weak type gives
$c\lambda|O_\lambda|\le C_d\|H_\lambda\|_1$, whence
$\int_{O_\lambda}H\le(1+C_d)\|H_\lambda\|_1$. This proves the
amplitude bound. Together with the endpoint, vanishing amplitude
tails imply the stated uniform integrability.
\end{proof}

\begin{corollary}\label{r7:orlicz}
For every finite Young function $\Phi$, with $c,C$ above,
\begin{equation}\label{r7:orliczmodular}
 \int\Phi(c|\nabla M_\kappa g|)
       \le cC\int\Phi(|\nabla g|),\qquad g\ge0,\ g\in W^{1,1}.
\end{equation}
The same conclusion holds for $M_\kappa f$ with $|\nabla f|$ on
the right; integrals may be infinite.
\end{corollary}
\begin{proof}
For the nondecreasing $p(t)=\Phi(t)/t$ and
$p_0=\lim_{t\downarrow0}p(t)$, Tonelli gives
\[
 \int\Phi(cF)=cp_0\int F+
 c\int_{(0,\infty)}\left(\int_{\{F>s/c\}}F\right)dp(s).
\]
Apply the amplitude bound, including its limit at zero, to
$F=|\nabla M_\kappa g|$. For signed inputs use
$|\nabla|f||\le|\nabla f|$.
\end{proof}

\subsection{Finite bands and local strong continuity}
\begin{lemma}\label{r6:finitebandcontinuity}\label{r6:gradientmeasurecontinuity}
Suppose~\eqref{r7:uniformbodydata}, $\kappa_n\to\kappa$ in $L^1$
and $h_n\to h$ in $W^{1,1}$. Then:
\begin{enumerate}
\item For $0<a\le b<\infty$,
$S_{a,b}^{\kappa_n}h_n\to S_{a,b}^\kappa h$ globally in $W^{1,1}$,
without a signed-band hypothesis.
\item Under~\eqref{r7:uniformsignedband}, the full output gradients
$\nabla M_{\kappa_n}h_n$ converge locally in measure to
$\nabla M_\kappa h$.
\end{enumerate}
\end{lemma}
\begin{proof}
Use $g_n=h_n,g=h$ for finite bands and $g_n=|h_n|,g=|h|$
for full maxima. Absolute value is continuous in $W^{1,1}$ by the
chain rule and $\nabla h=0$ on $\{h=0\}$. Denote the envelopes by
$u_n,u$. Finite-band values and average gradients converge uniformly
by Lemma~\ref{r6:averageC1}. For full maxima, the endpoint handles
$g=0$; otherwise extract summable $W^{1,1}$ errors from any subsequence.
Weak type and Borel--Cantelli give
\begin{equation}\label{r6:maximalerrors}
 M_B|g_n-g|+M_B|\nabla g_n-\nabla g|\longrightarrow0
 \quad\text{almost everywhere}.
\end{equation}
Together with~\eqref{r7:pointwisevalueconvergence}, this gives $u_n\to u$
almost everywhere.

At common differentiability, Lebesgue and locality points choose
active radii $r_n$, allowing zero on full contact sets. Pass to a
limit in $[0,\infty]$. A positive finite limit $r$ is active for $u$;
compact-band $C^1$ convergence and touching give
$\nabla u_n\to A_r^\kappa\nabla g=\nabla u$.
This settles finite bands. For full maxima, a zero limit gives
$u=g$ and $\nabla u_n\to\nabla g=\nabla u$ by the common ball
Lebesgue bound,~\eqref{r6:maximalerrors} and locality, including
indices $r_n=0$. At infinity,
\[
 |A_{r_n}^{\kappa_n}g_n|+|A_{r_n}^{\kappa_n}\nabla g_n|
 \le Br_n^{-d}(\|g_n\|_1+\|\nabla g_n\|_1)\to0,
\]
so $u=0$ and locality gives the zero gradient. Every radius limit
therefore gives the same gradient; the subsequence criterion proves
full convergence in measure.

For finite bands, both value and gradient differences are bounded by
\[
 2Ba^{-d}\int_{B(x,bR)}G+
 Ba^{-d}\int_{B(x,bR)}G_n,
\]
with $(G,G_n)=(|h|,|h_n-h|)$ or
$(|\nabla h|,|\nabla h_n-\nabla h|)$. The first term is integrable;
the second tends to zero in $L^1$. Subtracting the second and using
dominated convergence proves (1).
\end{proof}
For full maxima, the amplitude bound and Vitali upgrade (2) to local
$L^1$ gradient convergence. Lemma~\ref{lem:R7kernelvalues} and Sobolev
embedding give the value convergence in part (4).

\subsection{A global spatial-tightness certificate}
Only large active radii require the following extra hypothesis;
the common support remains $B_R$.

\begin{definition}\label{r12:tailcertificate}
The family has a global spatial-tightness certificate if there is a
function $\tau(L)\to0$ as $L\to\infty$ such that, for each
$\kappa$, nonnegative $h\in W^{1,1}$ and $L>0$, a measurable
$G_{\kappa,h,L}\ge0$ satisfies
\begin{equation}\label{r12:tailcertificatehyp}
 \int_{|x|>L}G_{\kappa,h,L}(x)\,dx\le\tau(L)\|h\|_1,
 \quad
 |\nabla_xA_r^\kappa h(x)|\le G_{\kappa,h,L}(x)
 \quad\begin{matrix}
 r\ge |x|/(4R),\\[-2pt]
 \partial_rA_r^\kappa h(x)=0.
 \end{matrix}
\end{equation}
The second assertion must hold for almost every $|x|>L$,
simultaneously for every radius satisfying the displayed conditions.
\end{definition}

\begin{lemma}\label{r6:gradienttails}\label{r12:tailtransfer}
Under this certificate, every nonnegative $h\in W^{1,1}$ obeys
\begin{equation}\label{r6:tailestimate}
 \int_{|x|>L}|\nabla M_\kappa h|
 \le C_0\int_{|y|>L/2}|\nabla h(y)|\,dy
       +\left(\frac{C_dC_0}{L}+\tau(L)\right)\|h\|_1.
\end{equation}
Consequently, convergent Sobolev inputs have uniformly tight output
gradients. Their local $L^1$ gradient convergence in
Theorem~\ref{r10:kerneltransfer} is therefore global.
\end{lemma}
\begin{proof}
Take $0\le\chi_L\le1$, equal to one on $B_{L/2}$, supported in
$B_{3L/4}$, with $|\nabla\chi_L|\le C_d/L$. Put
$h_L=(1-\chi_L)h$, $u=M_\kappa h$. On
$E_L=\{|x|>L:M_\kappa h_L=u\}$, locality and the endpoint give
\begin{equation}\label{r6:smallradiusoutside}
 \int_{E_L}|\nabla u|\le C_0\|\nabla h_L\|_1
 \le C_0\int_{|y|>L/2}|\nabla h|+\frac{C_dC_0}{L}\|h\|_1.
\end{equation}
This includes contact points and points with active
$r\le|x|/(4R)$ almost everywhere, since their supports avoid
$B_{3L/4}$. At remaining nonzero-gradient points, a larger positive
active radius exists and is stationary. Touching and the certificate
bound their gradients by $G_{\kappa,h,L}$, proving
\eqref{r6:tailestimate}.
For convergent Sobolev inputs, the input tails vanish uniformly and
the $L^1$ norms are bounded. Hence
\begin{equation}\label{r6:uniformtails}
 \lim_{L\to\infty}\sup_n\int_{|x|>L}|\nabla M_{\kappa_n}h_n|=0.
\end{equation}
Together with local $L^1$ convergence this proves part (5).
\end{proof}

\begin{corollary}\label{r12:stationarycertificate}
For nonnegative $h\in W^{1,1}$, the uniform pointwise condition
\begin{equation}\label{r10:stationaryhyp}
 \partial_r A_r^\kappa h(x)=0
 \ \Longrightarrow\ 
 |\nabla_xA_r^\kappa h(x)|\le C_{\rm st}\|h\|_1r^{-d-1}
\end{equation}
supplies the certificate with
$\tau(L)=d\omega_d(4R)^{d+1}C_{\rm st}/L$.
\end{corollary}
\begin{proof}
Use $G_{\kappa,h,L}(x)=(4R)^{d+1}C_{\rm st}\|h\|_1|x|^{-d-1}$
and integrate over $|x|>L$.
\end{proof}

\begin{corollary}\label{r10:convexlifttail}
Suppose $\kappa=\int_T\kappa_t\,d\pi(t)$, where $\pi$ is a
probability measure and the measurably parametrized $\kappa_t$ are
horizontal marginals of bounded open convex bodies
$L_t\subset\mathbb R^{d+k_t}$, $0\le k_t\le k_0$.
Assume $B_{d+k_t}(0,\rho)\subset L_t$ and common marginal support
$B_R$ and density bound $B$. Then, for $h\ge0$ in $W^{1,1}$,
\begin{equation}\label{r10:positiveEuler}
 \rho r|\nabla_xA_r^\kappa h|
 \le r\partial_rA_r^\kappa h+(d+k_0)A_r^\kappa h.
\end{equation}
Hence the preceding certificate applies with
$C_{\rm st}=(d+k_0)B/\rho$. This assertion supplies tail data, not
an endpoint estimate for a new mixture.
\end{corollary}
\begin{proof}
For one component set $D_t=d+k_t$, $X=(x,0)$,
$\Omega=X+rL_t$, $H(y,z)=h(y)$ and $V_t=A_r^{\kappa_t}h(x)$.
The nonnegative Sobolev trace, Gauss--Green and
$(Y-X)\cdot n(Y)\ge\rho r$ give
\[
 \rho r|\nabla_xV_t|
 \le\frac{1}{|L_t|r^{D_t}}
 \int_{\partial\Omega}\operatorname{Tr}H(Y)(Y-X)\cdot n(Y)
 \,d\mathcal H^{D_t-1}
 =r\partial_rV_t+D_tV_t.
\]
Replace $D_t$ by $d+k_0$ using $V_t\ge0$, then integrate and apply
the triangle inequality. Common density/support bounds and Sobolev
approximation on compact positive radius bands justify differentiation
under the probability integral. At stationarity,
$V\le Br^{-d}\|h\|_1$, so Corollary~\ref{r12:stationarycertificate}
applies with the stated constant.
\end{proof}

\begin{corollary}\label{r7:jointstability}\label{r6:continuitycriterion}
\label{r7:fixedcontinuity}
Let $K_n,K$ be open convex bodies with common bounds
\begin{equation}\label{r7:uniformconvexdata}
 B_\rho\subset K_n,K\subset B_R.
\end{equation}
Assume a common signed-band constant and
$\mathbf1_{K_n}\to\mathbf1_K$ in $L^1$. If $f_n\to f$ in
$W^{1,1}$, then
\begin{equation}\label{r6:fullcontinuity}
 \|M_{K_n}f_n-M_Kf\|_q+
 \|\nabla M_{K_n}f_n-\nabla M_Kf\|_1\to0.
\end{equation}
In particular, for every fixed bounded open convex body $K$ containing
zero with a verified signed-band estimate, including the classes in
Theorem~\ref{r3:bodyendpoint}, $f\mapsto(M_Kf,\nabla M_Kf)$ is
continuous from $W^{1,1}$ to $L^q\times L^1$.
\end{corollary}
\begin{proof}
The normalized indicators converge in $L^1$ and satisfy
$B=(\omega_d\rho^d)^{-1}$. Corollary~\ref{r10:convexlifttail}, with
$k_0=0$, gives the common tail certificate; apply
Theorem~\ref{r10:kerneltransfer}. The fixed-body case takes $K_n=K$.
\end{proof}

\begin{remark}\label{r6:centroidconstants}
The transfer constants depend on the density/support and signed-band
data, together with the stated spatial-tail data when global gradient
convergence is sought. They do not involve an inverse centroid.
Body translation changes an averaging family and is not covered by
linear covariance alone.
\end{remark}

\section{Kernel closures and joint stability}\label{r12:closures}
\begin{lemma}\label{lem:R5weaklaws}
Let $t\mapsto\kappa_t$ be continuous into $L^1$, with
$\|\kappa_t\|_1\le1$, and let $\pi_n,\pi$ be Borel probability laws.
If $T$ is compact metric and $\pi_n\rightharpoonup\pi$ on $T$, then
\[
 \int_T\kappa_t\,d\pi_n(t)\longrightarrow
 \int_T\kappa_t\,d\pi(t)\quad\hbox{in }L^1.
\]
The same holds for $T=(0,1]$ and weak convergence on $[0,1]$, provided
\[
 \sup_n\pi_n((0,\delta])+\pi((0,\delta])\longrightarrow0
 \quad(\delta\downarrow0),
\]
and all laws have zero mass at $0$.
\end{lemma}
\begin{proof}
On compact $T$, a continuous partition of unity approximates
$t\mapsto\kappa_t$ uniformly by a finite sum; apply weak convergence.
For $(0,1]$, first cut off near zero. The cutoff map extends
continuously to $[0,1]$, and the omitted Bochner integrals have norm
bounded by the displayed masses. Let $\delta\downarrow0$.
\end{proof}

\subsection{Horizontal marginals}\label{r7:marginalsection}
Let $d\ge2$, $k\ge1$, $D=d+k$, and let the bounded open convex
body $L\subset\R^d\times\R^k$ contain zero. Its horizontal marginal is
\begin{equation}\label{r7:marginalkernel}
 \phi(w)=|L|^{-1}|\{z:(w,z)\in L\}|,\qquad
 A_r^\phi h(x)=\int\phi(w)h(x+rw)\,dw.
\end{equation}
If $B_D(0,\rho)\subset L\subset B_D(0,R)$, then Fubini gives
$\phi\ge0$, $\int\phi=1$, $\supp\phi\subset\overline B_R$ and
$\|\phi\|_\infty\le\omega_kR^k/|L|$. The inner ball bounds
$\phi$ below on $B_{\rho/2}$.

\begin{theorem}\label{r7:marginalclosure}
Let $L\subset\mathbb R^{d+k}$ be a bounded open convex body containing
zero, and assume its signed all-band estimate has constant $C_L$.
Its horizontal marginal $\phi$ has all conclusions of
Theorem~\ref{r10:kerneltransfer}, with
\[
 C_0=C_L,\qquad B=\|\phi\|_\infty\le\omega_kR^k/|L|,
 \qquad \operatorname{supp}\phi\subset B_R,
\]
including global strong gradient continuity. Its directional jump
coefficient is
\[
 \theta_\phi(\nu)=|L|^{-1}
 |L\cap\{(w,z):w\cdot\nu>0\}|.
\]
In particular the assertion applies to every fixed
$C^{1,\theta}$ convex lift, $0<\theta\le1$, by
Corollary~\ref{thm:R4holderendpoint}, and to every other convex lift
for which the signed-band estimate has been verified.
\end{theorem}
\begin{proof}
For $h\in C_c^\infty(\mathbb R^d;\mathbb R)$, fix $[a,b]$ and
smooth $0\le\chi_T\le1$, equal to one on $Q_T=(-T,T)^k$ and supported in $Q_{T+1}$,
with $\int\chi_T/|Q_T|\to1$ and
$\int|\nabla\chi_T|/|Q_T|\to0$. For $H_T=h\chi_T$, Fubini
identifies the band averages on $Q_{T-bR}$, $T>bR$, so
\[
 |Q_{T-bR}|\|\nabla S_{a,b}^\phi h\|_1
 \le\|\nabla S_{a,b}^LH_T\|_1
 \le C_L\left(\|\nabla h\|_1\int\chi_T+
 \|h\|_1\int|\nabla\chi_T|\right).
\]
Division and $T\to\infty$ preserve $C_L$. Apply
Theorem~\ref{r10:kerneltransfer} with the displayed kernel bounds
and the single-lift certificate $C_{\rm st}=DB_\phi/\rho$.
\end{proof}

\subsection{Probability profiles on a compact scale window}\label{r8:profilessection}
Fix an admissible body $K$ and $0<\ell\le L<\infty$.
For any probability measure $\pi$ on $[\ell,L]$, put
\[
\begin{gathered}
 d\sigma(a)=a^{-d}\,d\pi(a),\qquad B=\ell^{-d},\qquad
 H_r(y)=\int\mathbf1_{x+arK}(y)\,d\sigma(a),\\
 T_r^\pi g=\int A_{ar}^Kg\,d\pi(a),\qquad
 \phi_\pi=|K|^{-1}\int\mathbf1_{aK}\,d\sigma(a).
\end{gathered}
\]
Thus $\sigma([\ell,L])\le B$ and $\phi_\pi$ is a probability
density. No atomicity, density or regularity of $\pi$ is assumed.

\begin{theorem}\label{r11:compactprofileendpoint}\label{r8:finiteprofileendpoint}\label{r9:finiteuniform}\label{r8:finiteprofile}
There is $C=C(K,\ell,L)$ such that, uniformly over all these
probability measures, every real $g\in BV\cap L^1$ satisfies
\[
 \left\|\nabla\sup_{a\le r\le b}T_r^\pi g\right\|_1
       \le C|Dg|(\mathbb R^d),\qquad 0<a\le b<\infty.
\]
The finite-band envelopes are globally Lipschitz and $W^{1,1}$.
The full operators have the BV and Sobolev endpoints, the exact
singular identity, and the common spatial gradient tail. If
$\pi_n\rightharpoonup\pi$ on $[\ell,L]$ and $f_n\to f$ in
$W^{1,1}$, their maximal values converge in $L^{d/(d-1)}$ and
their gradients converge globally in $L^1$.
\end{theorem}

\begin{proof}
The local comparison and master theorems give a common smooth
signed-band constant; $\phi_\pi\le(|K|\ell^d)^{-1}$ and
$\supp\phi_\pi\subset LK$ give common kernel data.
The inner ball $B_{\ell\rho_K}\subset aK$ supplies the common spatial
certificate by Corollary~\ref{r10:convexlifttail}.
Apply Theorem~\ref{r10:kerneltransfer} and Lemma~\ref{lem:R5weaklaws}
to the $L^1$-continuous layer map
$a\mapsto(|K|a^d)^{-1}\mathbf1_{aK}$.
\end{proof}

\begin{corollary}\label{r9:uniformmixedendpoint}\label{r9:jointprofilecontinuity}
The preceding theorem includes finite probability profiles
$\pi=\sum_{i=1}^Nc_i\delta_{\lambda_i}$ with constants independent of
$N$, weights and scale gaps, including the two-scale parameter square
$[0,1]\times[1,2]$. Coordinate convergence of finite profiles, or
weak convergence of arbitrary profiles on the fixed window, together
with $f_n\to f$ in $W^{1,1}$, gives joint value and strong gradient
continuity.
\end{corollary}

\paragraph{Convex marginal lifting does not generate step profiles.}
For $0<\eta<1$ and $s>1$, the step profile
$(1-\eta)|K|^{-1}\mathbf1_K+\eta(|K|s^d)^{-1}\mathbf1_{sK}$
violates midpoint log-concavity on a positive-measure set of pairs
crossing $\partial K$. Convex-body marginals are log-concave and
continuous in the projection interior by concavity of root fiber
volume. They therefore cannot equal this profile or approach it in
$L^1$: an almost-everywhere convergent subsequence passes the midpoint
inequality to almost every pair.

The stronger source-free range does not replace the positive-cone
estimate in Theorem~\ref{r3:localinputs}.

\section{Intrinsic prices and radial regularization}\label{r11:Diniflat}
For bounded probability kernels supported in $[-1,1]^d$, define
\begin{equation}\label{r15:intrinsicoscillation}
 \mathfrak o_\kappa(s)=\sup_{|v|\le s}\int_{\R^d}
  \operatorname*{ess\,sup}_{1\le r\le2}r^{-d}
  |\kappa((x+v)/r)-\kappa(x/r)|\,dx.
\end{equation}
Put $\mathfrak K(\kappa)=\|\kappa\|_\infty+
\sum_{j\ge0}\mathfrak o_\kappa(C_s2^{-j})$.
The radius essential supremum makes the price invariant under null-set
changes of $\kappa$. Operator maxima retain ordinary suprema.
The oscillation is nondecreasing and at most $C_d\|\kappa\|_\infty$
on bounded arguments.

\subsection{The intrinsic-price theorem}\label{r15:intrinsicsection}
\begin{lemma}\label{lem:R10bandmass}
Let $\kappa$ be a bounded probability density supported in $[-1,1]^d$.
Suppose $M_\kappa F\le A M_{\rm ball}|F|$ for locally integrable $F$, and put
\[
 N_\kappa=\int\operatorname*{ess\,sup}_{1\le r\le2}
 r^{-d}\kappa(x/r)dx,\qquad
 P_0=A+N_\kappa.
\]
If $P_0+\sum_j\mathfrak o_\kappa(C_s2^{-j})<\infty$, then for every
real $g\in C_c^\infty(\R^d)$ and $0<a\le b<\infty$,
\[
 \|\nabla S_{a,b}^\kappa g\|_1
 \le C_d[P_0+\sum_j\mathfrak o_\kappa(C_s2^{-j})]\|\nabla g\|_1.
\]
\end{lemma}
\begin{proof}
Use $w(s)=s^{1/2}+P_0^{-1}\mathfrak o_\kappa(C_ss)$.
For a zero-mass error $\beta_l$ supported in
$\overline B(z_l,C_ss_l)$, kernel subtraction and Minkowski--Tonelli give
\[
 \int_Q\sup_{1\le r\le2}\left|\int r^{-d}
 \kappa((y-x)/r)d\beta_l(y)\right|dx
 \le\mathfrak o_\kappa(C_ss_l)|\beta_l|(\R^d).
\]
Here the errors are absolutely continuous. Smooth $L^1$ approximation
and boundedness of the fixed kernel make their convolutions continuous
in $(x,r)$ on positive compact bands, so ordinary and essential
radius suprema agree, including endpoints. Size-one errors cost
$N_\kappa|\beta_l|$; positivity handles the cone-source term.
These give both local inputs with constants $CP_0$.

In the smooth assembly of Appendices~\ref{sec:globalprices}--\ref{sec:globalassembly},
the remaining density uses are the value-oscillation estimate and a
final short band. The assumed maximal domination replaces $B$ by
$A$ in Lemma~\ref{r3:automaticoscillation}, while
\[
 \int\sup_{c\le r\le2c}|A_r^\kappa\nabla g|
 \le N_\kappa\|\nabla g\|_1.
\]
The probability rows use $\operatorname{Lip}(g)$, not kernel density.
Thus the same assembly costs $C_dP_0\sum_jw(2^{-j})$, as claimed.
\end{proof}

\begin{theorem}\label{r15:intrinsicendpoint}
Let $d\ge2$ and let $\kappa$ be a probability density supported in
$[-1,1]^d$. If $\mathfrak K(\kappa)<\infty$, the signed-band,
full BV and Sobolev conclusions of Theorem~\ref{r3:master} hold with
constant $C_d\mathfrak K(\kappa)$, together with the exact singular
identity, truncation and local joint gradient continuity of
Theorem~\ref{r10:kerneltransfer}. Bounded prices give uniform conclusions.
\end{theorem}
\begin{proof}
For $B=\|\kappa\|_\infty$, bounded support gives
$M_\kappa F\le C_dB M_{\rm ball}|F|$ and $N_\kappa\le C_dB$.
Lemma~\ref{lem:R10bandmass} gives the smooth signed-band bound
$C_d\mathfrak K(\kappa)$. Apply Lemma~\ref{lem:R7bandextension}
and Theorem~\ref{r10:kerneltransfer}.
\end{proof}

\begin{corollary}\label{r15:intrinsicmixtures}
For jointly measurable probability densities $\kappa_t$ on the common
support, with measurable component oscillation norms, let measurable
$K_t\ge\mathfrak K(\kappa_t)$. Every probability
law $\pi$ with $\int K_t\,d\pi<\infty$ satisfies
\[
 \mathfrak K\left(\int\kappa_t\,d\pi(t)\right)
              \le\int K_t\,d\pi(t).
\]
Uniform expectations give uniform conclusions without component counts
or separation; global gradient continuity additionally requires a common
spatial-tail certificate.
\end{corollary}
\begin{proof}
The triangle inequality and Tonelli give both the $L^\infty$ bound and
$\mathfrak o_{\int\kappa_t d\pi}(s)\le\int\mathfrak o_{\kappa_t}(s)d\pi$;
measurable majorants suffice. Sum and apply the theorem and tail transfer.
\end{proof}

The same oscillation, possibly infinite for $L^1$ kernels, gives an
intrinsic-price approximation barrier for raw convex indicators.
\begin{proposition}\label{prop:intrinsic-approximation-barrier}
If $B_\rho\subset K\subset B_1$, $K$ is open convex and
$\kappa_K=|K|^{-1}\1_K$, then constants $c,C,s_0>0$ depending only on
$d,\rho$ satisfy, for every $\kappa\in L^1$ and
$\varepsilon=\|\kappa-\kappa_K\|_1$,
\[
 \mathfrak o_\kappa(s)\ge c-C\varepsilon/s\quad(0<s\le s_0),
 \qquad
 \sum_{j\ge0}\mathfrak o_\kappa(C_s2^{-j})
                       \ge c\log(1/\varepsilon)-C
\]
for sufficiently small $\varepsilon>0$. In particular
$\mathfrak K(\kappa_K)=\infty$.
\end{proposition}
\begin{proof}
Fix $e\in S^{d-1}$ and
$E=\{5/4<p_K(x)<3/2,\ |x/|x|-e|<\rho/4\}$.
Gauge polar integration gives $|E|\ge c_{d,\rho}$. At
$z=x/p_K(x)$, a supporting unit normal $n$ satisfies
$h=z\cdot n\in[\rho,1]$ and $n\cdot e\ge3\rho/4$.
The support functional $n/h$ and the gauge Lipschitz bound give
\[
 as\le p_K(x+se)-p_K(x)\le bs,
 \qquad a=\rho/2,\quad b=\rho^{-1}.
\]
For $s\le s_0=\rho/4$, the interval
$I_s(x)=(p_K(x),p_K(x+se))$ lies in $[1,2]$, where the raw kernel
difference is at least $m=2^{-d}/|K|$. Fubini gives
\[
 \int_E\int_{I_s(x)}r^{-d}
 |\kappa((x+se)/r)-\kappa(x/r)|\,dr\,dx
       \ge mas|E|-2\varepsilon.
\]
Its left side is at most $bs\mathfrak o_\kappa(s)$, proving the first
bound. Every dyadic argument between $2C\varepsilon/c$ and $s_0$
then contributes at least $c/2$; counting them proves the second.
For $\varepsilon=0$ infinitely many terms have this lower bound.
\end{proof}
Geometric compact-profile estimates may remain uniform as this price diverges.

\subsection{Dini continuity and weighted scale concentration}
For continuous compactly supported $\kappa$, set
\[
 \omega_\kappa(t)=\sup_{|z|\le t}
       \|\kappa(\cdot+z)-\kappa\|_\infty,
 \qquad \mathcal D_\kappa=\int_0^1\omega_\kappa(t)\,\frac{dt}{t}.
\]
\begin{corollary}\label{r11:Diniendpoint}
\label{r11:Holderendpoint}\label{r11:holderkernelmother}
A continuous probability density supported in $[-1,1]^d$ with
$\mathcal D_\kappa<\infty$ has all bounded-kernel conclusions of
Theorem~\ref{r15:intrinsicendpoint}, with constant
$C_d(\|\kappa\|_\infty+\mathcal D_\kappa)$.
In particular compactly supported $C^{0,\beta}$ probability kernels,
$0<\beta\le1$, qualify since
$\mathcal D_\kappa\le[\kappa]_{C^{0,\beta}}/\beta$ after support
normalization.
\end{corollary}
\begin{proof}
Translated physical kernels have common bounded support and difference
at most $\omega_\kappa(s)$, so
$\mathfrak o_\kappa(s)\le C_d\omega_\kappa(s)$ for $s\le1$.
Dyadic--Dini comparison gives the price bound; apply the intrinsic theorem.
\end{proof}

Let $K$ be any bounded open convex body containing zero, with gauge
$p$. For a finite measure $\zeta$ on $(0,1]$, extended by zero outside,
write
\[
 Q_\zeta(h)=\sup_u\zeta((u,u+h]),\qquad
 Z_\zeta=\zeta((0,1]),\qquad
 D_\zeta=\int_0^1Q_\zeta(h)\,\frac{dh}{h}.
\]
\begin{theorem}\label{r15:weightedDiniendpoint}
\label{r15:weightedscalesection}
If $\pi$ is a probability on $(0,1]$ and
$d\zeta(\lambda)=\lambda^{-d}d\pi(\lambda)$ satisfies
$Z_\zeta+D_\zeta<\infty$, then
\[
 \kappa_\pi(w)=|K|^{-1}\int\lambda^{-d}\1_{\lambda K}(w)\,d\pi(\lambda)
\]
is continuous and has the signed-band, BV/Sobolev, exact-singular and
local gradient-continuity conclusions, with constant
$C_{d,K}(Z_\zeta+D_\zeta)$. For laws on $(0,b]$ use the law of
$\lambda/b$; this reparameterizes radii without changing the full operator
or its endpoint constant.
\end{theorem}
\begin{proof}
Dini finiteness excludes atoms, so $J(t)=\zeta((t,1])$, constant
below zero and zero above one, is continuous and satisfies
$|J(t)-J(s)|\le Q_\zeta(|t-s|)$. Tonelli gives
$\kappa_\pi=|K|^{-1}J\circ p$ and unit mass. Gauge Lipschitzness yields
\[
 \|\kappa_\pi\|_\infty=Z_\zeta/|K|,
 \quad \omega_{\kappa_\pi}(t)\le|K|^{-1}Q_\zeta(L_pt),
 \quad
 \|\kappa_\pi\|_\infty+\mathcal D_{\kappa_\pi}
       \le C_{d,K}(Z_\zeta+D_\zeta).
\]
Apply Corollary~\ref{r11:Diniendpoint}; dilation by $b$ reparameterizes
the averaging radius as $br$.
\end{proof}

\begin{corollary}\label{r11:Diniprofileflat}\label{r11:frostmanprofiles}
\label{r11:innerzeroHolder}\label{r11:logflatregularization}
\label{r15:dyadicscaleprice}\label{r15:weakinnerzerostability}
The weighted-Dini theorem has the following verifications, with the
bounded-kernel transfer conclusions throughout.
\begin{enumerate}[label=\textup{(\roman*)}]
\item On a fixed positive window $[\lambda_-,\lambda_+]$, the condition
$\mathcal D_\pi=\int_0^1Q_\pi(h)dh/h<\infty$ gives constant
$C_{d,K,\lambda_-,\lambda_+}(1+\mathcal D_\pi)$ and global gradient
continuity. This includes Frostman laws $Q_\pi(h)\le Hh^\beta$ over
arbitrary convex bodies, including polytopes, with no layer count.
\item If $\zeta((u,v])\le H|v-u|^\beta$, $0\le u<v\le1$ and
$0<\beta\le1$, the kernel is globally $\beta$-H\"older and its price
is at most $C_{d,K}H(1+\beta^{-1})$. Sufficient data are
$\pi(I)\le H_0|I|^\gamma$, $t,\gamma>0$, and
$M_t=\int\lambda^{-(d+t)}d\pi<\infty$, since
\[
 \zeta(I)\le M_t^{d/(d+t)}H_0^{t/(d+t)}
                       |I|^{\gamma t/(d+t)}.
\]
This negative moment is sufficient, not necessary.
\item For $I_j=(2^{-j-1},2^{-j}]$, let $M_j=\zeta(I_j)$ and
$\rho_j$ be $\zeta|_{I_j}/M_j$ rescaled by $2^j$. With
$D_j=\int_0^1Q_{\rho_j}(t)dt/t$ and $M_jD_j=0$ if $M_j=0$,
the price is bounded by $C_{d,K}\mathfrak P(\pi)$ whenever
\begin{equation}\label{r15:dyadiclogprice}
 \mathfrak P(\pi)=\sum_{j\ge0}M_j(1+j+D_j)<\infty.
\end{equation}
Common bounded densities of $\rho_j$ reduce this to $\sum_j(1+j)M_j$.
\item For the uniform law on $[1,1+\epsilon]$, $0<\epsilon\le1$,
the endpoint constant is $C_{d,K}(1+\log(1/\epsilon))$. For
$f\in L^q$, $1<q<\infty$,
\[
 \|\kappa_\epsilon-|K|^{-1}\1_K\|_1\le d\epsilon,\qquad
 \|M_{\kappa_\epsilon}f-M_Kf\|_q\longrightarrow0.
\]
This divergent bound gives no raw flat-body endpoint or optimality claim.
\item For weakly convergent laws $\pi_n\to\pi$ on $[0,1]$,
supported in $(0,1]$ with common $Z_\zeta+D_\zeta$ bounds,
$W^{1,1}$ input convergence gives maximal-value convergence in
$L^{d/(d-1)}$ and gradient convergence in $L^1_{\rm loc}$.
A common cusp bound with exponent below $d$ gives global convergence;
in case \textup{(i)} the fixed window supplies the common tail.
\end{enumerate}
\end{corollary}
\begin{proof}
For \textup{(i)}, normalize $\lambda_+=1$ and use
$Z_\zeta+D_\zeta\le\lambda_-^{-d}(1+\mathcal D_\pi)$.
The convex components $\lambda K$ have common inner radius
$\lambda_-\rho_K$, so Corollary~\ref{r10:convexlifttail} supplies
their mixture tail. In \textup{(ii)}, integrate the concentration
bound; the moment estimate is H\"older's inequality. For
\textup{(iii)}, concentration subadditivity and Tonelli give
\[
 D_\zeta\le\sum_jM_j\int_0^1Q_{\rho_j}(2^jh)\,\frac{dh}{h}
 \le\sum_jM_j(D_j+j\log2).
\]
For \textup{(iv)}, $Q(h)\le\min(1,h/\epsilon)$ gives the logarithm.
Nestedness gives indicator-kernel distance
$2(1-\lambda^{-d})\le2d(\lambda-1)$; average it.
The common density/support bounds and Lemma~\ref{lem:R7kernelvalues}
give value convergence. For \textup{(v)},
$\pi_n((0,\delta])\le Z_{\zeta_n}\delta^d$ permits
Lemma~\ref{lem:R5weaklaws} for the norm-one layer map; then apply
the common endpoint/tail transfers. Hausdorff-convergent bodies with common inner and outer balls
also give kernel convergence, since their normalized indicators converge
uniformly in $L^1$ over compact positive scale windows.
\end{proof}
The normalized density
$\1_{(0,e^{-2p_0})}(\lambda-1)/
((\lambda-1)[\log(1/(\lambda-1))]^{p_0})$, $p_0>2$,
has $Q(h)\asymp[\log(1/h)]^{1-p_0}$: it is Dini but not positive-power
Frostman. These prices exclude atoms, which compact-profile geometry
treats on fixed positive windows. Outer support remains separate:
mixture-active radii need not maximize components, and domination by
component maxima does not control variation.

\subsection{Spatial tails and their mixtures}
\begin{lemma}\label{lem:R7Hardyenvelope}
Let $B_\rho\subset K\subset B_R$ have gauge $p$, and let a bounded
$W^{1,1}$ kernel supported in $\overline K$ satisfy
$|\nabla\kappa(w)|\le F(p(w))$ on $K$, where $F\ge0$ is Borel.
Put
\[
 H(t)=\operatorname*{ess\,sup}_{0<u\le\min(t,1)}u^{d+1}F(u),
 \qquad \mathcal H_F=H(1)+\int_0^1H(t)\,\frac{dt}{t^2}.
\]
For $I_j=(2^{-j-1},2^{-j}]$ one has the equivalent dyadic criterion
\begin{equation}\label{eq:R13Hardydyadic}
 \mathcal H_F\asymp_d
 \sum_{j\ge0}2^{-jd}\operatorname*{ess\,sup}_{I_j}F.
\end{equation}
If $\mathcal H_F<\infty$, then the integrated band-gradient envelope
$\mathcal E_\kappa(z)=\operatorname*{ess\,sup}_{1\le r\le2}
 r^{-d-1}|\nabla\kappa(z/r)|$ has
$\|\mathcal E_\kappa\|_1\le C_{d,\rho,R}\mathcal H_F$.
Moreover, for $h\ge0$ in $W^{1,1}$, a measurable envelope controls
$|\nabla A_r^\kappa h(x)|$ simultaneously for $r\ge|x|/(4R)$ and
has integral over $|x|>L$ at most
$C_{d,\rho,R}\mathcal H_F\|h\|_1/L$.
\end{lemma}
\begin{proof}
Put $M_j=\operatorname*{ess\,sup}_{I_j}F$. Dyadic integration and
Tonelli applied to
\[
 2^{-(j+1)(d+1)}M_j\le H(2^{-j})
 \le\sum_{\ell\ge j}2^{-(d+1)\ell}M_\ell
\]
give~\eqref{eq:R13Hardydyadic}.
Homogeneity and gauge-polar integration on $2K$ give
$\mathcal E_\kappa(z)\le p(z)^{-d-1}H(p(z))$ and its $L^1$ bound.
For large radii the kernel derivative is bounded by
\[
 Q_H(x,y)=C|x-y|^{-d-1}H(C|x-y|/|x|).
\]
For fixed $y$, split at $|x-y|=|x|/2$. The outer integral over
$|x|>L$ costs $CH(1)/L$; on the inner part
$|x|\asymp|y|\ge L/2$, radial integration costs
$C|y|^{-1}\int_0^CH(t)t^{-2}dt\le C\mathcal H_F/L$.
Thus $\int Q_H(x,y)h(y)dy$ is the claimed envelope by Tonelli.
For each fixed $r$, the Sobolev convolution identity and the derivative
majorant hold at almost every center. Intersect over rational radii;
Lemma~\ref{r6:averageC1} then extends the bound to every eligible radius.
\end{proof}

\subsection{Integrated derivatives}\label{r15:integratedbandsection}
Recall the measurable band-gradient envelope $\mathcal E_\kappa$
from Lemma~\ref{lem:R7Hardyenvelope}; its definition applies to every
bounded compactly supported $W^{1,1}$ kernel.
\begin{corollary}\label{r15:bandgradientendpoint}
A bounded $W^{1,1}(\R^d)$ probability kernel supported in $[-1,1]^d$
with $B=\|\kappa\|_\infty$ and $\|\mathcal E_\kappa\|_1<\infty$
has the conclusions of Theorem~\ref{r15:intrinsicendpoint}, with
constant $C_d(B+\|\mathcal E_\kappa\|_1)$. No sup-norm Dini or
positive H\"older modulus is required.
\end{corollary}
\begin{proof}
For $F_r(y)=r^{-d}\kappa(y/r)$, Sobolev absolute continuity on
lines, Fubini and Tonelli give
\begin{equation}\label{r15:integratedtranslation}
 \int\operatorname*{ess\,sup}_{1\le r\le2}
      |F_r(y+a)-F_r(y)|dy\le |a|\|\mathcal E_\kappa\|_1.
\end{equation}
Indeed the segment integral is bounded by
$|a|\int_0^1\mathcal E_\kappa(y+ta)dt$ for almost every $(y,r)$.
Thus $\mathfrak o_\kappa(s)\le s\|\mathcal E_\kappa\|_1$;
sum and apply the intrinsic theorem.
\end{proof}

\begin{corollary}\label{r11:cusptail}\label{r15:completecuspkernel}\label{cor:R7criticallog}
A bounded compactly supported $W^{1,1}$ probability kernel with density
bound $B$, support radius $R$, and
$|\nabla\kappa(w)|\le A|w|^{-\sigma}$, $0\le\sigma<d$, has all
bounded-kernel conclusions, including global gradient continuity.
Its endpoint constant is at most
$C_{d,\sigma}(BR^d+AR^{d+1-\sigma})$, and
\begin{equation}\label{r12:cuspallradius}
 \int_{|x|>L}\sup_{r\ge|x|/(4R)}|\nabla A_r^\kappa h(x)|dx
 \le C_{d,\sigma}AR^{d+1-\sigma}\|h\|_1/L.
\end{equation}
Common displayed data give joint global continuity under kernel $L^1$
and input $W^{1,1}$ convergence.
For $R=1$ the same conclusions hold under the critical bound on $0<|w|<1$
\[
 |\nabla\kappa(w)|\le A|w|^{-d}[\log(e/|w|)]^{-\beta},
 \qquad \beta>1,
\]
with endpoint constant $C_d\{B+A[1+(\beta-1)^{-1}]\}$ and the
same derivative budget in the $L^{-1}$ spatial-tail bound.
\end{corollary}
\begin{proof}
Dilation to $R=1$ changes the derivative coefficient to
$AR^{d+1-\sigma}$. The two radial budgets in
Lemma~\ref{lem:R7Hardyenvelope} are
\[
\begin{array}{c|c|c}
 F(t)&H(t),\ 0<t\le1&\mathcal H_F\\ \hline
 At^{-\sigma}&At^{d+1-\sigma}&A[1+(d-\sigma)^{-1}]\\
 At^{-d}\log(e/t)^{-\beta}&At\log(e/t)^{-\beta}
                         &A[1+(\beta-1)^{-1}].
\end{array}
\]
Both use that $t^{d+1}F(t)$ increases. Thus the same lemma,
Corollary~\ref{r15:bandgradientendpoint} and the transfer theorem
give both cases. In the power case its Schur bound reads
\begin{equation}\label{r11:cuspQintegral}
 \sup_y\int_{|x|>L}|x-y|^{-\sigma}
       \max\{|x|,|x-y|\}^{-d-1+\sigma}dx\le C_{d,\sigma}/L.
\end{equation}
Undo the radius reparameterization.
\end{proof}

For an explicit kernel without a continuous representative at zero,
take $0\le\chi\in C_c^\infty([0,1))$, equal to one near zero, and
normalize
\[
 \kappa(w)=c\chi(|w|)\{2+\sin\varphi(|w|)\},\qquad
 \varphi(r)=r^{1-d}[\log(e/r)]^{-\beta},\quad \beta>1.
\]
The value at zero is immaterial. Since
$\varphi'(r)=r^{-d}[\log(e/r)]^{-\beta}
 [1-d+\beta/\log(e/r)]$, the gradient satisfies the critical bound
and is integrable. The inner boundary term $O(r^{d-1})$ vanishes,
proving $W^{1,1}$ membership. Arbitrarily small positive-measure annuli
with $\sin\varphi>1/2$ and $\sin\varphi<-1/2$ exclude a continuous
representative; annuli with $|\cos\varphi|\ge1/2$ exclude every
power bound $\sigma<d$. Thus the corollary applies.

\begin{corollary}\label{r12:integratedtailmixtures}
Let $\kappa_t$ be a jointly measurable probability family with common
support radius $R$ and measurable density bounds $B_t$. Suppose jointly
measurable $G_{t,h,L}$ control all radii $r\ge|x|/(4R)$ and
$\int_{|x|>L}G_{t,h,L}\le\tau_t(L)\|h\|_1$.
For laws $\pi$ with uniformly bounded $\int B_t\,d\pi$ and
$\sup_\pi\int\tau_t(L)d\pi\to0$, the mixtures have a common tail
certificate. Uniform signed-band constants therefore give joint global
gradient continuity under kernel $L^1$ and input $W^{1,1}$ convergence.
\end{corollary}
\begin{proof}
The expected density bound justifies differentiation on positive bands.
Average $G$; Fubini controls rational radii, continuity controls all
radii, and Tonelli gives the tail. Component stationarity alone cannot
be averaged, since it need not hold at a mixture-stationary radius.
\end{proof}
For instance, $\kappa=c(1-|w|^{\beta_0})_+$, $0<\beta_0<1$,
fails any all-stationary estimate
$|\nabla A_r^\kappa h|\le C_*\|h\|_1r^{-d-1}$.
Indeed $d\kappa+w\cdot\nabla\kappa
=c[d-(d+\beta_0)|w|^{\beta_0}]$ changes sign. At $x=0,r=1$,
combine a unit bump of width $o(\epsilon)$ at $\epsilon e_1$ with
a symmetric pair where this expression is negative; choose positive
weights to cancel its integral. The weights stay positive, whereas
the spatial gradient is $\asymp\epsilon^{\beta_0-1}\to\infty$.
The all-large-radius cusp envelope avoids this obstruction.

Dilation reparameterizes radii, preserving
$R^{d+\beta}[\kappa]_{C^\beta}$ and $AR^{d+1-\sigma}$.
Removing the inner-scale cutoff does not treat unbounded outer support,
which still needs a nonlocal error budget or another mixture certificate.

\section{Scope and further questions}\label{sec:R8scope}
The fixed-body and logarithmic-normal results
(Section~\ref{sec:R8lognormal}) provide no uniform bound under
arbitrary convex approximation; general $C^1$ and polytope raw
endpoints remain open here. The independent rough-body note records
closure and facet identities. For bounded compact probability kernels,
local BV output is required for the singular identity, whose Cantor
contact term need not vanish. Common endpoint data give local gradient
continuity; global continuity requires spatial tails.
Appendix~\ref{sec:R8singular} retains separate localization and
coarse-radius gradient hypotheses.

\section*{Acknowledgements}
The geometric endpoint framework originates in the released work of
OpenAI~\cite{OpenAI}; its inherited constructions are identified above.

\appendix
\section{A common summable price for all cone directions}\label{sec:globalprices}
The refinement, direction-law and signed-envelope constructions of
\cite[Sections~3--5]{OpenAI} are proved here in every fixed dimension
$d\ge2$. Throughout Appendices~\ref{sec:globalprices}--\ref{glob:regularity} fix the bounded compact probability
kernel and prices of Theorem~\ref{r3:master}, $g\in C_c^\infty(\R^d;\R)$,
$h=\nabla g\,dx$, $V=\|\nabla g\|_1$, a finite depth $M\ge1$ and
$r_k=2^{-k}$. Grids and masses are defined at all integer levels.

\subsection{Positive refinement and cone coordinates}
For $p\in\mathcal V_k:=r_k\mathbb Z^d$ define
\[
 Q_p=p+[0,r_k)^d,\qquad
 \varphi_p(x)=\prod_{j=1}^d\left(1-\frac{|x_j-p_j|}{r_k}\right)_+.
\]
Vertices retain their level as part of their identity. All constants below
may depend on the fixed dimension. For $p\in\mathcal V_i$,
$q\in\mathcal V_k$, $i<k$, set $\lambda_{pq}=\varphi_p(q)$; for $i=k$
set $\lambda_{pq}=\1_{\{p=q\}}$. Linear interpolation on the one-dimensional
dyadic intervals, followed by taking products, gives
\begin{equation}\label{glob:refinement}
 \sum_{p\in\mathcal V_i}\lambda_{pq}=1,\qquad
 \varphi_p=\sum_{q\in\mathcal V_k}\lambda_{pq}\varphi_q,\qquad
 \lambda_{ps}=\sum_{q\in\mathcal V_j}\lambda_{pq}\lambda_{qs}
 \quad(i\le j\le k).
\end{equation}
All sums involved at a fixed point are locally finite. In particular, these
weights are normalized in the coarse index for each fixed fine vertex.
If $q\in\mathcal V_{k+1}$ and $\lambda_{pq}>0$, then
$q_j-p_j\in\{-r_k/2,0,r_k/2\}$ for all coordinates, and consequently
\begin{equation}\label{glob:nest}
 \supp\varphi_q\subset\supp\varphi_p.
\end{equation}
Write
\[
 h_p=\int\varphi_p\,dh=a_p\theta_p,\qquad a_p=|h_p|,
\]
with an arbitrary prescribed unit vector $\theta_p$ when $a_p=0$.
Equation~\eqref{glob:refinement} implies
\begin{equation}\label{glob:vectormasses}
 h_p=\sum_{q\in\mathcal V_k}\lambda_{pq}h_q,\qquad
 \sum_{p\in\mathcal V_i}a_p\le\sum_{q\in\mathcal V_k}a_q\le V
 \quad(i\le k).
\end{equation}
At a fixed level, only finitely many vector masses are nonzero.

Fix
\[
 \tau=\min\{1/100,1/(100d)\},
 \qquad \mathcal N\subset S^{d-1}
\]
with $\mathcal N$ a finite $\tau/20$-net in chord distance. For every
$n\in\mathcal N$, independently choose unit vectors
$t_1(n),\ldots,t_d(n)\in n^\perp$ forming a regular simplex:
\[
 \sum_{i=1}^d t_i=0,\qquad
 t_i\cdot t_j=-\frac1{d-1}\quad(i\ne j),\qquad
 \sum_{i=1}^d t_i\otimes t_i=\frac d{d-1}P_{n^\perp}.
\]
These are choices on a finite set; no continuous tangent frame is used.
Put $a=\sqrt3/2$, $b=1/2$, and
\[
 e_i(n)=a n+b t_i(n),\qquad
 c_{n,i}(z)=\frac{n\cdot z}{ad}
                 +\frac{d-1}{bd}\,t_i(n)\cdot z.
\]
The following lemma supplies every cone property used in the proof.
\begin{lemma}\label{glob:simplexcone}
The vectors $e_i(n)$ are unit vectors and, for $i\ne j$,
\[
 e_i(n)\cdot e_j(n)=a^2-\frac{b^2}{d-1}\ge\frac12.
\]
For every $z\in\R^d$,
\begin{equation}\label{glob:coordinates}
 z=\sum_{i=1}^d e_i(n)c_{n,i}(z),\qquad
 \sum_{i=1}^d|c_{n,i}(z)|\le C_d|z|.
\end{equation}
All these coordinates are nonnegative whenever
\[
 n\cdot z\ge0,\qquad
 |P_{n^\perp}z|\le\frac{b}{a(d-1)}\,n\cdot z.
\]
There is a constant $C_d$ such that
\begin{equation}\label{glob:coneineq}
 |z|\le C_d\left(n\cdot z+
                  \sum_{i=1}^d(-e_i(n)\cdot z)_+\right).
\end{equation}
\end{lemma}
\begin{proof}
The simplex identities give
\[
 \sum_i c_{n,i}(z)=a^{-1}n\cdot z,\qquad
 \sum_i t_i c_{n,i}(z)=b^{-1}P_{n^\perp}z.
\]
They prove reconstruction, the coordinate bound and, together with the
formula for $c_{n,i}$, positivity on the stated cone. Expanding inner
products gives the first assertion.
Put $p_i=e_i(n)\cdot z$. Since $\sum_i p_i=da\,n\cdot z$,
\[
 n\cdot z+\sum_i(-p_i)_+
 =\frac1{da}\sum_i(p_i)_+
   +\left(1-\frac1{da}\right)\sum_i(-p_i)_+
 \ge c_d\sum_i|p_i|\ge c_d|z|.
\]
Here $da>1$, and the final bound uses the fixed Gram matrix of the
basis $e_i(n)$. This proves~\eqref{glob:coneineq} uniformly in the simplex.
\end{proof}

Call $(q,n)$ good if $c_{n,i}(\theta_q)\ge0$ for every $1\le i\le d$,
and bad otherwise. At a good pair every scalar coordinate of $h_q$ is
nonnegative. In particular, all directions sufficiently close to $n$
are good, with the explicit cone margin furnished by the lemma.

\subsection{Exact scalar pairing}
Let $\mathcal C$ be the indexed family of the $d$ scalar functionals
$c_{n,i}$ for each $n\in\mathcal N$. For $c\in\mathcal C$ write
$\nu^c=c(h)$ and $c_p=c(h_p)$. Define the nonnegative defects
\begin{equation}\label{glob:defects}
 d_p^c=\begin{cases}
 \displaystyle\frac12\left(\sum_{q\in\mathcal V_{k+1}}
       \lambda_{pq}|c_q|-|c_p|\right),&p\in\mathcal V_k,\quad k<M,\\[5pt]
 \displaystyle\frac12\left(\int\varphi_p\,d|\nu^c|-|c_p|\right),
       &p\in\mathcal V_M.
 \end{cases}
\end{equation}

\begin{lemma}\label{glob:scalarpairing}
For every $c\in\mathcal C$ and $p\in\mathcal V_k$, $1\le k\le M$,
there are absolutely continuous measures $\rho_p^c,\mu_p^{c,+},
\mu_p^{c,-}$ supported in $\supp\varphi_p$, such that $\rho_p^c$ has
constant sign, the $\mu_p^{c,\pm}$ are nonnegative, and
\[
 \rho_p^c(\R^d)=c_p,\quad |\rho_p^c|(\R^d)=|c_p|,
 \quad \mu_p^{c,+}(\R^d)=\mu_p^{c,-}(\R^d)=d_p^c.
\]
With $\beta_p^c=\mu_p^{c,+}-\mu_p^{c,-}$, the identity
\begin{equation}\label{glob:scalaridentity}
 \varphi_q\nu^c=\rho_q^c+
       \sum_{l=k}^M\sum_{s\in\mathcal V_l}\lambda_{qs}\beta_s^c
       \qquad(q\in\mathcal V_k)
\end{equation}
holds as measures on the whole space. Moreover
\begin{equation}\label{glob:scalarbudget}
 \begin{gathered}
 \beta_s^c(\R^d)=0,\quad |\beta_s^c|(\R^d)\le2d_s^c,
 \quad \supp\beta_s^c\subset\overline B(s,\sqrt d\,r_l),\\
 \sum_{l=1}^M\sum_{s\in\mathcal V_l}d_s^c
 =\frac12\left(\|c(\nabla g)\|_1-
                 \sum_{s\in\mathcal V_1}|c_s|\right).
 \end{gathered}
\end{equation}
\end{lemma}
\begin{proof}
For nonnegative measures $U,W$ of masses $P,N$, remove
\[
 \mu^+=\frac{\min(P,N)}P U,\qquad
 \mu^-=\frac{\min(P,N)}N W,
\]
with zero denominators interpreted as zero. The residual
$\rho=U-W-\mu^++\mu^-$ has constant sign, mass $P-N$ and variation
$|P-N|$; support and absolute continuity are preserved.
At level $M$ apply this to $U=\varphi_s(\nu^c)^+$,
$W=\varphi_s(\nu^c)^-$. Descending inductively, apply it at $p$ to
\[
 U=\sum_q\lambda_{pq}(\rho_q^c)^+,\qquad
 W=\sum_q\lambda_{pq}(\rho_q^c)^-.
\]
The removed mass is precisely~\eqref{glob:defects}, and
$\sum_q\lambda_{pq}\rho_q^c=\rho_p^c+\beta_p^c$.
Support nesting and refinement composition now prove
\eqref{glob:scalaridentity} and the support assertions. Equal removed
masses give cancellation and $|\beta_s^c|\le2d_s^c$.
Finally, summing~\eqref{glob:defects} telescopes: column normalization
cancels every intermediate absolute mass, and the terminal partition
of unity gives $\sum_s\int\varphi_s\,d|\nu^c|=\|c(\nabla g)\|_1$.
\end{proof}

\subsection{The common prices}
Use the price $w$ and finite total $W$ from Theorem~\ref{r3:master}. For $Q=Q_p$ at level $k$,
let
\[
 \mathcal R(Q)=\{q\in\mathcal V_k:|q-p|_\infty\le20r_k\},\qquad
 D_Q=D_Q^w=\sum_{c\in\mathcal C}\sum_{q\in\mathcal R(Q)}
         \sum_{l=k}^Mw(2^{k-l})
                     \sum_{s\in\mathcal V_l}\lambda_{qs}d_s^c.
\]
The radius sets and oscillation used at this scale are
\begin{align*}
 I_k(x)&=\{t\in[r_k,2r_k]:u_t(x)\ge u_s(x)
                    \text{ for all }s>0,\ |s-t|\le r_k/4\},\\
 J_k(x)&=\{t\in[r_k,2r_k]:u_t(x)=\max_{r_k\le s\le2r_k}u_s(x)\}
                       \cup\{r_k,2r_k\},\\
 v_k(x)&=\sup_{0<t,t'\le4r_k}|u_t(x)-u_{t'}(x)|.
\end{align*}

\begin{proposition}\label{glob:commonprices}
The prices above obey
\begin{equation}\label{glob:pricesum}
 \sum_{k=1}^M\sum_{p\in\mathcal V_k}D_{Q_p}\le C_dWV.
\end{equation}
For every $n\in\mathcal N$ and every level-$k$ cell $Q$,
\begin{align}
 \int_Q\sup_{t\in I_k(x)}\sum_{i=1}^d
          (-e_i(n)\cdot\nabla u_t(x))_+\,dx
 &\le C\left(D_Q+w(1)
          \sum_{\substack{q\in\mathcal R(Q)\\(q,n)\text{ bad}}}a_q\right),
                 \label{glob:directionprice}\\
 \int_Q\left[\sup_{t\in J_k(x)}|\nabla u_t(x)|+r_k^{-1}v_k(x)\right]dx
 &\le C\left(D_Q+w(1)\sum_{q\in\mathcal R(Q)}a_q\right).
                 \label{glob:sizeprice}
\end{align}
The local constant is at most $C_{d,\alpha_0}(C_D+C_B+B)$,
and is independent of $M,w,W$.
\end{proposition}
\begin{proof}
At each level $k\le l$, column normalization and the $41^d$
root-neighborhood multiplicity count a defect $d_s^c$ with total
coefficient $41^dw(2^{k-l})$. Hence Tonelli and
\eqref{glob:scalarbudget} give
\[
 \sum_{k,p}D_{Q_p}
 \le41^dW\sum_{c,l,s}d_s^c
 \le\frac{41^dW}{2}\sum_c\|c(\nabla g)\|_1\le C_dWV.
\]

Fix $Q=Q_p$ and $\Omega_Q=\{y:\dist_\infty(y,Q)<10r_k\}$.
Only tents with $|q-p|_\infty<12r_k$ meet $\Omega_Q$, so
$h=\sum_{q\in\mathcal R(Q)}\varphi_qh$ there.
Equations~\eqref{glob:coordinates} and~\eqref{glob:scalaridentity} give
\begin{equation}\label{glob:rootdecomposition}
 h=\sum_{q\in\mathcal R(Q)}\sum_{i=1}^d
              e_i(n)\rho_q^{c_{n,i}}
  +\sum_{q\in\mathcal R(Q)}\sum_{i=1}^d
       \sum_{l=k}^M\sum_{s\in\mathcal V_l}
            e_i(n)\lambda_{qs}\beta_s^{c_{n,i}}
       \quad\hbox{on }\Omega_Q.
\end{equation}
Good roots have nonnegative residuals and directions pairing at
least $1/2$ with every $e_i(n)$. Put bad-root residuals among
size-one errors; by~\eqref{glob:coordinates} their price is at most
$C_dw(1)\sum_{q\in\mathcal R(Q),\ (q,n)\text{ bad}}a_q$.
Each level-$l$ pair has zero vector mass, variation at most
$2\lambda_{qs}d_s^c$, and support in
$\overline B(s,\sqrt d\,r_l)$. Assigning relative size $r_l/r_k$
gives total error price at most
\[
 2D_Q+C_dw(1)\sum_{\substack{q\in\mathcal R(Q)\\(q,n)\text{ bad}}}a_q.
\]
After the change of variables $y=p+r_kX$, the gradient measure
transforms as $\widetilde h=r_k^{1-d}(T^{-1})_\#h$, and the physical
integral of a gradient supremum equals $r_k^{d-1}$ times its unit-scale
counterpart. Assumption~\eqref{r3:inputD} therefore gives
\eqref{glob:directionprice}, once for each of the $d$ test directions.

For~\eqref{glob:sizeprice}, put \emph{all} root residuals among the
size-one errors. The discounted mass is then at most
$2D_Q+C_dw(1)\sum_{q\in\mathcal R(Q)}a_q$. Assumption~\eqref{r3:inputB} gives
the gradient term. Lemma~\ref{r3:automaticoscillation}, under the same scaling,
gives $r_k^{-1}\int_Qv_k$ with this same upper bound: its unit-scale
left side is $r_k^{-d}\int_Qv_k$, whereas its measure price is
$r_k^{1-d}$ times the physical price. All measures remain absolutely
continuous, all sums are finite, and the support constant may be
fixed as $C_s=\max\{2,\sqrt d\}$. This proves~\eqref{glob:sizeprice}.
\end{proof}

\begin{remark}
The exact scalar decomposition is shared by every choice of later direction.
The last residual term in~\eqref{glob:directionprice} still cannot be summed
over scales directly. The next section controls that term by charging changes
of direction, using the vector masses $a_q$ rather than these scalar pairs.
\end{remark}

\section{A direction law with summable switching and spatial costs}
\label{sec:global-direction-law}

Only the positive vector refinement is used here. Retain the grids,
finite palette, cone coordinates and depth above, and write
\begin{equation}\label{eq:gd-refinement}
 h_p=a_p\theta_p=\sum_{q\in\mathcal V_k}\lambda_{pq}a_q\theta_q,
 \qquad A_i:=\sum_{p\in\mathcal V_i}a_p\le A_k\le V\quad(i\le k).
\end{equation}
All choices use a fixed tie rule independent of the observation point.

\subsection{A spatial direction discrepancy}

For $z\in\mathcal V_k$ choose $w(z)\in\mathcal V_k$ maximizing $a_w$ subject
to $|w-z|_\infty\le300r_k$.  Set $\Theta_z=\theta_{w(z)}$ and choose
$n_z\in\mathcal N$ with $|n_z-\Theta_z|\le\tau/20$.  Define
\begin{equation}\label{eq:gd-discrepancy}
 \mathcal E_k(q)=
 \max_{\substack{z\in\mathcal V_k\\|z-q|_\infty\le200r_k}}
 |\Theta_z-\theta_q|^2,
 \qquad q\in\mathcal V_k.
\end{equation}
In particular $0\le\mathcal E_k(q)\le4$.

\begin{lemma}\label{lem:gd-neighbor}
There is a constant $C$ depending only on $d$ such that
\begin{equation}\label{eq:gd-neighbor-budget}
 \sum_{k=1}^M\sum_{q\in\mathcal V_k}a_q\mathcal E_k(q)\le CV.
\end{equation}
\end{lemma}
\begin{proof}
For $a_q>0$, choose maximizing $z$ in \eqref{eq:gd-discrepancy}
and $w=w(z)$. Then $a_w\ge a_q$ and $|w-q|_\infty\le500r_k$.
A nearest $p\in\mathcal V_{k-12}$ satisfies
$|q-p|_\infty\le r_{k-12}/2$ and
$|w-p|_\infty<(5/8)r_{k-12}$; hence
$\lambda_{pq}\ge2^{-d}$, $\lambda_{pw}\ge(3/8)^d$.
Vector refinement, dotted with $\theta_p$, gives
\begin{equation}\label{eq:gd-polar-defect}
 H_{p,k}:=\sum_s\lambda_{ps}a_s-a_p
 =\tfrac12\sum_s\lambda_{ps}a_s|\theta_s-\theta_p|^2,
\end{equation}
also when $a_p=0$. The triangle inequality and $a_w\ge a_q$
therefore give $a_q\mathcal E_k(q)\le CH_{p,k}$.
Each coarse $p$ is chosen by boundedly many $q$ (the lag is fixed),
so column normalization yields
$\sum_qa_q\mathcal E_k(q)\le C(A_k-A_{k-12})$.
Summing in $k$ counts each nonnegative increment $A_j-A_{j-1}$
at most twelve times and costs at most $12CV$.
\end{proof}

\subsection{The finite-state law}

For $k\ge2$ set
\begin{equation}\label{eq:gd-switch-function}
 p_k(n,z)=P(|n-\Theta_z|),\qquad
 P(s)=\min\{1,\tau^{-2}(s-\tau)_+^2\}.
\end{equation}
The stochastic rows
\begin{align}
 \pi_1(x;*,\cdot)&=\sum_z\varphi_z(x)\delta_{n_z},
                                      \label{eq:gd-initial-row}\\
 \pi_k(x;n,\cdot)&=\delta_n+\sum_z\varphi_z(x)p_k(n,z)
                           (\delta_{n_z}-\delta_n),\quad k\ge2,
                                      \label{eq:gd-transition-row}
\end{align}
define a law $\mathbb P_x$ on $\mathcal N^M$, expectation $\mathbb E_x$
and marginals $\zeta_k$, with $\zeta_0=\delta_*$. Equivalently, sample
$z$ with probabilities $\varphi_z(x)$, then replace $n$ by $n_z$ with
probability $p_k(n,z)$. A positive replacement probability implies
$n_z\ne n$, since $|n_z-\Theta_z|\le\tau/20$. Thus, writing
$\rho_k(x\mid n)=\sum_z\varphi_z(x)p_k(n,z)$,
\begin{equation}\label{eq:gd-switch-cost}
 \mathsf R_k=\sum_n\zeta_{k-1}(x;n)\rho_k(x\mid n)
 =\mathbb P_x(n_k\ne n_{k-1}),\quad 2\le k\le M;
 \qquad\mathsf R_1=\mathsf R_{M+1}=1.
\end{equation}
The rows are $C/r_k$-Lipschitz. Off the union $\mathcal G$ of grid
hyperplanes through depth $M$, put
\begin{align}
 \mathsf T_k&=r_k\sum_n\zeta_{k-1}(x;n)
                     \sum_{n'}|\nabla_x\pi_k(x;n,n')|,
                        &&2\le k\le M,\label{eq:gd-row-derivative}\\
 \mathsf T_1&=r_1\sum_{n'}|\nabla_x\pi_1(x;*,n')|,\notag
\end{align}
and set $\mathsf T_k=0$ on $\mathcal G$. Only $2^d$ tents contribute
at a point, each with gradient at most $\sqrt d/r_k$, so
\begin{equation}\label{eq:gd-elementary-bounds}
 0\le\mathsf R_k\le1,\qquad0\le\mathsf T_k\le C.
\end{equation}
The derivative acts on the row alone, not on $\zeta_{k-1}$.
Switch rates are continuous; derivative-cost suprema below are essential.

Retain the preceding definition: $(q,n)$ is good when all $d$ simplex
coordinates $c_{n,i}(\theta_q)$ are nonnegative.

\begin{lemma}\label{lem:gd-good-labels}\label{glob:goodlabels}
Set $c_\ast=(\tau/10)^2$.  If $q\in\mathcal V_k$,
$\mathcal E_k(q)\le c_\ast$, and $|x-q|_\infty\le50r_k$, then
$(q,n_k)$ is good with $\mathbb P_x$-probability one.
\end{lemma}
\begin{proof}
Every sampled vertex of positive probability satisfies
$|z-q|_\infty<51r_k$, and hence $|\Theta_z-\theta_q|\le\tau/10$.
An initial or newly chosen label is within $3\tau/20$ of $\theta_q$.
A retained label of positive conditional probability has
$P(|n-\Theta_z|)<1$, so $|n-\Theta_z|<2\tau$ and
$|n-\theta_q|<21\tau/10$. For a unit vector $\theta$ with
$|n-\theta|\le\delta:=21\tau/10$, one has
\[
 n\cdot\theta\ge1-\delta^2/2>0,\qquad
 \frac{|P_{n^\perp}\theta|}{n\cdot\theta}
 \le\frac{\delta}{1-\delta^2/2}
 \le3\tau<\frac1{\sqrt3(d-1)}=\frac b{a(d-1)}.
\]
The cone condition of Lemma~\ref{glob:simplexcone} therefore applies to
every label of positive probability. This proves the assertion.
\end{proof}

For a vertex $q\in\mathcal V_l$ let
\begin{equation}\label{eq:gd-neighborhood}
 D(q)=\{x:|x-q|_\infty\le25r_l\}.
\end{equation}
In particular, this cube contains every level-$l$ pricing cell having
$q$ in its root neighborhood of radius $20r_l$.

\begin{proposition}\label{prop:gd-packing}\label{glob:packing}
There is a constant depending only on $d$, independent of $M$, such that
\begin{equation}\label{eq:gd-packing}
 \sum_{l=1}^M\sum_{q\in\mathcal V_l}a_q
 \left[
 \sup_{x\in D(q)}\bigl(\mathsf R_l(x)+\mathsf R_{l+1}(x)\bigr)
 +\sum_{k=1}^l2^{k-l}\mathop{\rm ess\,sup}_{x\in D(q)}\mathsf T_k(x)
 \right]\le CV.
\end{equation}
\end{proposition}

\subsection{A positive measure on refinement paths}

A node at level $k$ is a positive-weight refinement path
$\xi=(q_1,\ldots,q_k)$, and it has weight, mass, and direction
\[
 b_\xi=\prod_{j=1}^{k-1}\lambda_{q_jq_{j+1}},\qquad
 M_\xi=b_\xi a_{q_k},\qquad\vartheta_\xi=\theta_{q_k}.
\]
Its vector is $M_\xi\vartheta_\xi=b_\xi h_{q_k}$.
These vectors add from children to parent, and hence
\begin{equation}\label{eq:gd-node-monotonicity}
 M_\xi\le\sum_{\eta\text{ child of }\xi}M_\eta.
\end{equation}
Use all level-one roots whose closed tent support meets the compact
support of the gradient measure, and all of their positive-weight paths
through level $M$.  This is a finite forest.  Support nesting and the
column normalization of $\lambda$ show that
\begin{equation}\label{eq:gd-path-normalization}
 \sum_{\substack{|\xi|=k\\q_k=q}}b_\xi=1\quad\text{if }a_q>0,
 \qquad\sum_{|\xi|=k}M_\xi=A_k.
\end{equation}
Indeed, every positive-weight prefix leading to such a vertex starts
in the retained set of roots.  If the gradient measure is zero there
is nothing to prove, and the forest may be taken empty.

Mark every root. Below a mark, mark the first nodes whose direction
differs from its direction by more than $\tau/8$, and repeat. These
marks depend only on the refinement forest.

\begin{lemma}
\label{lem:gd-particles}\label{lem:gd-intrinsic}
There is a positive measure $\mu$ on terminal paths with birth level
$B\in\{1,\ldots,M\}$ such that
\begin{align}
 \mu\{B\le k,\ (q_1,\ldots,q_k)=\xi\}&=M_\xi,
                                      \label{eq:gd-node-occupancy}\\
 \mu\{B\le k,\ q_k=q\}&=a_q,
 \qquad\mu(\mathrm{all\ particles})=A_M\le V.
                                      \label{eq:gd-vertex-occupancy}
\end{align}
Every such path, including its portion before birth, obeys
\begin{equation}\label{eq:gd-path-locality}
 |y-q_s|_\infty<r_s\quad(1\le s\le M),\qquad y=q_M.
\end{equation}
Moreover
\begin{align}
 \sum_{\xi\ \mathrm{marked}}M_\xi&\le CV,
                                      \label{eq:gd-mark-budget}\\
 \int\sum_{s=B}^M\mathcal E_s(q_s)\,d\mu&\le CV,
                                      \label{eq:gd-particle-discrepancy}\\
 \int\sum_{s=B}^M\mathsf R_s(y)\,d\mu&\le CV.
                                      \label{eq:gd-particle-switches}
\end{align}
\end{lemma}
\begin{proof}
At a node $\xi$, put $S_\xi=\sum_{\eta\text{ child}}M_\eta\ge M_\xi$.
Send the fraction $M_\eta/S_\xi$ of each existing particle type to
$\eta$, and create the missing mass
$M_\eta(1-M_\xi/S_\xi)$ there, with that child's complete prefix and
birth level. Start with mass $M_\xi$ at each root; when $S_\xi=0$
all masses are zero. This finite construction proves node occupancy;
\eqref{eq:gd-path-normalization} gives vertex occupancy. Positive
edges have length at most $r_j/2$, so
$|q_M-q_s|_\infty\le\sum_{j=s}^{M-1}r_j/2<r_s$.

For each nonterminal mark $\xi$, stop its descendants at their next
mark or at the terminal level. Vector refinement on the resulting
frontier $L_\xi$ gives
\begin{equation}\label{eq:gd-mark-frontier}
 \sum_{\eta\in L_\xi}M_\eta-M_\xi
 =\frac12\sum_{\eta\in L_\xi}M_\eta
       |\vartheta_\eta-\vartheta_\xi|^2
 \ge\frac{\tau^2}{128}
       \sum_{\substack{\eta\in L_\xi\\\eta\ \mathrm{marked}}}M_\eta.
\end{equation}
Each nonroot mark has one preceding marked frontier. Summing telescopes
to $A_M-A_1\le V$; adding the root masses proves the mark budget
(the case $M=1$ is immediate). Occupancy and
Lemma~\ref{lem:gd-neighbor} give the discrepancy budget. Occupancy also
identifies the integrated number of postbirth marks with their mass sum.

For a fixed particle, run the label law at $y$. In levels $B,\ldots,M$
call a level exceptional if it is marked or
$\mathcal E_s(q_s)>c_\ast$. Between exceptional levels the path directions
stay within $\tau/8$ of one marked direction. By locality, every
vertex sampled at $y$ is within $2r_s$ of $q_s$, so its $\Theta_z$
is within $\tau/10$ of $\theta_{q_s}$. Any two such sampled directions
therefore differ by at most $9\tau/20$. After the first switch, the
new label is within $\tau/2$ of all subsequent sampled directions;
the dead zone of~\eqref{eq:gd-switch-function} precludes a second
switch on that interval. The same holds for the initial assignment.
Thus, with $X_1=1$, $X_s=\mathbf1_{\{n_s\ne n_{s-1}\}}$,
\[
 \sum_{s=B}^M X_s
 \le1+2\#\{s\ge B:s\ \mathrm{marked}\}
       +2c_\ast^{-1}\sum_{s=B}^M\mathcal E_s(q_s).
\]
Taking expectation at $y$ and integrating the two budgets and total
particle mass proves~\eqref{eq:gd-particle-switches}.
\end{proof}

\subsection{Comparison at nearby observation points}

\begin{lemma}\label{lem:gd-spatial-comparison}
Fix a particle, $1\le l\le M$, and $x\in D(q_l)$, and put
$c_s=\mathsf R_s(y)+\mathcal E_s(q_s)$.  For
$1\le k\le\min(l+1,M)$ one has
\begin{equation}\label{eq:gd-spatial-comparison}
 \mathsf R_k(x)+\mathsf T_k(x)
 \le C\sum_{s=1}^k2^{s-k}c_s.
\end{equation}
The switch-rate assertion holds everywhere in $D(q_l)$, and the
row-derivative assertion holds almost everywhere there.
\end{lemma}
\begin{proof}
Locality gives $|x-y|_\infty<26r_l$. If $s\le l+1$, every tent
positive at $x$ or $y$ has
$|z-q_s|_\infty\le2r_s+26r_l\le54r_s$; the same holds for
nonzero tent gradients off grid hyperplanes. Thus all involved
$\Theta_z$ lie within $\sqrt{\mathcal E_s(q_s)}$ of $\theta_{q_s}$.
Since $\sqrt{P}$ is $\tau^{-1}$-Lipschitz,
\[
 p_s(n,z)\le2p_s(n,z')+8\tau^{-2}\mathcal E_s(q_s).
\]
Average in $z'$ with the tents at $y$. The row formulas, at most
$2^{d+1}$ contributing tents, and their $C/r_s$ Lipschitz bound give
\begin{align}
 \rho_s(x\mid n)&\le C\{\rho_s(y\mid n)+\mathcal E_s(q_s)\},
 \label{eq:gd-rate-comparison}\\
 r_s\sum_{n'}|\nabla_x\pi_s(x;n,n')|
 &\le C\{\rho_s(y\mid n)+\mathcal E_s(q_s)\}\quad\text{a.e.},
 \label{eq:gd-derivative-comparison}\\
 \|\pi_s(x;n,\cdot)-\pi_s(y;n,\cdot)\|_{\ell^1}
 &\le C2^{s-l}\{\rho_s(y\mid n)+\mathcal E_s(q_s)\}.
 \label{eq:gd-row-comparison}
\end{align}
The initial row has scaled derivative at most $C$ and difference
at most $C2^{1-l}$, paid by $c_1\ge1$.
Stochastic $\ell^1$ contraction and the marginal recursion imply,
with $d_s=\|\zeta_s(x)-\zeta_s(y)\|_1$,
\[
 d_s\le d_{s-1}+C2^{s-l}c_s,\qquad
 d_j\le C\sum_{s\le j}2^{s-l}c_s.
\]
Indeed the row error is averaged against $\zeta_{s-1}(y)$, for
which the average switch rate is $\mathsf R_s(y)$.
Average the first two row bounds at step $k$ against that same law,
then replace it by $\zeta_{k-1}(x)$. Uniform row bounds make the
replacement cost $Cd_{k-1}$, so
\[
 \mathsf R_k(x)+\mathsf T_k(x)
 \le Cc_k+C\sum_{s<k}2^{s-l}c_s
 \le C\sum_{s\le k}2^{s-k}c_s,
\]
where $k\le l+1$ was used. At $k=1$ the elementary bounds and
$c_1\ge1$ suffice. The pointwise/a.e. qualifications are inherited
from the switch-rate and derivative formulas, respectively.
\end{proof}

\begin{proof}[Proof of Proposition~\ref{prop:gd-packing}]
Write $C_l(q)$ for its nonnegative measurable bracket. Occupancy gives
\begin{equation}\label{eq:gd-cost-occupancy}
 \sum_{l,q}a_qC_l(q)=\int\sum_{l=B}^MC_l(q_l)\,d\mu.
\end{equation}
For a particle, Lemma~\ref{lem:gd-spatial-comparison} controls all
suprema in $C_l(q_l)$; the sentinel $\mathsf R_{M+1}=1$ costs
one unit. Its two convolutions combine by
\[
 \sum_{k\le l}2^{k-l}\sum_{s\le k}2^{s-k}c_s
 =\sum_{s\le l}(l-s+1)2^{s-l}c_s.
\]
Summing this geometric kernel over $l\ge B$, including the rate
at $l+1$, gives
\begin{equation}\label{eq:gd-particle-cost}
 \sum_{l=B}^MC_l(q_l)
 \le C\left[1+\sum_{s=B}^Mc_s+
       \sum_{s<B}(B-s+2)2^{s-B}c_s\right].
\end{equation}
Integrate using~\eqref{eq:gd-cost-occupancy}. Postbirth terms cost
$CV$ by the discrepancy and switching budgets; prebirth terms cost
$CV$ since $c_s\le5$, $\mu(\mathrm{all})\le V$ and
$\sum_{j\ge1}(j+2)2^{-j}<\infty$.
\end{proof}

\section{Signed finite-radius envelopes and global assembly}
\label{sec:globalassembly}
Retain the data of the preceding appendices and put $u_t=A_t^\kappa g$.
With $I_k,J_k,v_k$ as above, abbreviate
\[
 A_{k,n}=\sup_{t\in I_k(x)}\sum_i(-e_i(n)\cdot\nabla u_t)_+,
 \qquad B_k=\sup_{t\in J_k(x)}|\nabla u_t|.
\]
Proposition~\ref{glob:commonprices} says
\begin{align}
 \int_Q A_{k,n}
 &\le C\left(D_Q+w(1)\sum_{\substack{q\in\mathcal R(Q)\\(q,n)\,\mathrm{bad}}}a_q\right),
                                    \label{glob:assembly-localdirection}\\
 \int_Q(B_k+r_k^{-1}v_k)
 &\le C\left(D_Q+w(1)\sum_{q\in\mathcal R(Q)}a_q\right),
                                    \label{glob:assembly-localendpoint}
\end{align}
with $\sum_{k,Q}D_Q\le C_dWV$. Pricing cells for a root $q$ lie
in $D(q)$ with bounded multiplicity. Retain the packing bound
\begin{equation}\label{glob:assembly-inputpacking}
 \sum_{l,q}a_q\left[\sup_{D(q)}(\mathsf R_l+\mathsf R_{l+1})+
  \sum_{k\le l}2^{k-l}\operatorname*{ess\,sup}_{D(q)}\mathsf T_k\right]\le CV.
\end{equation}
Only transition rows, never a spatial direction field, are differentiated.

\begin{lemma}\label{lem:R7payments}
For $U_l=\mathsf R_l+\mathsf R_{l+1}
               +\sum_{k\le l}2^{k-l}\mathsf T_k$ one has
\[
 \int\sum_{k,n}\zeta_k(n)A_{k,n}
 +\int\sum_lU_l(B_l+r_l^{-1}v_l)\le CWV.
\]
\end{lemma}
\begin{proof}
One has $U_l\le C$ and
$\sum_{l,q}a_q\operatorname*{ess\,sup}_{D(q)}U_l\le CV$ by packing. Multiplying
\eqref{glob:assembly-localendpoint} by cell essential suprema of $U_l$ pays
its prices by $C\sum D_Q$ and its roots by $Cw(1)V$.
For the directional term use cell suprema of $\zeta_k(n)\le1$ in
\eqref{glob:assembly-localdirection}. Good-label exclusion kills
bad roots with $\mathcal E_k(q)\le c_*$; the rest cost at most
$Cc_*^{-1}w(1)\sum_{k,q}a_q\mathcal E_k(q)\le CWV$.
The finite palette and bounded root multiplicity affect only $C$.
\end{proof}

\begin{proposition}\label{glob:band}
There is $C=C_{d,\alpha_0}(C_D+C_B+B)W$,
independent of $a,b$ and $g$, such that for
$0<a\le b$ with $b/a$ a nonnegative integer power of two,
\[
 S_{a,b}g:=\max_{a\le t\le b}A_t^{\kappa} g,
 \qquad \int_{\R^d}|\nabla S_{a,b}g|\le C\int_{\R^d}|\nabla g|.
\]
The assertion is for real signed inputs.
\end{proposition}
\begin{proof}
In this proof, the generic constants are independent of $w,W$;
their local dependence is at most $C_{d,\alpha_0}(C_D+C_B+B)$.
We first prove the normalized estimate for $a=r_M$, $b=1$. For
$1\le i\le j\le M$ define
\[
 F_{ij}(x)=\max_{r_j\le t\le2r_i}u_t(x).
\]
These functions are compactly supported and Lipschitz, with Lipschitz
constant at most $\|\nabla g\|_\infty$. Outside one null set they are all
differentiable. If $t$ attains $F_{ij}(x)$ at such a point, then
\begin{equation}\label{glob:touching}
 \nabla F_{ij}(x)=\nabla u_t(x).
\end{equation}
Indeed, with $t$ held fixed, $F_{ij}-u_t\ge0$ everywhere and has value zero
at $x$. Its gradient there is zero. This includes ties between maximizers
and involves no differentiation of a selected radius.

A history $\boldsymbol n=(n_1,\ldots,n_M)\in\mathcal N^M$ has runs
$[i,j]$, namely the maximal constant-label intervals. Their radius intervals
$[r_j,2r_i]$ meet successively at endpoints and cover $[r_M,1]$.
Consequently $F_{1M}=\max_{[i,j]\text{ a run}}F_{ij}$. At a differentiability
point, a run attaining this maximum also has the same gradient. Apply
\eqref{glob:coneineq} to that run and add the nonnegative cone brackets of
all other runs. Averaging over histories gives
\begin{equation}\label{glob:runcone}
 |\nabla F_{1M}|\le C\mathbb E_x
 \sum_{[i,j]\text{ a run}}
 \left[n_i\cdot\nabla F_{ij}+
       \sum_{\ell=1}^d(-e_\ell(n_i)\cdot\nabla F_{ij})_+\right].
\end{equation}
All history sums are finite.

\emph{Unfavorable components and run endpoints.}
For a maximizing radius $t$ of a run $[i,j]$, take a bin
$[r_k,2r_k]$ containing it, where $i\le k\le j$. If $i<k<j$, the whole
comparison window $[t-r_k/4,t+r_k/4]$ lies in the run interval, so
$t\in I_k(x)$. If $k=i$ or $k=j$, the bin is contained in the run interval,
so $t\in J_k(x)$. The touching identity applies in either case. Summing
bounds over bins if necessary avoids any measurable selection of $t$.
Each bin belongs to one run, and the probabilities that it starts or ends
a run are $\mathsf R_k$ and $\mathsf R_{k+1}$. Thus
\begin{align}
 &\mathbb E_x\sum_{[i,j]\text{ a run}}\sum_{\ell=1}^d
          (-e_\ell(n_i)\cdot\nabla F_{ij})_+
 \notag\\
 &\hspace{1cm}\le\sum_{k,n}\zeta_k(x;n)A_{k,n}(x)
          +d\sum_k(\mathsf R_k+\mathsf R_{k+1})B_k(x).
 \label{glob:runnegative}
\end{align}
Lemma~\ref{lem:R7payments} integrates both terms, giving
\begin{equation}\label{glob:negativepaid}
 \int\mathbb E_x\sum_{[i,j]\text{ a run}}\sum_{\ell=1}^d
                (-e_\ell(n_i)\cdot\nabla F_{ij})_+\le CWV.
\end{equation}

\emph{Signed projections and baseline subtraction.}
A run beginning at $i$ has the forced endpoint $r_i\in J_i(x)$.
Thus Lemma~\ref{lem:R7payments} also gives
\begin{equation}\label{glob:baselinepaid}
 \int\mathbb E_x\sum_{[i,j]\text{ a run}}|\nabla u_{r_i}|
 =\sum_i\int\mathsf R_i|\nabla u_{r_i}|\le CWV.
\end{equation}
For a fixed history define its vector potential
\[
 \Phi_{\boldsymbol n}(x)=\sum_{[i,j]\text{ a run}}
                 n_i\bigl(F_{ij}(x)-u_{r_i}(x)\bigr).
\]
It is compactly supported and Lipschitz. Its divergence is the sum of the
signed projections in \eqref{glob:runcone} after baseline subtraction.
Regard the initial row as a transition from a fixed auxiliary state $*$.
The history probability is
\[
 w_{\boldsymbol n}(x)=\prod_{s=1}^M\pi_s(x;n_{s-1},n_s),\qquad n_0=*.
\]
The weak product rule and compact support give
\begin{equation}\label{glob:weightedibp}
 \sum_{\boldsymbol n}\int w_{\boldsymbol n}\operatorname{div}\Phi_{\boldsymbol n}
       =-\sum_{\boldsymbol n}\int\Phi_{\boldsymbol n}\cdot\nabla w_{\boldsymbol n}.
\end{equation}
Cancellation in the right-hand side must precede taking absolute values.

At the product-rule term $k$, fix a prefix
$\boldsymbol a=(n_1,\ldots,n_{k-1})$ and let
$J_{k,\boldsymbol n}=\prod_{s>k}\pi_s(x;n_{s-1},n_s)$.
Close the last prefix run at $k-1$ and call its vector potential
$\Psi_{\boldsymbol a}$; put $\Psi_\varnothing=0$.
It is independent of the continuation. Since future probabilities
sum to one and each differentiated row sums to zero,
\begin{equation}\label{glob:prefixcancel}
 \sum_{n_k,\ldots,n_M}J_{k,\boldsymbol n}\Phi_{\boldsymbol n}
           \cdot\nabla\pi_k
 =\sum_{n_k,\ldots,n_M}J_{k,\boldsymbol n}
    (\Phi_{\boldsymbol n}-\Psi_{\boldsymbol a})\cdot\nabla\pi_k.
\end{equation}
Completed prefix runs cancel without dividing transition
probabilities. If the last run continues, its baseline cancels and
only radii below $r_{k-1}=2r_k$ are added, costing $v_k$.
New runs have distinct starting levels $l\ge k$ and cost
$F_{lj}-u_{r_l}\le v_l$; when the prefix continues, $l>k$.
Including the empty prefix gives
\begin{equation}\label{glob:prefixoscillation}
 |\Phi_{\boldsymbol n}-\Psi_{\boldsymbol a}|\le\sum_{l=k}^M v_l.
\end{equation}

Now take absolute values in \eqref{glob:prefixcancel}. Future probabilities
are nonnegative and sum to one. Summing prefixes identifies their total
step-$k$ row-derivative cost as $\mathsf T_k/r_k$. Expanding the finite
product in \eqref{glob:weightedibp} thus yields the signed upper bound
\begin{equation}\label{glob:signedcost}
 \sum_{\boldsymbol n}\int w_{\boldsymbol n}\operatorname{div}\Phi_{\boldsymbol n}
       \le\sum_{1\le k\le l\le M}\int\frac{\mathsf T_k}{r_k}v_l.
\end{equation}
The right side is
$\int\sum_l(\sum_{k\le l}2^{k-l}\mathsf T_k)r_l^{-1}v_l\le CWV$
by Lemma~\ref{lem:R7payments}. Insert this bound,
\eqref{glob:baselinepaid} and~\eqref{glob:negativepaid} into
\eqref{glob:runcone} to obtain
$\int|\nabla F_{1M}|\le CWV$, uniformly in $M$.

For $b/a=2^M$, $M\ge1$, rescale by $g_b(x)=g(bx)$; its
normalized envelope is $S_{a,b}g(bx)$ and both gradient integrals
scale by $b^{1-d}$. For $a=b$, Tonelli and
$\nabla A_a^\kappa g=A_a^\kappa(\nabla g)$ give the bound directly.
\end{proof}

\section{BV localization and homogeneous extension}\label{glob:regularity}
We now use only the smooth dyadic-band bound of
Proposition~\ref{glob:band}, together with the standard BV product,
chain, Poincar\'e, structure and coarea theorems~\cite{AFP}. The signed
cutoff argument is the BV version of~\cite[Section~6]{OpenAI}. Write
\[
 S_{a,b}^{\kappa}g=\sup_{a\le r\le b}A_r^\kappa g,
 \qquad E_\rho=\{x:\dist(x,E)<\rho\}.
\]

\begin{lemma}\label{lem:R7bandextension}
Let $\kappa$ be a probability density supported in $B_R$ with
$\|\kappa\|_\infty\le B$. If all positive finite bands satisfy
$\|\nabla S_{a,b}^\kappa h\|_1\le C_0\|\nabla h\|_1$ for real
$h\in C_c^\infty$, then for real $h\in\mathcal X_d$ they belong to
$L^{q_d}\cap\operatorname{Lip}$ and obey
\begin{equation}\label{r5:allbands}
 \operatorname{Lip}S_{a,b}^\kappa h\le Ba^{-d}|Dh|(\R^d),\qquad
 \|\nabla S_{a,b}^\kappa h\|_1\le C_0|Dh|(\R^d).
\end{equation}
They are globally $W^{1,1}$ when $h\in L^1$.
Every $h\in\mathcal X_d$ satisfies
$\|h(\cdot+v)-h\|_1\le|v||Dh|(\mathbb R^d)$ for $v\in\mathbb R^d$.
If $\supp\kappa\subset[-R_\infty,R_\infty]^d$, $R_\infty>0$, a cube $Q$ of side
$\ell$ and real $h\in BV(5Q)$ obey
\begin{equation}\label{r7:whitneyconstantcost}
 \int_Q|\nabla S_{a,b}^\kappa h|
 \le C_dC_0|Dh|(5Q),
 \qquad 0<a\le b\le\ell/R_\infty.
\end{equation}
For $f\in\mathcal X_d$,
\begin{equation}\label{r5:BVbound}
 M_\kappa f\in\mathcal X_d,\qquad
 |D(M_\kappa f)|(\mathbb R^d)\le C_0|D|f||(\mathbb R^d).
\end{equation}
\end{lemma}
\begin{proof}
Put $q=q_d$. Since
$T^{-1}\int_{B_{2T}\setminus B_T}|h|
 \le C_d\|h\|_{L^q(B_{2T}\setminus B_T)}\to0$,
cutoff and mollification give
\begin{equation}\label{r5:strictapprox}
 h_j\in C_c^\infty,\quad h_j\to h\text{ in }L^q,\quad
 \limsup_j\|\nabla h_j\|_1\le|Dh|(\R^d),
\end{equation}
also in $L^1$ when $h\in L^1$. Fatou applied to
$\|h_j(\cdot+v)-h_j\|_1\le|v|\|\nabla h_j\|_1$ gives the translation
bound, without global $L^1$ integrability of $h$. H\"older gives
$\|S_{a,b}^\kappa h_j-S_{a,b}^\kappa h\|_\infty
 \le(Ba^{-d})^{1/q}\|h_j-h\|_q$.
Lower semicontinuity gives the $C_0$ variation bound. The
translation estimate directly gives
$|S_{a,b}^\kappa h(x+v)-S_{a,b}^\kappa h(x)|
 \le Ba^{-d}|v||Dh|(\mathbb R^d)$, proving the Lipschitz bound.
Ball domination gives $L^q$ membership; for $h\in L^1$, the
integrable majorant $Ba^{-d}\int_{B(x,bR)}|h|$ gives global $W^{1,1}$.
For $g=|f|$, $G_n=S_{2^{-n},2^n}^\kappa g$ increases to
$M_\kappa f$ in $L^1_{\rm loc}$, dominated by
$B\omega_dR^dM_{B_1}g\in L^q$. Thus, with $H=D(M_\kappa f)$,
\begin{equation}\label{glob:BVlimit}
 |H|(\R^d)\le C_0|Dg|(\R^d),\qquad
 |H|(O)\le\liminf_n\int_O|\nabla G_n|\quad(O\text{ open}).
\end{equation}
For the local bound put $c_Q=|5Q|^{-1}\int_{5Q}h$ and
$h_Q=\chi_Q(h-c_Q)$, where
$0\le\chi_Q\le1$, $\chi_Q=1$ on $3Q$,
$\supp\chi_Q\Subset5Q$ and $|\nabla\chi_Q|\le C_d/\ell$.
Product and Poincar\'e give
\begin{equation}\label{glob:cutoffcost}
 |Dh_Q|(\mathbb R^d)\le C_d|Dh|(5Q).
\end{equation}
On $Q$, the support condition gives
$S_{a,b}^\kappa h_Q=S_{a,b}^\kappa h-c_Q$; apply the global
signed-band bound to $h_Q\in BV\cap L^1$.
No singular-derivative statement is used.
\end{proof}

\begin{proposition}\label{r5:BVtransfer}\label{glob:sobolevbands}\label{glob:endpoint}
Let $d\ge2$ and let $\kappa$ be a bounded probability density
supported in $[-1,1]^d$, with $B=\|\kappa\|_\infty$. Suppose that for every real $g\in C_c^\infty(\R^d)$ and every
dyadic-ratio band $[a,b]$,
\[
 \int|\nabla S_{a,b}^{\kappa} g|\le C_*\int|\nabla g|,
\]
with $C_*$ independent of the band. All conclusions of
Lemma~\ref{lem:R7bandextension} hold with $C_0=C_*+4^dB$.
For $f\in BV\cap L^1$ and $g=|f|$, one also has, for every
Borel $E$ of finite volume and $t>0$,
\begin{equation}\label{r5:localmeasure}
 |D(M_\kappa f)|(E)
 \le C|Dg|(E_{C_dt})+B|E|t^{-d}|Dg|(\R^d),
 \qquad C_d=1+3\sqrt d,
\end{equation}
and, as Radon measures,
\begin{equation}\label{r5:singularinheritance}
 |D^s(M_\kappa f)|\le C|D^sg|\le C|D^sf|,
\end{equation}
where $C$ depends only on $d,B,C_*$.
\end{proposition}
\begin{proof}
For smooth $g$, choose $c=a2^N\le b<2c$. Touching gives
\[
 |\nabla S_{c,b}^\kappa g|
 \le Bc^{-d}\int_{x+[-2c,2c]^d}|\nabla g|.
\]
Its integral is at most $4^dB\|\nabla g\|_1$. The maximum rule
for $S_{a,b}=\max(S_{a,c},S_{c,b})$ therefore extends the dyadic
hypothesis to all bands with $C_0=C_*+4^dB$ (if $N=0$, only the
short band is needed). Lemma~\ref{lem:R7bandextension} proves
\eqref{r5:allbands}--\eqref{r5:BVbound}, since $BV\cap L^1\subset\mathcal X_d$.
Retain $g=|f|$, $G_n=S_{2^{-n},2^n}^\kappa g$ and $H=D(M_\kappa f)$
as in that lemma.

\emph{Localization.}
Fix $t>0$ and a Borel set $E$ of finite volume. Choose open sets
$O_j$ with $E\subset O_j\subset E_t$ and $|O_j|\downarrow|E|$. For $O=O_j$ and large $n$, split the band at $t$:
\begin{equation}\label{glob:nullsplit}
 \int_O|\nabla G_n|\le\int_O|\nabla S_{2^{-n},t}^\kappa g|
       +B|O|t^{-d}|Dg|(\R^d).
\end{equation}
Partition into cubes $Q$ of side $t$. Since
$\supp\kappa\subset[-1,1]^d$, the common cube bound
\eqref{r7:whitneyconstantcost}, with $R_\infty=1$, gives
\[
 \int_O|\nabla S_{2^{-n},t}^\kappa g|
 \le C\sum_{Q\cap O\ne\varnothing}|Dg|(5Q)
 \le C|Dg|(E_{C_dt}).
\]
Here $5Q$ have bounded overlap and lie in
$E_{(1+3\sqrt d)t}$.
The nonnegative sums also cover unbounded $O$. Lower semicontinuity
\eqref{glob:BVlimit} in~\eqref{glob:nullsplit}, followed by
$|H|(E)\le|H|(O_j)$ and $j\to\infty$, proves
\eqref{r5:localmeasure}.

\emph{The singular derivative.}
For compact null $E$, let $t\downarrow0$ in
\eqref{r5:localmeasure} to obtain $|H|(E)\le C|Dg|(E)$.
Inner regularity extends this to Borel null sets; restricting to
a null support of $H^s$ and applying the BV chain rule gives
\eqref{r5:singularinheritance}.
\end{proof}

The BV coarea formula also gives, for $f\in BV\cap L^1$,
\[
 \int_0^\infty P(\{M_\kappa f>t\})\,dt
       =|D(M_\kappa f)|(\R^d)\le C|Df|(\R^d).
\]
The Sobolev and SBV consequences are included in
Theorem~\ref{r10:kerneltransfer}; no additional endpoint proof is needed.

\begin{corollary}\label{r5:homogeneousBV}
All conclusions of Proposition~\ref{r5:BVtransfer} concerning
$M_\kappa f$ hold for $f\in\mathcal X_d$. Signed finite bands lie in
$L^{q_d}\cap\operatorname{Lip}$ with the same integrable-gradient bound;
they belong globally to $W^{1,1}$ when $f\in L^1$.
\end{corollary}
\begin{proof}
Apply Lemma~\ref{lem:R7bandextension} with $C_0=C_*+4^dB$.
Its cube estimate is local in the input, so the localization and
null-set argument above apply unchanged on $\mathcal X_d$.
\end{proof}

\section{Kernel certificates and geometric laws}\label{r15:kernelcalculus}
We verify the intrinsic price and spatial-tail hypotheses for angular
kernels and geometric probability laws. A law may regularize boundaries
even when its components have infinite intrinsic price.

\subsection{Angular jumps}\label{r15:angularsection}
Let $a\ge0$ belong to $BV(S^{d-1})\cap L^\infty$ and let $b\ge0$
be continuous, zero on $[1,\infty)$, with
\[
 \left(\int_{S^{d-1}}a\,d\mathcal H^{d-1}\right)
 \left(\int_0^1b(t)t^{d-1}dt\right)=1,
 \qquad\kappa(z)=a(z/|z|)b(|z|)\quad(z\ne0).
\]
Write $A_0=\|a\|_\infty$,
$A_1=\|a\|_{L^1(S^{d-1})}+|D_Sa|(S^{d-1})$,
$b_0=\|b\|_\infty$ and $\mathcal D_b=\int_0^1\omega_b(s)ds/s$.
\begin{theorem}\label{r15:angularendpoint}
If $\mathcal D_b<\infty$, then
\begin{equation}\label{r15:angularnorm}
 \mathfrak K(\kappa)\le
 C_d\{A_0(b_0+\mathcal D_b)+b_0A_1\}.
\end{equation}
Hence the bounded-kernel endpoint and local transfer apply, allowing
angular jump discontinuities.
\end{theorem}
\begin{proof}
A radial cutoff of $a(x/|x|)$, equal to it on $B_4$ and supported
in $B_5$, has variation $\le C_dA_1$ because
$\int_0^5t^{d-2}dt<\infty$. Its BV translation bound handles the
angular term, which is independent of $r$. The radial term is at most
$A_0\omega_b(s)$ on $B_3$, uniformly for $1\le r\le2$. Hence
$\mathfrak o_\kappa(s)\le C_d(b_0A_1s+A_0\omega_b(s))$ for $s\le1$;
larger dyadic arguments cost $C_dA_0b_0$. Sum these bounds.
\end{proof}
\begin{theorem}\label{r15:angulartail}
If additionally $b$ is Lipschitz, $L_b=\|b'\|_\infty$, then for
every $h\ge0$ in $W^{1,1}$ and $L>0$, an all-large-radius envelope $G_{\kappa,h,L}$ controls
$|\nabla A_r^\kappa h(x)|$ for $r\ge|x|/4$ at almost every $|x|>L$,
and
\begin{equation}\label{r15:angulartailbound}
 \int_{|x|>L}G_{\kappa,h,L}\le
 C_d\{A_0L_b+b_0|D_Sa|(S^{d-1})\}\|h\|_1/L.
\end{equation}
The resulting joint global gradient continuity is uniform under the
displayed bounds, also for jointly measurable probability mixtures with
bounded expectation of $A_0(b_0+L_b)+b_0A_1$.
\end{theorem}
\begin{proof}
The polar derivative consists of the radial density
$a(\theta)b'(t)\theta$ and an angular measure of variation at most
$b(t)t^{d-2}dt\,d|D_Sa|(\theta)$. There are no boundary masses at
$t=1$ since $b(1)=0$, or at zero since their mass is $O(t^{d-1})$.
The radial convolution is bounded by the $\sigma=0$ potential
in~\eqref{r11:cuspQintegral}, with coefficient $A_0L_b$. For the angular part, the substitution $t=r\varrho$ in
the convolution formula gives
\[
 r^{-d}\int_{S^{d-1}}\int_0^\infty
 h(x+t\theta)b(t/r)t^{d-2}dt\,d|D_Sa|(\theta).
\]
Nonzero terms require $r\ge t$. For $r\ge|x|/4$ this is bounded by
\[
 P_h(x)=C_db_0\int_{S^{d-1}}\int_0^\infty
 h(x+t\theta)t^{d-2}\max\{|x|,t\}^{-d}dt\,d|D_Sa|(\theta).
\]
For fixed $y,\theta$,
\begin{equation}\label{r15:angularlineintegral}
 \int_0^\infty\1_{\{|y-t\theta|>L\}}
 t^{d-2}\max\{|y-t\theta|,t\}^{-d}dt\le C_d/L.
\end{equation}
On this domain $\max\{|y-t\theta|,t\}\ge\max\{L,|y|/2,t\}$;
integrating $t^{d-2}\max\{L,|y|/2,t\}^{-d}$ proves the bound.
Tonelli with $y=x+t\theta$ gives the angular contribution in
\eqref{r15:angulartailbound}. Distributional integration by parts and
Fubini give the bounds at rational radii; radius continuity extends them
to all eligible radii at almost every center. Add the radial envelope
and use the endpoint/tail transfer. For mixtures apply
Corollary~\ref{r12:integratedtailmixtures}, with density bounded by
the expected $A_0b_0$.
\end{proof}
This includes $a=\1_E$ for positive-measure spherical sets of finite
perimeter; $E=-E$ gives an even kernel. The global conclusion requires
the Lipschitz radial factor. Under uniform supremum bounds, $L^1$
convergence of the angular and weighted radial factors gives kernel
$L^1$ convergence.

\begin{corollary}\label{r15:weakKernelLaw}
For an $L^1$-continuous probability family $t\mapsto\kappa_t$ on
compact metric $T$, with common support and bounded intrinsic price,
$\pi_n\rightharpoonup\pi$ and $f_n\to f$ in $W^{1,1}$ give maximal-value
convergence in $L^{d/(d-1)}$ and gradient convergence in $L^1_{\rm loc}$,
globally if the mixtures have a common spatial-tail certificate.
\end{corollary}
\begin{proof}
Apply Lemma~\ref{lem:R5weaklaws}, Corollary~\ref{r15:intrinsicmixtures}
and Theorem~\ref{r10:kerneltransfer}.
\end{proof}
For angular kernels with a fixed Lipschitz radial factor, rotations
are $L^1$-continuous and preserve the angular $BV$ and supremum bounds.
Theorem~\ref{r15:angulartail} therefore gives the global conclusion for
arbitrary laws on $SO(d)$.

\subsection{Euler domination and convolution}\label{r15:Eulersection}
\begin{proposition}\label{r15:exactEulerMeasure}
For a compactly supported $BV$ probability density $\kappa$, define
$\mathsf E_D\kappa=(D-d)\kappa\,dz-z\cdot D\kappa$.
For $\rho>0$, domination $\mathsf E_D\kappa\ge\rho|D\kappa|$ is
equivalent to the inequality
\[
 \rho r|\nabla_x A_r^\kappa h(x)|
 \le r\partial_r A_r^\kappa h(x)+D A_r^\kappa h(x)
\]
for every nonnegative $h\in C_c^\infty$, $x\in\R^d$ and $r>0$.
Probability mixtures of a measurable family with common compact support
and common $\rho,D$ preserve this condition.
\end{proposition}
\begin{proof}
Integration by parts identifies the two sides with
$-\int h(x+rz)dD\kappa$ and
$\int h(x+rz)d\mathsf E_D\kappa$.
This proves sufficiency. For necessity, set $x=0,r=1$ and resolve
the norm against a countable dense set of unit vectors on one dominating
measure. A cutoff equal to one on the common support gives
$\mathsf E_D\kappa(\R^d)=D$, hence $|D\kappa|(\R^d)\le D/\rho$.
This uniform BV bound justifies differentiation under probability
integration; linearity of $\mathsf E_D$ and
$|D\int\kappa_t d\pi|\le\int|D\kappa_t|d\pi$ prove preservation.
\end{proof}
\begin{corollary}\label{r15:EulerConvolution}
If $\mathsf E_{D_i}\kappa_i\ge\rho_i|D\kappa_i|$, $i=1,2$, then
\[
 \mathsf E_{D_1+D_2}(\kappa_1*\kappa_2)
             \ge(\rho_1+\rho_2)|D(\kappa_1*\kappa_2)|.
\]
Convolution of $\kappa_1$ with a probability measure supported in $B_s$,
$s<\rho_1$, preserves $D_1$ and replaces $\rho_1$ by $\rho_1-s$.
\end{corollary}
\begin{proof}
The distributional convolution rule gives
$\mathsf E_{D_1+D_2}(\kappa_1*\kappa_2)
=\mathsf E_{D_1}\kappa_1*\kappa_2+
 \kappa_1*\mathsf E_{D_2}\kappa_2$.
Both $|D\kappa_1|*\kappa_2$ and $\kappa_1*|D\kappa_2|$
dominate $|D(\kappa_1*\kappa_2)|$.
For $\eta$ supported in $B_s$, the additional term from $z=y+t$
is $-\int t\cdot D\kappa_1(z-t)d\eta(t)$, whose variation is at
most $s(|D\kappa_1|*\eta)$.
\end{proof}
Euler domination is a tail certificate, not an endpoint hypothesis.
For smooth kernels it reads
$(D-d)\kappa-z\cdot\nabla\kappa\ge\rho|\nabla\kappa|$;
for $|K|^{-1}\1_K$ and $D=d$ it is $z\cdot n_K\ge\rho$.
\begin{corollary}\label{r15:mixedpriceEuler}
Jointly measurable probability mixtures with common support, finite
expected intrinsic price and common Euler data $(\rho,D)$ satisfy the
endpoint and global gradient continuity, uniformly under common data
and $L^1$ kernel/$W^{1,1}$ input convergence.
\end{corollary}
\begin{proof}
Use Corollary~\ref{r15:intrinsicmixtures} and the Euler stationary-radius
certificate in Theorem~\ref{r10:kerneltransfer}.
\end{proof}

\begin{corollary}\label{r7:powerprofiles}
For every real $\alpha>0$,
$\phi_\alpha(w)=c_{d,\alpha}(1-|w|^2)_+^\alpha$ has all
conclusions of Theorem~\ref{r10:kerneltransfer}, including global
gradient continuity, with constants depending on $d,\alpha$.
\end{corollary}
\begin{proof}
It belongs to $C_c^{0,\min(1,\alpha)}\cap W^{1,1}$ and
\[
 \nabla\phi_\alpha=-2\alpha c_{d,\alpha}w
       (1-|w|^2)^{\alpha-1}\1_{B_1},\qquad
 \mathsf E_{d+2\alpha}\phi_\alpha
 =2\alpha c_{d,\alpha}(1-|w|^2)^{\alpha-1}\1_{B_1}\,dw
 \ge|D\phi_\alpha|.
\]
There is no boundary measure. Apply the Dini endpoint and Euler tail
certificate with $\rho=1$; their data are uniform on compact
$\alpha$-intervals in $(0,\infty)$.
\end{proof}

\subsection{Boundary incidence}\label{r15:randomshapesection}
For a measurable probability family of open convex bodies
$B_\rho\subset K_t\subset B_R$, define
\begin{equation}\label{r15:randomshapekernel}
 \kappa_\pi(x)=\int\frac{\1_{K_t}(x)}{|K_t|}d\pi(t),\qquad
 Q_\pi(h)=\sup_x\pi\{t:\operatorname{dist}(x,\partial K_t)\le h\},
 \quad \mathfrak D_\pi=\int_0^1Q_\pi(h)\frac{dh}{h}.
\end{equation}
The membership and distance maps are assumed measurable.
\begin{theorem}\label{r15:randomshapemother}
If $\mathfrak D_\pi<\infty$, the mixture is continuous and has all
bounded-kernel conclusions, including global gradient continuity, with
constant $C_{d,\rho,R}(1+\mathfrak D_\pi)$. Common data give joint global
continuity under $L^1$ kernel and $W^{1,1}$ input convergence.
\end{theorem}
\begin{proof}
With $B_0=(\omega_d\rho^d)^{-1}$, differing membership at $x,x+v$
forces their segment to meet the boundary, so
\begin{equation}\label{r15:randomshapemodulus}
 |\kappa_\pi(x+v)-\kappa_\pi(x)|\le B_0Q_\pi(h)
                    \qquad(|v|\le h).
\end{equation}
Thus $Q_\pi(h)\to0$, $\|\kappa_\pi\|_\infty\le B_0$ and its Dini
integral is at most $B_0\mathfrak D_\pi$. Convexity of each body gives
\begin{equation}\label{r15:randomshapeEuler}
 \rho r|\nabla A_r^{\kappa_\pi}h|
        \le r\partial_r A_r^{\kappa_\pi}h+d A_r^{\kappa_\pi}h.
\end{equation}
Apply the Dini endpoint and global transfer with stationary constant
$dB_0/\rho$.
\end{proof}
\begin{corollary}\label{r15:randomboxes}\label{r15:randomrotations}\label{r15:randomellipsoids}
The boundary-incidence conclusions hold for the following laws.
\begin{enumerate}[label=\textup{(\roman*)}]
\item For a joint law $a\in[\rho,R]^d$, set
$K_a=\prod_i(-a_i,a_i)$ and
$Q_i(h)=\sup_u\mathbb P\{a_i\in[u,u+h]\}$.
If each $Q_i$ is Dini, the constant is
$C_{d,\rho,R}(1+\sum_i\int_0^1Q_i(h)dh/h)$.
Neither independence nor conditional concentration is required.
\item Let $K$ be a polytope with at most $N$ facets and common balls
$B_\rho\subset K\subset B_R$. If the law of $U\in SO(d)$ is bounded
by $H$ times normalized Haar measure, the kernel
$\mathbb E[\1_{UK}/|K|]$ has variation constant
$C_{d,\rho,R}(1+\log(2+HN))$. More precisely, with
$\beta_2=1/2$ and $\beta_d=1$ for $d\ge3$,
\begin{equation}\label{r15:rotationincidence}
 Q_\pi(h)\le\min\{1,C_dHN(h/\rho)^{\beta_d}\}
                    \qquad(0<h\le\rho/2).
\end{equation}
Haar averaging is radial; bounded nonuniform orientation densities
need not be radial.
\item Let $a\in[m,M]^d$, $0<m<M$, have an arbitrary joint law.
Suppose every regular conditional law satisfies
$\mathbb P\{a_i\in[u,u+h]\mid a_j,j\ne i\}\le Q(h)$ almost surely,
where $Q$ is nondecreasing and Dini. Then the normalized mixture of
$E_a=\{\sum_i a_ix_i^2<1\}$ has the boundary-incidence conclusions
uniformly in $d,m,M$ and the Dini data of $Q$.
The condition may be imposed on rational intervals, using the
right-continuous envelope of $Q$.
\end{enumerate}
\end{corollary}
\begin{proof}
It suffices to bound $Q_\pi$ in each case.
For \textup{(i)}, a boundary point lies on a face, so boundary distance at most $h$
forces $\big||x_i|-a_i\big|\le h$ for some $i$. Hence
$Q_\pi(h)\le\sum_iQ_i(2h)$; the outer radius is $\sqrt d R$.

For \textup{(ii)}, write $K=\bigcap_{j\le N'}\{z:z\cdot n_j<b_j\}$, with
$b_j\ge\rho$ and $N'\le N$. Boundary distance $\le h$ implies
$|x\cdot Un_j-b_j|\le h$ for some $j$; the event is empty for
$|x|<\rho/2$. Otherwise $x\cdot Un_j/|x|$ has density
$\pi^{-1}(1-t^2)^{-1/2}$ in $d=2$, and a bounded multiple of
$(1-t^2)^{(d-3)/2}$ in $d\ge3$. Intervals of length $4h/\rho$
therefore have mass $\le C_d(h/\rho)^{\beta_d}$.
Multiply by $HN$ and integrate using
$\int_0^{h_0}\min(1,Ah^\beta)dh/h
\le\beta^{-1}(1+\log_+(Ah_0^\beta))$.

For \textup{(iii)}, the common radii are $M^{-1/2},m^{-1/2}$. Boundary distance at most
$h\le1$ implies $|\sum_i a_ix_i^2-1|\le C_{m,M}h$.
For small $h$ this forces $|x|^2\ge1/(2M)$; a largest coordinate
has $x_i^2\ge1/(2Md)$. Conditioning on the others restricts $a_i$
to an interval of length $4MdC_{m,M}h$. Thus
$Q_\pi(h)\le Q(4MdC_{m,M}h)$ for small $h$.
\end{proof}

In case \textup{(i)}, for example, $a_i=b_i+c_is$ with $c_i\ne0$ inherits a Frostman law
of $s$ when the lengths stay in the fixed interval. The nonnested,
nonhomothetic family $(a_1,a_2)=(1+s,2-s)$, $0\le s\le1$, allows a
singular Cantor Frostman law.

\paragraph{The Haar square.}\label{r15:explicitrotationprofile}
For $P=(-1,1)^2$, its radial Haar kernel is
\[
 \kappa_{P,\mathrm{Haar}}(z)=
 \begin{cases}
  1/4,&|z|\le1,\\
  1/4-\pi^{-1}\arccos(1/|z|),&1<|z|<\sqrt2,\\
  0,&|z|\ge\sqrt2.
 \end{cases}
\]
The uniform angle's distance to the nearest axis is uniform on
$[0,\pi/4]$ and must exceed $\arccos(1/|z|)$ for membership. The inner
decrement is comparable to $\sqrt{|z|-1}$, proving sharpness of $1/2$.

\begin{remark}\label{r15:randomshapelimit}
For compact metric $T$, Hausdorff continuity of $t\mapsto K_t$ implies
$L^1$ continuity of $t\mapsto\1_{K_t}/|K_t|$.
Lemma~\ref{lem:R5weaklaws} gives kernel convergence under weak law
convergence; common finite incidence-Dini bounds then give joint global
continuity. A point mass at a ball fails the incidence criterion.
\end{remark}

\subsection{Scale densities and a common tail}\label{r15:scaledensitysection}
\begin{theorem}\label{r15:scaledensitymother}
For a bounded open convex $K$ containing zero, let
$d\pi(\lambda)=a(\lambda)d\lambda$ be a probability on $(0,1]$.
Extend $q(\lambda)=\lambda^{-d}a(\lambda)$ by zero and put
\[
 Z=\int_0^1q,\quad
 H_j=\operatorname*{ess\,sup}_{2^{-j-1}<\lambda\le2^{-j}}q(\lambda),
 \quad S=\sum_{j\ge0}2^{-jd}H_j.
\]
If $Z+S<\infty$, the kernel
$\kappa_\pi=|K|^{-1}\int_{p(w)}^1q(\lambda)d\lambda$ on $K$,
zero outside, is a bounded continuous $W^{1,1}$ probability density.
Its bounded-kernel conclusions have constant $C_{d,K}(Z+S)$;
for support radius $R$ and $h\ge0$ in $W^{1,1}$,
\begin{equation}\label{r15:scaledensitytail}
 \int_{|x|>L}\sup_{r\ge|x|/(4R)}
 |\nabla A_r^{\kappa_\pi}h(x)|dx
                 \le C_{d,K}S\|h\|_1/L.
\end{equation}
In particular common $Z+S$ bounds give joint global gradient continuity
under weak scale-law and $W^{1,1}$ input convergence. Up to dimensional
factors, $S$ is the sum of the essential suprema of $a$ on these bands.
\end{theorem}
\begin{proof}
The tail formula gives height $Z/|K|$ and unit mass; truncation and
gauge-polar integration justify
\[
 \nabla\kappa_\pi=-|K|^{-1}q(p)\nabla p\quad\text{on }K,
 \qquad \int_0^1t^{d-1}q(t)dt\le Z.
\]
There is no boundary measure. Taking $F=q$ in
\eqref{eq:R13Hardydyadic} gives, for fixed $C>0$,
\begin{equation}\label{r15:scaledensityHardy}
 H(1)+\int_0^CH(t)\,\frac{dt}{t^2}\le C_{d,C}S.
\end{equation}
Apply Lemma~\ref{lem:R7Hardyenvelope} with
$F=|K|^{-1}\operatorname{Lip}(p)q$, then
Corollary~\ref{r15:bandgradientendpoint} and the transfer theorem.
For weak-law continuity, $\pi((0,\delta])\le Z\delta^d$ permits
Lemma~\ref{lem:R5weaklaws} for the norm-one layer map.
\end{proof}

\begin{corollary}\label{r15:logarithmicnondini}\label{r15:logarithmicinnerscales}
Normalize $d\pi=c_{d,p_0}\lambda^{d-1}[\log(e/\lambda)]^{-p_0}d\lambda$,
$p_0>1$. Its kernel is
\begin{equation}\label{r15:logarithmickernel}
 \kappa_{p_0}(w)=\frac{c_{d,p_0}}{|K|(p_0-1)}
 \left(1-[\log(e/p_K(w))]^{1-p_0}\right)
                      \1_{\{0<p_K(w)<1\}},
\end{equation}
with continuous central value $c_{d,p_0}/(|K|(p_0-1))$.
It has the preceding full conclusions, uniformly for $p_0$ in compact
subintervals of $(1,\infty)$. It has no positive H\"older exponent at
zero, every moment $\int\lambda^{-(d+t)}d\pi$, $t>0$, diverges,
and its sup-norm Dini integral is finite exactly for $p_0>2$.
\end{corollary}
\begin{proof}
$q=c_{d,p_0}/(\lambda[\log(e/\lambda)]^{p_0})$ gives
$Z=c_{d,p_0}/(p_0-1)$ and
$H_j\le C_{d,p_0}2^j(j+1)^{-p_0}$, hence $S<\infty$.
The estimates and weak-law convergence are locally uniform in $p_0$
by dominated convergence. Integration gives the formula and central
decrement $\asymp\log(e/|w|)^{1-p_0}$. Near zero $q$ decreases and
away from zero it is bounded, so its mass on a length-$h$ interval is
$\le C\log(e/h)^{1-p_0}+Ch$. Gauge Lipschitzness gives the matching
modulus. Finally $s=\log(e/\lambda)$ turns every extra negative moment
into an exponentially growing integral.
\end{proof}

\section{Singular kernels and logarithmic scale prices}\label{sec:R8singular}
\label{r21:singularkernelsection}
For unbounded kernels, radial majorant mass replaces the density norm.
A logarithmic moment controls coarse-radius gradients and supplies
absolute continuity, which lower semicontinuity alone cannot pass.
Normalize support to $\overline B_1$ and assume
\begin{equation}\label{r21:radialkernelmajorant}
 \begin{gathered}
 \kappa\ge0,\quad\int\kappa=1,\quad
 \kappa(x)\le\psi(|x|),\quad \psi\ge0\text{ nonincreasing},\\
 \psi=0\text{ on }(1,\infty),\qquad
 A=\int\psi(|x|)dx<\infty.
 \end{gathered}
\end{equation}
Allow $\psi(0)=\infty$ and set
\begin{equation}\label{r21:singularprices}
 P_\kappa=A+\sum_{j\ge0}\mathfrak o_\kappa(C_s2^{-j}),\qquad
 J_\psi=\int\psi(|x|)\log_+(1/|x|)dx.
\end{equation}
There is no density norm in these prices. A common support dilation
only reparameterizes the full radius envelope.

\begin{lemma}\label{r21:densityfreebands}
For bounded $\kappa$ satisfying \eqref{r21:radialkernelmajorant} and
$P_\kappa<\infty$, real $g\in C_c^\infty$ satisfies
\begin{equation}\label{r21:densityfreebandbound}
 \int|\nabla S_{a,b}^\kappa g|
                \le C_dP_\kappa\int|\nabla g|\qquad(0<a\le b<\infty),
\end{equation}
independently of $\|\kappa\|_\infty$.
\end{lemma}
\begin{proof}
Layer cake gives $M_\kappa F\le A M_{\rm ball}|F|$, and
$N_\kappa\le\int\psi(|x|/2)dx=2^dA$.
Apply Lemma~\ref{lem:R10bandmass}; its constant is $C_dP_\kappa$.
\end{proof}

\begin{theorem}\label{r21:singularkernelmother}
Assume \eqref{r21:radialkernelmajorant} and $P_\kappa<\infty$.
For real $g\in\mathcal X_d$,
\begin{equation}\label{r21:unboundedbandendpoint}
 S_{a,b}^\kappa g\in W^{1,1}_{\rm loc},\qquad
 \|\nabla S_{a,b}^\kappa g\|_1\le C_dP_\kappa|Dg|(\R^d).
\end{equation}
For $g\in BV\cap L^1$ this membership is global. Moreover
\begin{equation}\label{r21:unboundedBVendpoint}
 \begin{gathered}
 S_{0,\infty}^\kappa g,\ M_\kappa g\in\mathcal X_d,\\
 |DS_{0,\infty}^\kappa g|(\R^d)\le C_dP_\kappa|Dg|(\R^d),\qquad
 |D(M_\kappa g)|(\R^d)\le C_dP_\kappa|D|g||(\R^d).
 \end{gathered}
\end{equation}
At almost every center the radius field is continuous on positive bands,
so ordinary and rational suprema agree; values are understood almost everywhere.

If $J_\psi<\infty$, put
$V_t(z)=\sup_{r\ge t}r^{-d}\psi(|z|/r)$.
Then $V_t*|Dg|\in L^1_{\rm loc}$ and, for $u=M_\kappa g$,
\begin{equation}\label{r21:singularlocalization}
 |Du|(E)\le C_dP_\kappa|D|g||(E_{c_dt})
             +\int_E(V_t*|D|g||)(x)dx
\end{equation}
for every finite-volume Borel $E$ and $t>0$, permitting infinite right
side. Consequently
\begin{equation}\label{r21:singularkernelinheritance}
 |D^s(M_\kappa g)|\le C_dP_\kappa|D^s|g||
                                      \le C_dP_\kappa|D^sg|,
\end{equation}
and if $D^sg=0$,
\begin{equation}\label{r21:singularkernelSobolev}
 M_\kappa g\in W^{1,1}_{\rm loc}\cap L^{d/(d-1)},\qquad
 \|\nabla M_\kappa g\|_1\le C_dP_\kappa\|\nabla g\|_1.
\end{equation}
The signed full envelope has the same localization and inheritance,
with $|Dg|$ replacing $|D|g||$.
\end{theorem}
\begin{proof}
Truncate $h_m=\min(\kappa,m)$, $p_m=\int h_m\uparrow1$ and
$\kappa_m=h_m/p_m$. For large $m$, scalar contraction gives
$\mathfrak o_{\kappa_m}\le2\mathfrak o_\kappa$,
$\kappa_m\le2\psi$ and $P_{\kappa_m}\le2P_\kappa$.
For each $m$, the lemma, strict smooth approximation and homogeneous
cutoffs give uniform band variation bounds on $\mathcal X_d$.
The radial decreasing error majorant
\[
 |\kappa_m-\kappa|\le\eta_m(|x|)
  =(\psi(|x|)-m)_++(1-p_m)\psi(|x|)/p_m,
 \qquad\delta_m=\int\eta_m\to0
\]
gives, for $1<p<\infty$,
\begin{equation}\label{r21:heightfieldconvergence}
 \sup_{r>0}|A_r^{\kappa_m}g-A_r^\kappa g|
              \le\delta_mM_{\rm ball}|g|\to0\quad\text{in }L^p.
\end{equation}
The bound controls both envelopes, gives their variation by lower
semicontinuity and their values by $M_\kappa g\le A M_{\rm ball}|g|$.
Its almost-everywhere uniformity in all radii preserves positive-band
continuity and decay at infinity.

For $N_{a,b}(z)=\sup_{a\le r\le b}r^{-d}\psi(|z|/r)$,
monotonicity gives $u^d\psi(u)\le C_d\int_{u/2}^u\psi(v)v^{d-1}dv$,
hence, with $s=|z|>0$,
\[
 N_{a,b}(z)\le C_ds^{-d}\int_{s/(2b)}^{s/a}\psi(v)v^{d-1}dv,
\]
where the lower endpoint is zero for $b=\infty$. Tonelli yields
$\|N_{a,b}\|_1\le C_dA\log(2b/a)$ for finite $b$, and, if $J_\psi<\infty$,
\begin{equation}\label{r21:coarseenvelopelocal}
 \int_{B_L}V_t\le C_d\int_0^1\psi(v)v^{d-1}
       \log_+(L/(tv))dv
 \le C_d\{A(1+\log_+(L/t))+J_\psi\}<\infty.
\end{equation}
Radial monotonicity bounds translated-ball integrals by centered ones,
so $V_t*|Dg|\in L^1_{\rm loc}$. For finite $b$, take increasing finite
subsets of $(\mathbb Q\cap[a,b])\cup\{a,b\}$ containing the endpoints;
for coarse radii use $(\mathbb Q\cap[t,\infty))\cup\{t\}$.
Their maxima share the gradient majorant $N_{a,b}*|Dg|$ globally,
or $V_t*|Dg|$ locally, and converge locally in $L^1$ to the envelopes.
Distributional passage gives absolute continuity and the coarse gradient
bound; the preceding variation bound gives~\eqref{r21:unboundedbandendpoint}.
For $g\in L^1$, $|S_{a,b}^\kappa g|\le N_{a,b}*|g|$ gives global membership.

For localization let $G_n=S_{2^{-n},2^n}^\kappa|g|\to u$ locally
in $L^1$. Splitting its band at $t$ and using the maximum rule gives,
for open $O$,
\[
 \int_O|\nabla G_n|\le
 \int_O|\nabla S_{2^{-n},t}^\kappa|g||+
                         \int_OV_t*|D|g||.
\]
Use the signed cutoffs of Proposition~\ref{r5:BVtransfer}:
on side-$t$ cubes $Q$ meeting $O$, take $\chi_Q=1$ on $3Q$,
supported in $5Q$, and $g_Q=\chi_Q(|g|-(|g|)_{5Q})$.
Then $|Dg_Q|(\R^d)\le C_d|D|g||(5Q)$ by BV product/Poincar\'e,
and the small-radius bands agree on $Q$ up to the subtracted constant.
The band estimate and bounded overlap bound the first integral by
$C_dP_\kappa|D|g||(E_{c_dt})$ for $O\subset E_t$.
Pass $n\to\infty$ and approximate $E$ from outside; local integrability
of the coarse density makes its excess integral tend to zero whenever
the right side is finite. This proves localization.
For compact null $E$, let $t\downarrow0$; the coarse term vanishes
and finite-measure continuity on shrinking neighborhoods gives
inheritance. Inner regularity and the BV chain rule give the stated
Sobolev conclusion. For signed envelopes use $g$ in the same cutoffs
and $|S_{a,b}^\kappa g|\le A M_{\rm ball}|g|$.
\end{proof}

\begin{corollary}\label{r21:kernelbackground}
For $f\in L^1_{\rm loc}$ with finite global variation, some constant
$c$ satisfies $M_\kappa f-|c|\in\mathcal X_d$ and
$|D(M_\kappa f)|(\R^d)\le C_dP_\kappa|Df|(\R^d)$.
When $J_\psi<\infty$, localization and inheritance hold, in particular
$|D^s(M_\kappa f)|\le C_dP_\kappa|D^sf|$; if $D^sf=0$ the
maximal value is locally Sobolev with the corresponding gradient bound.
\end{corollary}
\begin{proof}
Write $f=c+h$, $h\in\mathcal X_d$, by homogeneous BV Sobolev.
Then $h_*=|c+h|-|c|\in\mathcal X_d$,
$|Dh_*|\le|Df|$, and $M_\kappa f=|c|+S_{0,\infty}^\kappa h_*$.
Apply the signed theorem, including its local estimates and chain rule.
\end{proof}

\subsection{A logarithmic density price}\label{r21:innerdensitysection}
For any $B_\rho\subset K\subset B_1$ and probability density $q$ on
$(0,1)$, put
\begin{equation}\label{r21:innerdensitykernel}
 \kappa_{K,q}(z)=|K|^{-1}\int_{p_K(z)}^1q(\lambda)\lambda^{-d}d\lambda
 \quad(p_K(z)<1),\qquad\kappa_{K,q}=0\quad(p_K(z)\ge1),
\end{equation}
\begin{equation}\label{r21:logdensityprice}
 \mathcal L(q)=\sum_{j\ge0}(j+1)2^{-j}Q_j,\qquad
 Q_j=\operatorname*{ess\,sup}_{2^{-j-1}<\lambda\le2^{-j}}q(\lambda).
\end{equation}
\begin{theorem}\label{r21:innerdensitymother}
If $\mathcal L(q)<\infty$, all conclusions of
Theorem~\ref{r21:singularkernelmother} apply, with
\begin{equation}\label{r21:innerdensityconstant}
 P_{\kappa_{K,q}}+J_\psi\le C_{d,\rho}(1+\mathcal L(q)).
\end{equation}
The same holds for jointly measurable, arbitrarily correlated mixtures
of $(K,q_K)$ with common balls and finite expected $\mathcal L(q_K)$.
The constant depends only on that expectation and the common geometry,
not on component counts, boundary regularity or a minimum scale.
\end{theorem}
\begin{proof}
Tonelli gives unit mass and the radial decreasing majorant
\[
 \begin{gathered}
 \psi(t)=|K|^{-1}\int_t^1q(\lambda)\lambda^{-d}d\lambda,
 \qquad A=|B_1|/|K|\le\rho^{-d},\\
 J_\psi\le C_{d,\rho}\int_0^1q(\lambda)
                 (1+\log(1/\lambda))d\lambda
          \le C_{d,\rho}\mathcal L(q).
 \end{gathered}
\]
For $|h|\le s$ split the oscillation center integral at $|x|=4s$.
On the near part integrate the translated kernels first, using
$\sup_{1\le r\le2}\kappa_r(z)\le\psi(|z|/2)$, to obtain
\begin{equation}\label{r21:polecancellation}
 C_{d,\rho}\int_0^1q(\lambda)
                   \min\{1,(Cs/\lambda)^d\}d\lambda.
\end{equation}
No translated-pole supremum is taken before its center integral.
On the far part $|x+th|\asymp|x|$ and $p_K$ is $\rho^{-1}$-Lipschitz.
The derivative $-|K|^{-1}q(t)t^{-d}$ is bounded on compact
subintervals of $(0,1]$, since every $Q_j<\infty$.
Uniform annular Lipschitz bounds give
\begin{equation}\label{r21:farcancellation}
 C_{d,\rho}s\int_s^C\overline q(t)\,\frac{dt}{t},\qquad
 \overline q(t)=\operatorname*{ess\,sup}_{ct\le\lambda\le C_\rho t}
                         q(\lambda)\1_{(0,1)}(\lambda).
\end{equation}
For dyadic $s=C_s2^{-j}$, each $\lambda$ contributes at most
$C_d(1+\log(1/\lambda))$ to the sum of minima in the near bound.
Tonelli and $\sum_{s\le t}s\le Ct$ sum the far bound to
$C_{d,\rho}\int_0^C\overline q(t)dt
\le C_{d,\rho}\sum_j2^{-j}Q_j$, since fixed-ratio annuli meet
boundedly many dyadic bands. Larger $s$ cost the majorant mass.
For correlated mixtures, Tonelli averages the majorant mass and
logarithmic moment, and $\mathfrak o_{\int\kappa}\le\int\mathfrak o_\kappa$
averages the price. Apply Theorem~\ref{r21:singularkernelmother}.
\end{proof}

\begin{corollary}\label{r21:allpositivepowers}\label{r11:powerinnerzero}
Every law $q_a(\lambda)=a\lambda^{a-1}$, $a>0$, has the preceding
variation/Sobolev endpoints and singular inheritance, uniformly in
$q_{a,\epsilon}=q_a\1_{[\epsilon,1]}/(1-\epsilon^a)$,
$0<\epsilon\le1/2$. For $p=p_K(z)<1$,
\[
 \kappa_{K,a}(z)=
 \begin{cases}
 a(1-p^{a-d})/(|K|(a-d)),&a\ne d,\\
 d\log(1/p)/|K|,&a=d.
 \end{cases}
\]
For $a>d$ the kernel is bounded and globally H\"older, and also has the
exact singular identity and global strong gradient continuity. These
extra claims are not made for the unbounded regime $a\le d$.
\end{corollary}
\begin{proof}
$\mathcal L(q_a)=a\max(1,2^{1-a})/(1-2^{-a})^2$;
truncation increases it by at most $(1-2^{-a})^{-1}$.
The near/far estimates also give, uniformly in $\epsilon$,
\[
 \mathfrak o_{\kappa_{K,a,\epsilon}}(s)\le C_{d,\rho,a}
 \begin{cases}s^a,&a<1,\\s\log(e/s),&a=1,\\s,&a>1.
 \end{cases}
\]
For $a\ge d$ the central contribution is also $O(s)$.
For $a>d$, set $\beta_0=a-d$.
The zero-boundary-trace derivative
$\nabla\kappa_{K,a}=-(a/|K|)p_K^{\beta_0-1}\nabla p_K$ is
integrable and has cusp exponent $\max(0,1-\beta_0)<1$.
Use the bounded-kernel transfer and Corollary~\ref{r11:cusptail}.
\end{proof}
\begin{corollary}\label{r21:logheavyinnerlaw}
For $p>2$, the probability density
$q_p(\lambda)=(p-1)/(\lambda(1+\log(1/\lambda))^p)$ has the same
variation/Sobolev endpoints and singular inheritance. Its kernel is
unbounded and every negative power moment is infinite. No failure is
claimed for $1<p\le2$, where this sufficient price diverges.
\end{corollary}
\begin{proof}
Substitute $s=\log(1/\lambda)$ for mass and moments.
$2^{-j}Q_j\le2(p-1)(1+j\log2)^{-p}$ gives $\mathcal L(q_p)<\infty$;
the kernel grows like
$p_K(z)^{-d}(1+\log(1/p_K(z)))^{-p}$ at zero.
\end{proof}
In logarithmic coordinates $b(s)=e^{-s}q(e^{-s})$, the sufficient
price is equivalent, up to fixed numerical constants, to
$\sum_{j\ge0}(j+1)\|b\|_{L^\infty([j,j+1])}$.
It is finite, for example, if $b\in BV_{\rm loc}$ and
$\int_0^\infty(1+s)b(s)ds+\int_{[0,\infty)}(1+s)d|Db|(s)<\infty$,
by $\sup_I b\le\int_Ib+|Db|(I)$ on unit intervals.

\begin{corollary}\label{r21:singularinnersmoothing}
For any joint probability law of $B_\rho\subset K\subset B_1$ and
$0<\lambda\le1$ with $m_\mu=\int\log(1/\lambda)d\mu<\infty$,
independently multiply the scale by $e^{-U}$, with $U$ uniform on
$[0,\delta]$, $0<\delta\le1$.
The smoothed mixture has all singular-kernel conclusions with constant
\begin{equation}\label{r21:innersmoothingprice}
 C_{d,\rho}\{1+\log(1/\delta)+m_\mu\}.
\end{equation}
The law may be atomic, singular or correlated. Variable measurable
widths $\delta(K,\lambda)\in(0,1]$ are allowed if both logarithms
are integrable; no negative scale moment is needed.
\end{corollary}
\begin{proof}
The base law has density
$\1_{[e^{-\delta},1]}(t)/(\delta t)$.
Its kernel $k_{K,\delta}$ has height at most
$(e^{d\delta}-1)/(d\delta|K|)\le C_{d,\rho}$ and global Lipschitz
constant $C_{d,\rho}/\delta$. Hence
$\mathfrak o_k(s)\le C_{d,\rho}\min(1,s/\delta)$ and its dyadic
sum costs $C_{d,\rho}(1+\log(1/\delta))$.
Its radial majorant has mass and central logarithmic moment bounded
by $C_{d,\rho}$. For $k_\lambda=\lambda^{-d}k(\cdot/\lambda)$,
$\mathfrak o_{k_\lambda}(s)=\mathfrak o_k(s/\lambda)$.
The first $\lceil\log_2(1/\lambda)\rceil$ terms cost the uniform
band-envelope mass; the rest cost the base sum. Dilation preserves
majorant mass and adds $A\log(1/\lambda)$ to its logarithmic moment.
Average and apply Theorem~\ref{r21:singularkernelmother}; variable widths
add the expected $\log(1/\delta(K,\lambda))$ term.
\end{proof}
The singular-kernel conclusions preserve total variation, singular
inheritance and $SBV$, but do not eliminate general Cantor derivatives
or give exact singular identities or gradient joint continuity by
truncation. Constants may grow as scale laws concentrate; outer scales
remain bounded after normalization.

\section{Contact criteria for the homogeneous comparison}
\label{subsec:infinitecontact}
We quantify exact homogeneous Euler comparison, whose restrictions
are not necessary for the maximal endpoint. Theorem~\ref{r3:bodyendpoint}
gives the raw H\"older-normal endpoint without a gap assumption;
Section~\ref{sec:R8lognormal} treats logarithmic normals and compact profiles.

\begin{corollary}\label{thm:contactdiniendpoint}\label{thm:mildinfiniteendpoint}
Let $K$ have a $C^{1,\vartheta}$ gauge and a positive nondecreasing
boundary support-gap modulus $\eta$, meaning
$1-\nabla p(v)\cdot u\ge\eta(|u-v|)$ for $0<|u-v|\le r_0$. If
\begin{equation}\label{eq:contactdini-integral}
 \int_0^{r_0}r^{\vartheta/2-1}
             \sqrt{\log\frac e{\eta(r)}}\,dr<\infty,
\end{equation}
then a summable price and shell mesh verify both local inputs and their
signed-band and regularity conclusions, with constants controlled by
the gauge and contact data.
\end{corollary}

\begin{proof}[Proof of Corollary~\ref{thm:contactdiniendpoint}]
By Lemma~\ref{lem:R6gaugedictionaries}(i), replace the boundary
modulus by $\eta_B(r)=cr^2\eta(c'r)$; its logarithm changes by the
integrable term $O(\log(e/r))$. Set
\[
 \rho_j=\rho_*2^{-j},\quad a_j=\rho_j^\vartheta,\quad
 L_j=\log\frac C{\eta_B(c\rho_j)},\quad D_j=2L_j+j^2+C_0.
\]
The next finer dyadic interval in~\eqref{eq:contactdini-integral} gives
\begin{equation}\label{eq:contactdini-dyadic}
 \sum_j\sqrt{a_jD_{j+1}}<\infty.
\end{equation}
For large fixed $K_0$, choose $\alpha_0=\vartheta/2$ and
\[
 \delta_j=e^{-D_j},\quad h_j=e^{-K_0D_{j+1}},\quad
 v_j=\sqrt{a_j/D_{j+1}},\qquad
 w(s)=s^{\alpha_0}+\sum_kv_k\1_{\{s\ge h_k\}}.
\]
Since $w(h_j)\ge v_j$ and $v_j$ decreases, splitting $b_j$ at
$h_{j-1}$ gives
\[
 H_j\le v_j^{-1},\quad
 b_j\le\max\{v_{j-1}^{-1},e^{-[K_0(1-\alpha_0)-1]D_j}\},
 \quad C_j\le C/v_j.
\]
Thus Tonelli bounds the price and cap budgets together:
\[
 \sum_nw(2^{-n})+\sum_ja_jC_j
 \le C+CK_0\sum_j\sqrt{a_jD_{j+1}}<\infty.
\]
The power floor gives
$f_j\le e^{-[K_0(1-\alpha_0)-1]D_{j+1}}
+e^{-K_0(\vartheta-\alpha_0)D_{j+1}}\le Ce^{-cD_{j+1}}$.
Theorem~\ref{lem:modulusconcentration}(i) gives admissible widths and
$G_j\le Ce^{-j^2}$ after increasing $C_0$.
For $u_j=2^{-j}$, \eqref{eq:contactdini-dyadic} gives
$D_{j+1},C_j\le C/a_j$ and $\log(e+b_{j+1})=O(j+1)$, so the anomaly
summands are at most $Ce^{-j^2+C(j+1)}$.
Apply Corollary~\ref{thm:meshpricechain}, including its forced-endpoint
tests, then the master and transfer theorems.
\end{proof}

\begin{corollary}\label{cor:nearcritical-contact}
The preceding quantitative verification applies to finite-power gaps,
to $c\exp(-Cr^{-a})$ for $0<a<\vartheta$, and to
\[
 \eta(r)=c\exp\left\{-\frac{Cr^{-\vartheta}}
                       {[\log(e/r)]^\beta}\right\},\qquad\beta>2.
\]
\end{corollary}
\begin{proof}
The three integrands are bounded respectively by a logarithmic factor
times $r^{\vartheta/2-1}$, by $Cr^{(\vartheta-a)/2-1}$, and by
$Cr^{\vartheta/2-1}+C/[r\log(e/r)^{\beta/2}]$.
\end{proof}

\begin{example}\label{ex:explicitinfinitebody}
For every $d\ge2$ and $\beta>2$ there is a bounded strictly convex
$C^\infty$ body satisfying Corollary~\ref{thm:contactdiniendpoint}
with $\vartheta=1$, but no gap $c\exp(-Cr^{-a})$, $a<1$.
\end{example}
\begin{proof}
Set $J(t)=1/[t\log(e/t)^\beta]$ near zero. Since $e^{-J(t)}$
is increasing and flat there, choose an even smooth $b$, positive off
zero, equal to $e^{-J(|t|)}$ near zero and to $2$ outside a small interval.
Put
\[
 \psi(t)=\int_0^{|t|}(|t|-s)b(s)\,ds,\qquad
 F(x',t)=|x'|^2+\psi(t).
\]
Then $K=\{F<1\}$ is bounded, smooth and strictly convex.
An interval of length $h$ has at least $h/4$ of its middle half outside
$(-h/8,h/8)$. There the Bregman weight is at least $h/4$ and
$b\ge ce^{-CJ(h)}$, since $J(h/8)\le8J(h)$. Hence
$D_\psi(s,t)\ge ch^2e^{-CJ(h)}$ on bounded intervals.
Together with the quadratic coordinates and
Lemma~\ref{lem:R6gaugedictionaries}(ii), this gives
$1-\nabla p(v)\cdot u\ge cr^2e^{-CJ(r)}\ge c'e^{-C'J(r)}$.
The near-critical corollary applies. Conversely, with
$v=(e_1,0)$ and $u_t=(\sqrt{1-\psi(t)}e_1,t)$,
\[
 1-\nabla p(v)\cdot u_t\le Ct^2e^{-J(t)},\qquad
 |u_t-v|\asymp t,\qquad J(t)/t^{-a}\longrightarrow\infty\ (a<1).
\]
This rules out all the asserted lower gaps, including finite powers.
\end{proof}

\subsection{The exact scalar budget of this method}
\begin{proposition}\label{prop:contactbudgetnecessity}
Let $0<\vartheta\le1$, $r_*>0$, and let $L:(0,r_*]\to[0,\infty)$ be
nonincreasing. Let $w:(0,1]\to(0,\infty)$ be nondecreasing with
$\sum_nw(2^{-n})<\infty$, let $0<\delta_j<1$ decrease to zero, and assume
$0<h_j\le\delta_{j+1}$, $0<\rho_j\le r_*$ and
\begin{equation}\label{eq:necessarywidthcap}
 L(\rho_j)\le\log(1/\delta_j)+K,\qquad
 \mathcal B:=\sum_j\frac{\rho_j^\vartheta}{w(h_j)}<\infty
\end{equation}
for fixed $K\ge0$. No monotonicity of $\rho_j$ is needed. Writing
$T_j=\log(1/\delta_j)$ and
$\mathcal P=\int_0^\infty w(e^{-t})\,dt$, one has
\begin{equation}\label{eq:necessarycontactintegral}
 \int_0^{r_*}r^{\vartheta/2-1}\sqrt{L(r)}\,dr
 \le\frac2\vartheta\{r_*^{\vartheta/2}\sqrt{T_0+K}
                      +\sqrt{\mathcal P\mathcal B}\}<\infty.
\end{equation}
\end{proposition}
\begin{proof}
Put $W(t)=w(e^{-t})$, $d_j=T_{j+1}-T_j$ and $V_j=W(T_{j+1})$.
Monotonicity gives $\mathcal P\le(\log2)\sum_nw(2^{-n})$,
$\sum_jd_jV_j\le\mathcal P$ and
$\sum_j\rho_j^\vartheta/V_j\le\mathcal B$. Cauchy--Schwarz yields
\begin{equation}\label{eq:necessarydiscretemixed}
 \sum_j\rho_j^{\vartheta/2}\sqrt{d_j}
 \le\sqrt{\mathcal P\mathcal B}.
\end{equation}
If $T_j+K<L(r)$, the width bound and monotonicity of $L$ imply
$r<\rho_j$. Since $T_j\to\infty$, summing increments up to the first
$T_j+K\ge L(r)$ gives
\[
 L(r)\le T_0+K+\sum_jd_j\1_{\{r<\rho_j\}}.
\]
Square-root subadditivity and Tonelli bound the integral by
\[
 \frac2\vartheta\left\{r_*^{\vartheta/2}\sqrt{T_0+K}
 +\sum_j\rho_j^{\vartheta/2}\sqrt{d_j}\right\}.
\]
Apply~\eqref{eq:necessarydiscretemixed}.
\end{proof}

\begin{remark}\label{rem:pooledpricesexact}
Pooling preserves the criterion. Fix $K_0\ge1$ and nonnegative
summable weights $v_j$. Merge equal thresholds by adding their weights
and reindex. Set
\[
 h_j=e^{-K_0T_{j+1}},\qquad V_j=\sum_{k\ge j}v_k,\qquad
 w_{\rm step}(s)=\sum_jv_j\1_{\{s\ge h_j\}},
\]
with $V_j>0$. Then $h_j\le\delta_{j+1}$ and $w_{\rm step}(h_j)=V_j$.
Expanding $T_{j+1}=T_1+\sum_{i=1}^j(T_{i+1}-T_i)$, Tonelli gives
\[
 \int_0^\infty w_{\rm step}(e^{-t})dt
 =K_0\sum_jT_{j+1}v_j
 =K_0\{T_1V_0+\sum_{j\ge1}(T_{j+1}-T_j)V_j\}.
\]
For nonnegative costs $a_j$, in particular $a_j=\rho_j^\vartheta$
on a strict mesh, put $d_0=T_1$ and
$d_j=T_{j+1}-T_j$ ($j\ge1$), Cauchy--Schwarz bounds
$\sum_j\sqrt{a_jd_j}$ by
$(\sum_ja_j/V_j)^{1/2}(\sum_jd_jV_j)^{1/2}$.
Thus arbitrary pointwise price tails need not be realizable.
\end{remark}

\begin{corollary}\label{cor:holderbudgetoptimality}
For a fixed canonical Bregman modulus $\eta$ and the sufficient width
$\delta\le c\eta(c'\rho)$ with H\"older cap cost $C\rho^\vartheta$,
\eqref{eq:contactdini-integral} is the exact finiteness criterion for
the mesh-and-price framework. Neither pooled prices nor irregular
meshes remove $\beta>2$ for
$\eta(r)=\exp\{-Cr^{-\vartheta}/\log(e/r)^\beta\}$.
This is not a necessary condition for the maximal endpoint.
\end{corollary}
\begin{proof}
Sufficiency is Corollary~\ref{thm:contactdiniendpoint}. Necessity follows
from Proposition~\ref{prop:contactbudgetnecessity} with
$L(r)=\log(e/\eta(cr))$, since the cap bound contains
$\sum_j\rho_j^\vartheta C_j$ and $C_j\ge1/w(h_j)$.
The model gives $\int_0dr/[r\log(e/r)^{\beta/2}]$; a factor $r^2$ in
$\eta$ adds an integrable square-root logarithm.
\end{proof}

\begin{remark}
For $\eta(r)=e^{-Cr^{-a}}$, $a>0$, put $T=\log(1/\delta)$.
The inverse-modulus coefficient is $\exp(c\rho^{-a})/(\rho T^{1/a})$,
up to fixed exponent constants; its vanishing forces
$T\ge\exp(c'\rho^{-a})$. Every nondecreasing summable price has
$Tw(e^{-T})\to0$, so $b_j/T_j\to\infty$ for
$T_j=\log(1/\delta_j)$. Hence $\rho_j^\vartheta b_j\not\to0$.
This excludes simultaneous cap and far-target budgets for this method,
not the compact-profile endpoint.
\end{remark}

\section*{Use of artificial intelligence}
The author acknowledges the use of AI tools. The author assumes full
responsibility for the mathematical validity, accuracy, and integrity of
the proofs presented in the manuscript.

\end{document}